\documentclass{article}
\usepackage[utf8]{inputenc}

\usepackage{amsmath,amssymb}
\usepackage{amsthm}
\usepackage{bm}
\usepackage{graphicx}
\usepackage{ascmac}
\newtheorem{definition}{Definition}[section]
  \newtheorem{theorem}[definition]{Theorem}
   \newtheorem{lemma}[definition]{Lemma}
  
  \newtheorem{proposition}[definition]{Proposition}
  \newtheorem{assertion}[definition]{Assertion}
  \newtheorem{claim}[definition]{Claim}%[section]
  \newtheorem{case}{Case}
  \newtheorem{subcase}{Case 1.\hspace{-0.4em}} 
  \newtheorem{subcase2}{Case 2.\hspace{-0.4em}}
  \newtheorem{subccase}{Case 1.\hspace{-0.4em}} 
  \newtheorem{subccase2}{Case 2.\hspace{-0.4em}}
  \newtheorem{subccase3}{Case 3.\hspace{-0.4em}}
   
  \newtheorem{subcccase2}{Case 2.\hspace{-0.4em}}

\newtheorem{remark}[definition]{Remark}

\title{On keen bridge splittings of links}

\author{Ayako Ido, Yeonhee Jang and Tsuyoshi Kobayashi}
\date{\empty}

\begin{document}

\maketitle

\begin{abstract}
In this paper, we extend the concept of {\it (strongly) keenness} for Heegaard splittings to bridge splittings, and show that, for any integers $g$, $b$ and $n$ with $g\ge 0$, $b\ge 1$, $n\ge 1$ except for $(g,b)=(0,1)$ %, $(0,2)$ 
and $(g,b,n)=(0,3,1)$, there exists a strongly keen $(g,b)$-splitting of a link with distance $n$. 
We also show that any $(0,3)$-splitting of a link with distance $1$ cannot be keen.
\end{abstract}

\tableofcontents

\part{Introduction and background materials}\label{part1}

\section{Introduction}

Hempel \cite{He} introduced the notion of {\it distance} for Heegaard splittings by using the curve complexes of the Heegaard surfaces, as a measure of the complexity of Heegaard splittings.
There have been many works concerning the Hempel distance, some of which asserts that the distance of a Heegaard splitting is closely related with the topology or the geometric structure of the ambient manifold.

The notion of the Hempel distance can be extended to the distance for bridge splittings of links in $3$-manifolds as follows: 
For a bridge splitting $(V_1,t_1)\cup_{(F,P)}(V_2,t_2)$ of a link in a closed orientable $3$-manifold, the {\it distance} of the splitting is defined to be $d_{F\setminus P}(\mathcal{D}^0(V_1\setminus t_1), \mathcal{D}^0(V_2\setminus t_2))={\rm min}\{d_{F\setminus P}(x,y)\mid x\in \mathcal{D}^0(V_1\setminus t_1), y\in\mathcal{D}^0(V_2\setminus t_2)\}$, where $d_{F\setminus P}$ is the simplicial distance in $\mathcal{C}(F\setminus P)$ and $\mathcal{D}(V_i\setminus t_i)$ is the subcomplex of $\mathcal{C}(F\setminus P)$ spanned by the vertices with representatives bounding disks in $V_i\setminus t_i$ for $i=1,2$. 
(See Section \ref{sec-pre} for details.) 

In \cite{IJK3}, the authors introduced the concept of {\it keen} and {\it strongly keen} Heegaard splittings, and showed the existence of strongly keen Heegaard splittings, that is, Heegaard splittings \lq\lq with unique geodesics\rq\rq\ realizing the Hempel distance.
It is shown in \cite{IK} that keenness and strongly keenness of Heegaard splittings imply some finiteness properties of the Goeritz groups.

The purpose of this paper is to extend the concept of the {\it keenness} to the bridge splittings of links, and to show the existence of strongly keen bridge splittings.
We say that a bridge splitting $(V_1,t_1)\cup_{(F,P)}(V_2,t_2)$ is {\it keen} if its distance is realized by a unique pair of elements of $\mathcal{D}^0(V_1\setminus t_1)$ and $\mathcal{D}^0(V_2\setminus t_2)$, that is, 
\begin{center}
If $d_{F\setminus P}(a,b)=d_{F\setminus P}(a',b')=d_{F\setminus P}(\mathcal{D}^0(V_1\setminus t_1), \mathcal{D}^0(V_2\setminus t_2))$ for $a,a'\in \mathcal{D}^0(V_1\setminus t_1)$ and $b,b'\in\mathcal{D}^0(V_2\setminus t_2)$, then $a=a'$ and $b=b'$.
\end{center}
For a keen bridge splitting $(V_1,t_1)\cup_{(F,P)}(V_2,t_2)$, the geodesics joining the unique pair of the elements of $\mathcal{D}^0(V_1\setminus t_1)$ and $\mathcal{D}^0(V_2\setminus t_2)$ may not be unique (see Remark \ref{rem}). 
We say that $(V_1,t_1)\cup_{(F,P)}(V_2,t_2)$ is {\it strongly keen} 
if the set of the geodesics joining the pair of the elements of $\mathcal{D}^0(V_1\setminus t_1)$ and $\mathcal{D}^0(V_2\setminus t_2)$ realizing the distance consists of one element.

\begin{theorem}\label{thm-1}
For any integers $g$, $b$ and $n$ with $g\ge 0$, $b\ge 1$, $n\ge 1$ except for $(g,b)=(0,1)$, $(0,2)$ and $(g,b,n)=(0,3,1)$, there exists a strongly keen $(g,b)$-splitting of a link with distance $n$. 
\end{theorem}

%We remark that the number of the component $c$ of the link in the above theorem can be any positive integer less than or equal to $b$, due to the results in \cite{Kra} or in \cite[Example 1]{IM}.
In the proof of Theorem~\ref{thm-1}, we see that any $(1,1)$-splitting of distance $1$ must be strongly keen, and we note the proof of the fact is due to Saito \cite{Sai} (see Proposition~\ref{prop-111}).

\begin{remark}
{\rm
Part \ref{part2} of this paper, which consists of 4 sections, is devoted to the proof of Theorem~\ref{thm-1} for the case $n\ge 2$.
The sectioning looks like a case-by-case analysis, that is, they are exhaustive and mutually exclusive.
But a careful reader will realize that the essences of the arguments are not mutually exclusive.
For example, the case when $(g,b,n)=(2,2,3)$ is treated in Section~\ref{sec-proof-1} formally, and we would like to note that the arguments in Case 1 of Section~\ref{sec-proof-4} also work for this case.
}
\end{remark}

We note that the case when $(g,b)=(0,1)$ is absurd, that is, $\mathcal{D}^0(V_i\setminus t_i)$ is empty.
For the case when $(g,b)=(0,2)$, we have the following (see Section \ref{app-2}).

\begin{theorem}\label{thm-3}
There exist strongly keen $(0,2)$-splittings of links with distance $n$ for any given positive integer $n$.
\end{theorem}

In fact, the $(0,2)$-splitting of any $2$-bridge link corresponding to the continued fraction $[a_1,a_2,\dots,a_{n-1}]$ with $a_i\ge 3$ for every $i\in\{1,2,\dots,n\}$ is a strongly keen bridge splitting of distance $n$. 
We note that any $(0,2)$-splitting of a link is keen since the $2$-string trivial tangle admits a unique essential disk.

On contrast, we show that any $(0,3)$-splitting of a link with distance $1$ cannot be keen.
In fact, we prove: %as follows (see Section \ref{app-3}).

\begin{theorem}\label{thm-2}
A link $L$ admits a $(0,3)$-splitting with distance $1$ if and only if 
$L$ is either {\rm (i)} a trivial knot, {\rm (ii)} a $2$-bridge link which is not a $2$-component trivial link, or {\rm (iii)} the connected sum of two $2$-bridge links neither of which is a $2$-component trivial link.
Further, any of such $(0,3)$-splitting admits at least two distinct pairs of essential disks realizing the distance $1$, and hence, it is not keen.
\end{theorem}

\begin{remark}
{\rm 
%(1) There exist $(0,3)$-splittings of links with distance $1$. In fact, the (non-splittable) connected sum of two 2-bridge links admits such a bridge splitting.
%
%(2) Any $(0,2)$-splitting of a link is keen since the $2$-string trivial tangle admits a unique essential disk. 
%There exist strongly keen $(0,2)$-splittings of links with distance $n$ for any given positive integer $n$. In fact, the $(0,2)$-splitting of any $2$-bridge link corresponding to the continued fraction $[a_1,a_2,\dots,a_n]$ with $a_i\ge 3$ for every $i\in\{1,2,\dots,n\}$ is a strongly keen bridge splitting of distance $n$. 
%See Appendix for detail.

%(3) 
Note that any keen bridge splittings of distance $1$ is strongly keen by the definition.
For $n\ge 4$, by modifying the construction of strongly keen bridge splittings in this paper slightly, it can be seen that there exist bridge splittings of distance $n$ which are keen but not strongly keen. 
See Remark~\ref{rem} for example.
}
\end{remark}

\begin{remark}
{\rm
We had given a proof of the existence of bridge splittings of links with distance $n$ in \cite{IJK2}, but found out there is a gap in the proof.
More precisely, \lq\lq $P_i(a)\ne \emptyset$ for any $a\in\mathcal{C}^0(F_i)$\rq\rq\ in Line 13 of Page 613 of the paper does not necessarily hold, and hence
the inequality \lq\lq ${\rm diam}_{\partial_- W_i}(P_i(A)) \le {\rm diam}_{F_i}(A)$\rq\rq, which is used in the last line of Page 613 and in Line 20 of Page 614, may not be correct.
We are not able to fix the gap at the moment, but Theorem~\ref{thm-1} above covers the result in \cite{IJK2}.
}
\end{remark}

\section{Preliminaries}\label{sec-pre}

Throughout this paper, for a submanifold $Y$ of a manifold $X$, $N_X(Y)$ denotes a regular neighborhood of $Y$ in $X$.
When $X$ is clear from the context, we denote $N_X(Y)$ by $N(Y)$ in brief.
We denote ${\rm cl}_X(Y)$ (or ${\rm cl}(Y)$ in brief) the closure of $Y$ in $X$.

\subsection{Curve complexes}\label{sec-cc}

Let $S$ be a genus-$g$ orientable surface with $e$ boundary components and $p$ punctures. 
A simple closed curve in $S$ is {\it essential} if it does not bound a disk or a once-punctured disk in $S$ and is not parallel to a component of $\partial{S}$. 
We say that $S$ is {\it non-simple} if there exists an essential simple closed curve in $S$, and $S$ is {\it simple} otherwise.
By an {\it arc properly embedded} in $S$, we mean an arc intersecting $\partial S$ only in its endpoints.
An arc properly embedded in $S$ is {\it essential} if it does not co-bound a disk with no puncture in $S$ together with an arc on $\partial{S}$. 
Two simple closed curves or two arcs in $S$ are {\it isotopic} if there is an ambient isotopy of $S$ which sends one to the other.
We say that $S$ is {\it sporadic} if either \lq\lq $g=0$ and $e+p\leq 4$\rq\rq\ or \lq\lq $g=1$ and $e+p\leq1$\rq\rq.

For a non-sporadic surface $S$, the {\it curve complex} $\mathcal{C}(S)$ is defined as follows: Each vertex of $\mathcal{C}(S)$ is the isotopy class of an essential simple closed curve in $S$, and a collection of $k+1$ vertices forms a $k$-simplex of $\mathcal{C}(S)$ if they can be realized by disjoint curves in $S$. 
For sporadic surfaces, we need to modify the definition of the curve complex slightly. 
We assume that either $g=1$ and $e+p\leq 1$ or $g=0$ and $e+p=4$ since, otherwise, $S$ is simple.
When $g=1$ and $e+p\leq 1$ (resp. $g=0$ and $e+p=4$), a collection of $k+1$ vertices forms a $k$-simplex of $\mathcal{C}(S)$ if they can be realized by curves in $S$ which mutually intersect transversely exactly once (resp. twice). 
The {\it arc-and-curve complex} $\mathcal{AC}(S)$ is defined similarly: Each vertex of $\mathcal{AC}(S)$ is the isotopy class of an essential properly embedded arc or an essential simple closed curve in $S$, and a collection of $k+1$ vertices forms a $k$-simplex of $\mathcal{AC}(S)$ if they can be realized by disjoint arcs or simple closed curves in $S$. 
The symbols $\mathcal{C}^0(S)$ and $\mathcal{AC}^0(S)$ denote the 0-skeletons of the curve complexes $\mathcal{C}(S)$ and $\mathcal{AC}(S)$, respectively.
Throughout this paper, for a vertex $x\in\mathcal{C}^0(S)$ or $x\in\mathcal{AC}^0(S)$ we often abuse notation and use $x$ to represent (the isotopy class of) a geometric representative of $x$.

We can define the {\it distance} between two vertices in the curve complex $\mathcal{C}(S)$ to be the minimal number of 1-simplices of a simplicial path in $\mathcal{C}(S)$ joining the two vertices.
We denote by $d_{S}(a, b)$ the distance in $\mathcal{C}(S)$ between the vertices $a$ and $b$. 
For subsets $A$ and $B$ of the vertices of $\mathcal{C}(S)$, we define ${\rm diam}_{S}(A, B)={\rm diam}_{S}(A\cup B)$. 
Similarly, we can define the distance $d_{\mathcal{AC}(S)}(a, b)$ and ${\rm diam}_{\mathcal{AC}(S)}(A, B)$.
Let $a_0,a_1,\dots, a_n$ be a sequence of vertices in $\mathcal{C}(S)$ such that $a_{i-1}\cap a_i=\emptyset$ $(i=1,2,\dots,n)$.
Then $[a_0,a_1,\dots,a_n]$ denotes the path in $\mathcal{C}(S)$ with vertices $a_0,a_1,\dots,a_n$ in this order.
We call a path $[a_0,a_1,\dots,a_n]$ a {\it geodesic} if $n=d_S(a_0,a_n)$.

\subsection{Subsurface projections}\label{sec-subsurface}

Throughout this paper, $\mathcal{P}(Y)$ denotes the power set of a set $Y$.
Let $S$ be a genus-$g$ orientable surface with $e$ boundary components and $p$ punctures.
We say that a subsurface $X(\subset S)$ is {\it essential} if each component of $\partial X$ is an essential simple closed curve in $S$.
Suppose that $X$ is a non-simple essential subsurface of $S$.
We call the composition $\pi_X:=\pi_0\circ\pi_{AC}$ of maps $\pi_{AC}:\mathcal{C}^{0}(S)\rightarrow \mathcal{P}(\mathcal{AC}^{0}(X))$ and $\pi_0:\mathcal{P}(\mathcal{AC}^{0}(X))\rightarrow\mathcal{P}(\mathcal{C}^{0}(X))$ a {\it subsurface projection}, where $\pi_{AC}$ and $\pi_0$ are defined as follows: 
For a vertex $\alpha$, take a representative $\alpha$ 
so that $|\alpha\cap X|$ is minimal, where $|\cdot|$ is the number of connected components. Then 

\begin{itemize}
\item $\pi_{AC}(\alpha)$ is the set of all isotopy classes of the components of $\alpha\cap X$,
\item $\pi_0(\{\alpha_1,\dots,\alpha_n\})$ is the union for all $i=1,\dots,n$ of the set of all isotopy classes of the components of $\partial N_X(\alpha_{i}\cup\partial X)$ which are essential in $X$.%, where $N(\alpha_{i}\cup\partial X)$ is a regular neighborhood of $\alpha_i\cup\partial X$ in $X$.
\end{itemize}

%We say that $\alpha$ {\it cuts} $X$ (resp. $\alpha$ {\it misses} $X$) if $\alpha\cap X\ne\emptyset$ (resp. $\alpha\cap X=\emptyset$).
We say that $\alpha$ {\it misses} $X$ if $\alpha$ can be isotoped on $S$ so that $\alpha\cap X=\emptyset$. 
Otherwise, we say that $\alpha$ {\it cuts} $X$.

The next lemma is due to \cite[Lemma 2.2]{MM2}.
\begin{lemma}\label{lem_mm}
%Let $X$ be a non-simple essential subsurface of $S$ as above. 
Let $X$ be a non-simple surface.
If $d_{\mathcal{AC}(X)}(\alpha,\beta)\le 1$, then ${\rm diam}_{X}(\pi_0(\{\alpha\}),\pi_0(\{\beta\}))\le 2$.
\end{lemma}

The next lemma can be easily proved by using the above lemma.
%The next lemma can be easily proved by using \cite[Lemma 2.2]{MM2}.
%
\begin{lemma}[{\rm cf. \cite[Lemma 2.1]{IJK1}}]\label{subsurface distance}
Let $X$ be a non-simple essential subsurface of $S$ as above. 
Let $[\alpha_{0}, \alpha_{1}, . . . , \alpha_{n}]$ be a path in $\mathcal{C}(S)$ such that every $\alpha_i$ cuts $X$. 
Then ${\rm diam}_{X}(\pi_{X}(\alpha_{0}), \pi_{X}(\alpha_{n}))\leq 2n$.

Furthermore, we have ${\rm diam}_X(\pi_X(\alpha))\le 2$ for any $\alpha\in\mathcal{C}^0(S)$ which cuts $X$.
\end{lemma}

%\subsection{Maps induced on curve complexes}\label{sec-maps}

Let $Y, Z$ be non-simple surfaces. Suppose that there exists an embedding $\varphi:Y\rightarrow Z$ such that $\varphi(Y)$ is an essential subsurface of $Z$. %for each component $l$ of $\partial Y$ either $\varphi(l)\subset \partial Z$ or $\varphi(l)$ is essential in $Z$.
Note that $\varphi$ naturally induces maps $\mathcal{C}^0(Y)\rightarrow \mathcal{C}^0(Z)$ and $\mathcal{P}(\mathcal{C}^0(Y))\rightarrow \mathcal{P}(\mathcal{C}^0(Z))$. Throughout this paper, under this setting, we abuse notation and use $\varphi$ to denote these maps.

The next lemma can be proved by using arguments in the proof of \cite[Lemma 2.3]{IJK3} and Appendix~\ref{app-b}. % together with \cite[Claim 2]{IS}.

%\begin{lemma}\label{lem-2-3-0}{\rm (\cite[Proposition 4.6]{MM1})}
%Let $X$ be a non-sporadic surface.
%There exists a constant $c>0$ such that, for any pseudo-Anosov homeomorphism $h:X\rightarrow X$, $\gamma\in\mathcal{C}(X)$ and $n\in\mathbb{Z}$, $d_{\mathcal{C}(X)}(\gamma,h^n(\gamma))\ge c|n|$.
%\end{lemma}

\begin{lemma}\label{lem-2-3-1}
Let $S$ be a non-simple surface, and let $X$ be a non-simple essential subsurface of $S$. % such that each component of $\partial X$ is essential in $S$. 
Let $\alpha,\beta\in\mathcal{C}^0(S)$ such that $\alpha, \beta$ cut $X$.
For any $k\in\mathbb{N}$, there exists a homeomorphism $h:S\rightarrow S$ such that $h|_{S\setminus X}={\rm id}_{S\setminus X}$ and that
$d_X(\pi_X(\alpha), \pi_X(h(\beta)))>k$.
In particular, ${\rm diam}_X(\pi_X(\alpha), \pi_X(h(\beta)))>k$ also holds.
\end{lemma}

%The following lemma can be proved by using arguments similar to those in the proof of \cite[Proposition 4.1]{IJK1}.

%\begin{lemma}\label{lem-2-3-2}
%Let $S$ be a non-sporadic surface, and let $[\alpha_0,\alpha_1,\dots,\alpha_n]$ and $[\beta_0,\beta_1,\dots,\beta_m]$ be geodesics in $\mathcal{C}(S)$. Suppose one of the following holds.
%\begin{itemize}
%\item each of $\alpha_n$ and $\beta_0$ are non-separating on $S$. In this case, let $X$ be the closure of $S\setminus N(\alpha_n)$.
%\item each of $\alpha_n$ and $\beta_0$ cuts off a twice-punctured disk $D$ from $S$. In this case, let $X$ be the closure of $S\setminus D$.
%\end{itemize}
%Let $h:S\rightarrow S$ be a homeomorphism such that 
%\begin{itemize}
%\item $h(\beta_0)=\alpha_n$, and
%\item ${\rm diam}_X(\pi_X(\alpha_0)\cup \pi_X(h(\beta_m)))>2(n+m)$.
%\end{itemize}
%Then $[\alpha_0,\alpha_1,\dots,\alpha_n(=h(\beta_0)),h(\beta_1),\dots,h(\beta_m)]$ is a geodesic in $\mathcal{C}(S)$.

%Moreover, every geodesic connecting $\alpha_0$ and $h(\beta_m)$ passes through $\alpha_n$. 
%In fact, for any geodesic $[\gamma_0,\gamma_1,\dots,\gamma_{n+m}]$ in $\mathcal{C}(S)$ such that $\gamma_0=\alpha_0$ and $\gamma_{n+m}=h(\beta_m)$, we have $\gamma_n=\alpha_n$.
%\end{lemma}

\subsection{$(g,b)$-splittings}\label{sec-11}

It is well known that every closed orientable 3-manifold $M$ has a genus-$g$ {\it Heegaard splitting} for some $g(\geq0)$, i.e., $M =V_{1}\cup_{F}V_{2}$, where $V_{1}$ and $V_{2}$ are genus-$g$ handlebodies such that $M=V_{1}\cup V_{2}$ and $V_{1}\cap V_{2}=\partial{V_{1}}=\partial{V_{2}}=F$. 
The surface $F$ is called a {\it Heegaard surface}.
Let $L$ be a link in $M$ which intersects $F$ transversely. 
We say that $(V_{1},t_1)\cup_{(F,P)}(V_{2},t_2)$ is a {\it $(g, b)$-splitting} (or {\it bridge splitting}) of the link $L$ if $F\cap L=P$ and $F$ separates $(M,L)$ into two components $(V_{1},t_1)$ and $(V_{2},t_2)$, where $t_{i}=L \cap V_{i}$ is a union of $b$ arcs properly embedded in $V_i$ which is parallel to $\partial V_i$ ($i=1,2$). 
%Then we call the surface $F$ with $2b$ punctures a {\it $(g, b)$-bridge surface} (or a {\it bridge surface} for short).
It is known that every $(M,L)$ has a $(g,b)$-splitting for some $g$ and $b$. (For a detailed discussion, see \cite[Lemma 2.1]{HS}).

For a 3-manifold $V$ and a 1-dimensional submanifold $t$ of $V$, 
the {\it disk complex} of $V\setminus t$, denoted by $\mathcal{D}(V\setminus t)$, is the subcomplex of $\mathcal{C}(\partial V\setminus t)$ spanned by the vertices with representatives bounding disks in $V\setminus t$.
%
%For $i=1$ and $2$, $\mathcal{D}(V_{i}\setminus t_i)$ denotes the subcomplex of $\mathcal{C}(\partial V_i\setminus t_i)$ spanned by the vertices with representatives bounding disks in $V_{i}\setminus t_{i}$.
%
Then the {\it (Hempel) distance} of a bridge splitting $(V_{1}, t_{1})\cup_{(F,P)}(V_{2}, t_{2})$ is defined to be  $d_{F\setminus P}(\mathcal{D}^0(V_{1}\setminus t_1), \mathcal{D}^0(V_{2}\setminus t_2))$.
%Throughout this paper, we use the symbol $\mathcal{D}(V_{i}\setminus t_i)$ to denote $\mathcal{D}^0(V_{i}\setminus t_i)$ and $\mathcal{D}(V_{i}\setminus t_i)$ itself for simplicity.
We note that it is elementary to show that $\mathcal{D}(V_{i}\setminus t_i)$ is connected by using so-called outermost disk arguments.

\section{Unique geodesics}\label{sec-geod}

Throughout this section, let $S$ be a non-sporadic genus-$g\,(\ge 0)$ orientable surface with no boundary components and $p\,(\ge 2)$ punctures. For a technical reason, we assume that $p\ge 6$ when $g=0$. 
In this section, for any integer $n\,(\ge 2)$, we construct a geodesic of length $n$ in $\mathcal{C}(S)$, by using the idea  in \cite[Section 3]{IJK2} slightly modified so that the geodesic has the uniqueness property.
%Throughout this paper, we denote by $N(\cdot)$ the regular neighborhood of an arc or a simple closed curve in $S$, and denote by ${\rm cl}(\cdot)$ the closure in $S$.

For a simple closed curve $l$ on $S$ which cuts off a twice-punctured disk from $S$ or is non-separating in $S$, we call the surface $X$ defined as follows the {\it subsurface of $S$ associated with $l$} throughout this paper.
\begin{itemize}
\item When $l$ cuts off a twice-punctured disk from $S$, $X$ is the closure of the component of $S\setminus N(l)$ which is not the twice-punctured disk.
\item When $l$ is non-separating in $S$, $X={\rm cl}(S\setminus N(l))$.
\end{itemize}

\begin{remark}\label{rmk_miss}
{\rm
Under the above notations, we note that if $m$ is an essential simple closed curve in $S$ which misses $X$, then $m=l$.
}
\end{remark}

The next proposition follows from \cite[Proposition 3.1]{IJK2} and its proof.

\begin{proposition}\label{prop-extending-geodesic}%{\rm (\cite[Proposition 3.1]{IJK2})}
Let $[l_0,l_1,\dots,l_n]$ be a path in $\mathcal{C}(S)$ $(n\ge 2)$. Assume that, for some $i$ $(1\le i\le n-1)$, the following conditions are satisfied.
\begin{itemize}
\item[{\rm (H1)}] $[l_0,l_1,\dots,l_i]$ and $[l_i,l_{i+1},\dots,l_n]$ are geodesics in $\mathcal{C}(S)$,
\item[{\rm (H2)}] $l_i$ cuts off a twice-punctured disk from $S$ or is non-separating in $S$,
\item[{\rm (H3)}] ${\rm diam}_{X_i}(\pi_{X_i}(l_0),\pi_{X_i}(l_n))>2n$, where $X_i$ is the subsurface of $S$ associated with $l_i$. 
\end{itemize} 
Then $[l_0,l_1,\dots,l_n]$ is a geodesic in $\mathcal{C}(S)$.
Moreover, any geodesic connecting $l_0$ and $l_n$ passes through $l_i$.
\end{proposition}

%\begin{proof}
%The fact that $[l_0,l_1,\dots,l_n]$ is a geodesic in $\mathcal{C}(S)$ follows from \cite[Proposition 3.1]{IJK2}, and the rest part of the proposition follows from the proof of \cite[Proposition 3.1]{IJK2}.
%\end{proof}

\begin{remark}
{\rm
We note that Remark~\ref{rmk_miss} implies: $\pi_{X_i}(l_0)$ and $\pi_{X_i}(l_n)$ in (H3) are not empty set.
}
\end{remark}

Next we prove the next proposition which asserts the existence of unique geodesics with a certain condition.

\begin{proposition}\label{prop-unique-geodesic}
Assume that $p\ge 4$.
For any $n\ge 2$, there is a geodesic $[l_0,l_1,\dots,l_n]$ in $\mathcal{C}(S)$ such that every $l_i$ cuts off a twice-punctured disk from $S$ and that $[l_0,l_1,\dots,l_n]$ is the unique geodesic connecting $l_0$ and $l_n$, 
i.e., the set of the geodesics in $\mathcal{C}(S)$ connecting $l_0$ and $l_n$ consists of one element $[l_0,l_1,\dots,l_n]$.
\end{proposition}

\begin{proof}
We construct the geodesics inductively using Proposition \ref{prop-extending-geodesic}.

\setcounter{case}{0}

\begin{case}\label{subsec-geod-length2}
$n=2$.
\end{case}

Let $l_0$ and $l_1$ be mutually disjoint simple closed curves in $S$ each of which cuts off a twice-punctured disk from $S$.
Let $X_1$ be the subsurface of $S$ associated with $l_1$.
Note that $l_0$ cuts $X_1$. 
By Lemma \ref{lem-2-3-1}, there exists a homeomorphism $h_1:S\rightarrow S$ such that $h_1(l_1)=l_1$ and ${\rm diam}_{X_1}(\pi_{X_1}(l_0), \pi_{X_1}(h_1(l_0)))>4$.
Let $l_2:=h_1(l_0)$.
Then $l_2$ also cuts off a twice-punctured disk from $S$, and $[l_0,l_1,l_2]$ is the unique geodesic connecting $l_0$ and $l_2$ by Proposition \ref{prop-extending-geodesic}.

\begin{case}
$n\ge 3$.
\end{case}

%We can extend geodesics inductively, using the geodesics constructed in the previous subsections and Proposition \ref{prop-extending-geodesic}.
Suppose we have constructed a geodesic $[l_0,l_1,\dots,l_{n-1}]$ such that every $l_i$ cuts off a twice-punctured disk from $S$ and that $[l_0,l_1,\dots,l_{n-1}]$ is the unique geodesic connecting $l_0$ and $l_{n-1}$.
Let $X_{n-1}$ be the subsurface of $S$ associated with $l_{n-1}$.
By Lemma \ref{lem-2-3-1}, there exists a homeomorphism $h_{n-1}:S\rightarrow S$ such that 
$h_{n-1}(l_{n-1})=l_{n-1}$ and ${\rm diam}_{X_{n-1}}(\pi_{X_{n-1}}(l_{0}),\pi_{X_{n-1}}(h_{n-1}(l_{n-2})))>2n$.
Let $l_{n}:=h_{n-1}(l_{n-2})$.
Then $l_n$ also cuts off a twice-punctured disk from $S$, and $[l_0,l_1,\dots,l_n]$ is the unique geodesic connecting $l_0$ and $l_n$ by Proposition \ref{prop-extending-geodesic} and the uniqueness of $[l_0,l_1,\dots,l_{n-1}]$.
\end{proof}

The next proposition can be proved similarly.

\begin{proposition}\label{prop-unique-geodesic2}
Assume that $g\ge 1$.
For any $n\ge 2$, there is a geodesic $[l_0,l_1,\dots,l_n]$ in $\mathcal{C}(S)$ such that every $l_i$ is non-separating in $S$ and that $[l_0,l_1,\dots,l_n]$ is the unique geodesic connecting $l_0$ and $l_n$.
\end{proposition}

\begin{remark}\label{rmk-unique-geodesic2}
{\rm
For the geodesic $[l_0,l_1,\dots,l_n]$ in Proposition~\ref{prop-unique-geodesic2}, we may further suppose that $l_0\cup l_1$ is separating in $S$, and $l_{n-1}\cup l_n$ is separating in $S$ by the construction. This fact will be used in Section~\ref{sec-proof-4}.}
\end{remark}

\begin{remark}\label{rem-sec3}
{\rm 
We may assume that the geodesic constructed in the above propositions %Proposition \ref{prop-unique-geodesic} 
satisfies the inequality
$$
{\rm diam}_{X_{n-1}}(\pi_{X_{n-1}}(l_{0}),\pi_{X_{n-1}}(l_{n}))>M
$$
for any given number $M>0$ if needed, or the inequality 
$$
{\rm diam}_{X_{1}}(\pi_{X_{1}}(l_{0}),\pi_{X_{1}}(l_{n}))>M
$$
by reversing the order of the vertices in the geodesic.
}
\end{remark}

\section{$(3,1)$-manifold pairs for the proof of Theorem~\ref{thm-1}}\label{sec-mfd-pairs}

In this section, we give a description of certain $(3,1)$-manifold pairs for the proof of Theorem~\ref{thm-1} and show some facts.

For $i=1,2$, let $V_i^{\ast,0}$ be a genus-$g$ handlebody and $t_i^{\ast,0}$ be the union of $b$ arcs $t_i^1,t_i^2,\dots,t_i^b$ properly embedded in $V_i^{\ast,0}$ which is parallel to $\partial V_i^{\ast,0}$.

\subsection{When $b\ge 2$}\label{sec-b>1}

Assume that $g\ge 0$, $b\ge 2$, and $(g,b)\ne (0,2)$.
For $i=1,2$, let $V_i\,(\subset V_i^{\ast,0})$ be a genus-$g$ handlebody such that 
\begin{itemize}
\item $t_i:=t_i^{\ast,0}\cap V_i$ is the union of $(b-1)$ arcs which is parallel to $\partial V_i$,
\item $W_i:={\rm cl}(V_i^{\ast,0}\setminus V_i)\cong \Sigma\times I$, where $\Sigma$ is a genus-$g$ closed orientable surface and $I=[0,1]$, and
\item $s_i:=t_i^{\ast,0}\cap W_i$ is the union of $2(b-1)$ $I$-fibers ($\subset \Sigma\times I$) and $t_i^b$.
\end{itemize}
Let $D_i$ be the disk properly embedded in $W_i$ as in Figure \ref{fig-wi-si}.
Then the closures of the components of $W_i\setminus D_i$ consists of two components $W_i^1$, $W_i^2$ such that 
$W_i^1\cong \Sigma\times I$, where $s_i^1(:=t_i^{\ast,0}\cap W_i^1)$ is the union of $2(b-1)$ $I$-fibers, $W_i^2$ is a $3$-ball and $s_i^2(:=t_i^b)$ is an arc parallel to $\partial W_i^2$.
\begin{figure}[tb]
 \begin{center}
 \includegraphics[width=46mm]{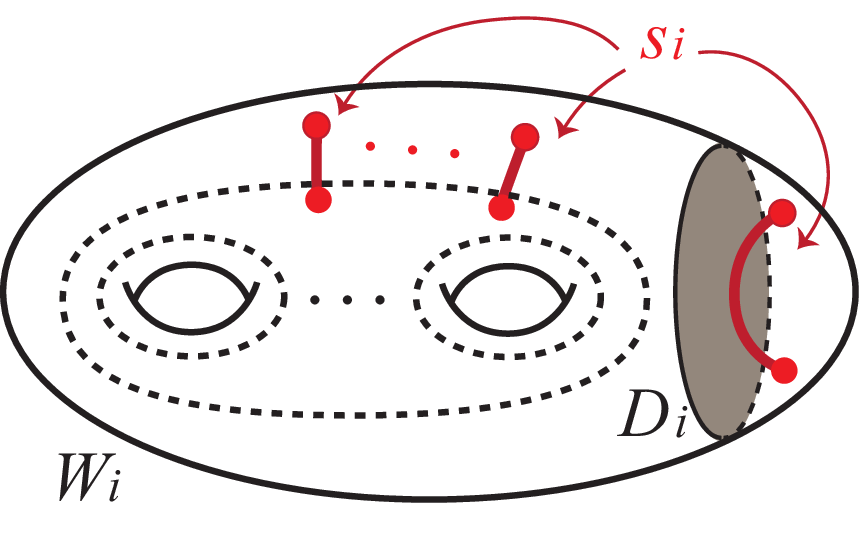}
 \end{center}
 \caption{$(W_i, s_i)$ and $D_i$.}
\label{fig-wi-si}
\end{figure}
Let $\partial_- W_i$ be the component of $\partial W_i$ disjoint from $D_i$, and let $\partial_+ W_i=\partial W_i\setminus \partial_- W_i$.
Note that $s_i\cap\,\partial_- W_i$ consists of $(2b-2)$ points, and $s_i\cap\,\partial_+ W_i$ consists of $2b$ points. 
Let $F_i$ be the subsurface $\partial_+ W_i\cap W_i^1$ of $\partial_+ W_i$.
Let $\pi_{F_i\setminus s_i}:\mathcal{C}^0(\partial_+ W_i\setminus s_i)\rightarrow \mathcal{P}(\mathcal{C}^0(F_i\setminus s_i))$ be the subsurface projection, and let $P_i:F_i\setminus s_i\rightarrow (F_i\setminus s_i)\cup D_i\rightarrow \partial_- W_i\setminus s_i$ be the natural projection.
%We note that, since $\partial F_i$ is connected, $P_i(a)\ne\emptyset$ for any $a\in\mathcal{C}^0(F_i)$
Let $\Phi_i:\mathcal{C}^0(\partial_+ W_i\setminus s_i)\rightarrow \mathcal{P}(\mathcal{C}^0(\partial_- W_i\setminus s_i))$ be the composition $P_i\circ \pi_{F_i\setminus s_i}$.
Let $h_i:\partial V_i\setminus t_i\rightarrow \partial_- W_i\setminus s_i$ be a homeomorphism, and let $\overline{h}_i:(\partial V_i,\partial t_i)\rightarrow (\partial_- W_i,s_i\cap \partial_- W_i)$ be the homeomorphism of pairs naturally induced from $h_i$.
Let $(V_i^{\ast}, t_i^{\ast}):=(W_i,s_i)\cup_{\overline{h}_i}(V_i,t_i)$.
Then $V_i^{\ast}$ is a genus-$g$ handlebody and $t_i^{\ast}$ is the union of $b$ arcs parallel to $\partial V_i^{\ast}$.

Recall that $\mathcal{D}^0(V_i\setminus t_i)$ denotes the $0$-skeleton of the disk complex of $V_i\setminus t_i$.
The next proposition can be proved by \cite[Claim 2]{IS} and Lemma \ref{lem-2-3-1}.
\begin{proposition}\label{prop-b>1-hi}
For $\alpha\in\mathcal{C}^0(\partial_+ W_i\setminus s_i)$ such that $\Phi_i(\alpha)\ne \emptyset$ and any positive integer $k$, there exists a homeomorphism $h_i:\partial V_i\setminus t_i\rightarrow \partial_- W_i\setminus s_i$ such that 
$$d_{\partial_- W_i\setminus s_i}(\Phi_i(\alpha),h_i(\mathcal{D}^0(V_i\setminus t_i)))>k.$$
%and that $\overline{h}_i|_{\partial t_i}:\partial t_i\rightarrow \partial s_i\cap \partial_- W_i$ corresponds to ${\rm id}_{\partial t_i}$, where $s_i\cap \partial_- W_i$ is identified with $\partial t_i$.
\end{proposition}

%Recall that $\mathcal{D}(V_i\setminus t_i)$ is the disk complex of $V_i\setminus t_i$.
%For $\alpha\in\mathcal{C}^0(\partial_+ W_i\setminus s_i)$ such that $\Phi_i(\alpha)\ne \emptyset$ and any positive integer $k$, there exists a homeomorphism $h_i:\partial V_i\setminus t_i\rightarrow \partial_- W_i\setminus s_i$ such that
%\begin{equation}\label{hi}
%d_{\partial_- W_i\setminus s_i}(\Phi_i(\alpha),h_i(\mathcal{D}^0(V_i\setminus t_i)))>k,
%\end{equation}
%by \cite[Claim 2]{IS} and Lemma \ref{lem-2-3-1}.
%
%Recall that $\mathcal{D}(V_i^{\ast}\setminus t_i^{\ast})$ is the disk complex of $V_i^{\ast}\setminus t_i^{\ast}$.

The next proposition will be used in Sections~\ref{sec-proof-1}, \ref{sec-proof-7} and \ref{sec-proof-8}.
\begin{proposition}\label{prop-b>1}
Let $\alpha$ be an element of $\mathcal{C}^0(\partial_+ W_i\setminus s_i)$ such that $\alpha\cap \partial D_i=\emptyset$ and $\Phi_i(\alpha)\ne \emptyset$ (hence, $\alpha\subset F_i$).
Then the following hold.
%, and $\beta$ an element of $\mathcal{D}(V_i^{\ast}\setminus t_i^{\ast})$.
%Let $D_{\beta}$ be a disk in $V_i^{\ast}\setminus t_i^{\ast}$ bounded by $\beta$.
%Assume that either of the following holds:
%\begin{itemize}
%\item[{\rm (i)}] $\beta\cap \partial D_i=\emptyset$, and $D_{\beta}$ is an essential disk in $(W_i^1\cup_{\overline{h}_i}V_i)\setminus t_i^{\ast}$,
%\item[{\rm (ii)}] $\beta\cap \partial D_i\ne\emptyset$, and the closure of a component of $D_{\beta}\setminus D_i$ outermost in $D_{\beta}$ bounds an essential disk in $(W_i^1\cup_{\overline{h}_i}V_i)\setminus t_i^{\ast}$.
%\end{itemize}
%Then the following hold.
\begin{itemize}
\item[{\rm (1)}] If there is an element $\beta$ of $\mathcal{D}^0(V_i^{\ast}\setminus t_i^{\ast})$ such that $\alpha\cap\beta=\emptyset$ and $\beta\ne \partial D_i$, then $d_{\partial_- W_i\setminus s_i}(\Phi_i(\alpha),h_i(\mathcal{D}^0(V_i\setminus t_i)))\le 1$. 
Moreover, if $\alpha\in \mathcal{D}^0(V_i^{\ast}\setminus t_i^{\ast})$ (that is, there is an element $\beta$ of $\mathcal{D}^0(V_i^{\ast}\setminus t_i^{\ast})$ such that $\alpha=\beta$) then $d_{\partial_- W_i\setminus s_i}(\Phi_i(\alpha),h_i(\mathcal{D}^0(V_i\setminus t_i)))=0$.
\item[{\rm (2)}] Suppose that $\alpha$ bounds a twice-punctured disk in $\partial_+ W_i\setminus s_i$ and that there is an element $\beta$ of $\mathcal{D}^0(V_i^{\ast}\setminus t_i^{\ast})$ such that $|\alpha\cap\beta|\le 2$. 
Then $d_{\partial_- W_i\setminus s_i}(\Phi_i(\alpha),h_i(\mathcal{D}^0(V_i\setminus t_i)))\le 2$. 
%If $\alpha$ bounds a twice-punctured disk in $\partial_+ W_i\setminus s_i$ and $|\alpha\cap\beta|\le 2$, then $d_{\partial_- W_i\setminus s_i}(\Phi_i(\alpha),h_i(\mathcal{D}^0(V_i\setminus t_i)))\le 2$. 
\end{itemize}
\end{proposition}

\begin{proof}
%We give a proof for the case when $i=1$, since the case when $i=2$ can be treated similarly.
Note that $D_i$ cuts $(V_i^{\ast},t_i^{\ast})$ into $(W_i^1,s_i^1)\cup_{\overline{h}_i}(V_i,t_i)$ and $(W_i^2,s_i^2)$, where $W_i^2$ is a 3-ball and $s_i^2$ is an arc parallel to $\partial W_i^2$.
Let $\beta$ be an element of $\mathcal{D}^0(V_i^{\ast}\setminus t_i^{\ast})$.
Let $D_{\beta}$ be a disk in $V_i^{\ast}\setminus t_i^{\ast}$ bounded by $\beta$.
We may assume that $|D_{\beta}\cap D_i|$ is minimal.

If $|D_{\beta}\cap D_i|=0$, then let $\Delta'=D_{\beta}$.
Note that $\Delta'$ is an essential disk in $(W_i^1\cup_{\overline{h}_i} V_i)\setminus t_i^{\ast}$, because $\beta\ne \partial D_i$, $W_i^2$ is a 3-ball and $s_i^2$ is an arc parallel to $\partial W_i^2$

If $|D_{\beta}\cap D_i|>0$, we see that $D_{\beta}\cap D_i$ has no loop components by using innermost disk arguments. 
In this case, let $\Delta$ be the closure of a component of $D_{\beta}\setminus D_i$ that is outermost in $D_{\beta}$.
%Note that $D_i$ cuts $(V_i^{\ast},t_i^{\ast})$ into $(W_i^1,s_i^1)\cup_{h_i} (V_i,t_i)$ and $(W_i^2,s_i^2)$.
Note that there is no essential disk in $W_i^2\setminus s_i^2$ since $W_i^2$ is a $3$-ball and $s_i^2$ is an arc parallel to $\partial W_i^2$.
By the minimality of $|D_{\beta}\cap D_i|$, we see that $\Delta$ must be an essential disk in $(W_i^1\cup_{\overline{h}_i} V_i)\setminus t_i^{\ast}$. 
Let $\Delta'$ be a disk properly embedded in $(W_i^1\cup_{\overline{h}_i} V_i)\setminus t_i^{\ast}$ with $\partial\Delta'\subset F_i$, such that $\Delta'$ is parallel to the union of $\Delta$ and one of the two components of $D_i\setminus\Delta$ (see Figure~\ref{fig-delta1}).
%Note that $\partial \Delta'\cap \alpha=\emptyset$ since $\alpha\cap \partial D_1=\emptyset$.
\begin{figure}[tb]
 \begin{center}
 \includegraphics[width=90mm]{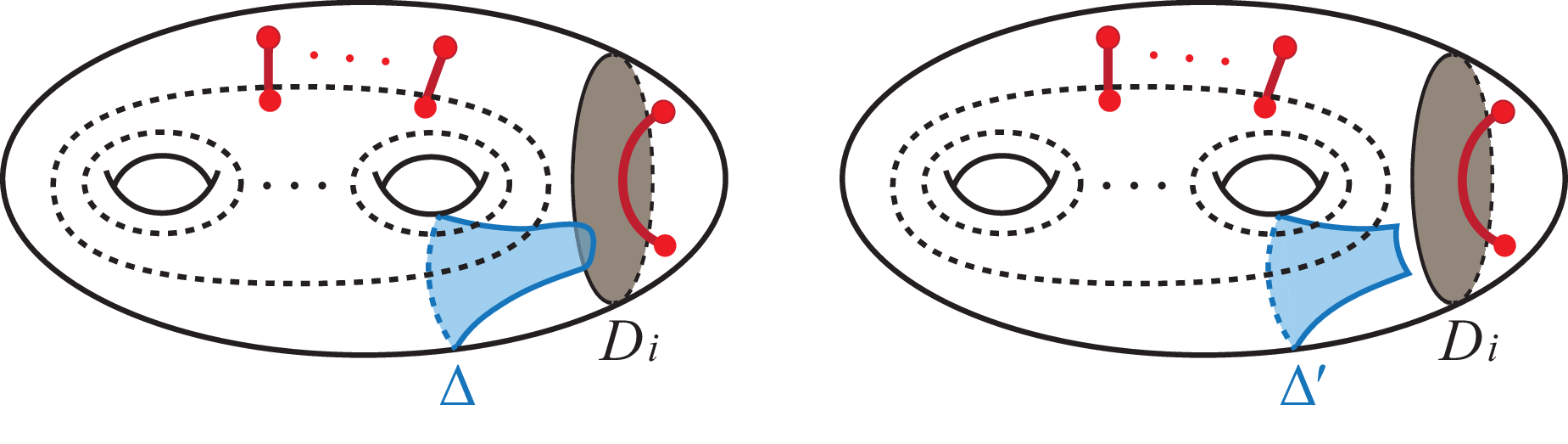}
 \end{center}
 \caption{$\Delta$ and $\Delta'$.}
\label{fig-delta1}
\end{figure}

Since $s_i^1$ is the union of $I$-fibers in $W_i^1(\cong \Sigma_g\times I)$, 
%after boundary compressions toward $\partial_- W_i$ if necessary, 
by applying ambient isotopy,
we may assume that $\Delta'':=\Delta'\cap V_i$ consists of a single disk and $\Delta'\cap W_i^1$ is a vertical annulus.
Then we have $h_i(\partial\Delta'')\in h_i(\mathcal{D}^0(V_i\setminus t_i))$ since $\Delta''$ is an essential disk in $V_i\setminus t_i$. 
\vspace{2mm}

(1) If $\alpha\cap\beta=\emptyset$, then $\partial \Delta'\cap \alpha=\emptyset$, and hence $\Phi_i(\alpha)\cap h_i(\partial \Delta'')(=\Phi_i(\alpha)\cap \Phi_i(\partial \Delta'))=\emptyset$.
%$h_i(\partial\Delta'')(=\Phi_i(\partial\Delta'))\cap \Phi_i(\alpha)=\emptyset$.
Hence we have 
%\begin{eqnarray*}
%\begin{array}{rcl}
%d_{\partial_- W_1\setminus s_1}(\Phi_1(\alpha_1),h_1(\mathcal{D}^0(V_1\setminus t_1)))&\le& {\rm diam}_{\partial_- W_1\setminus s_1}(\Phi_1(\alpha_1),\Phi_1(a))\\&&+ d_{\partial_- W_1\setminus s_1}(\Phi_1(a),h_1(\mathcal{D}^0(V_1\setminus t_1)))\\
%&\le &2+0=2.
%\end{array}
%\end{eqnarray*}.
$$
d_{\partial_- W_i\setminus s_i}(\Phi_i(\alpha),h_i(\mathcal{D}^0(V_i\setminus t_i)))\le 
d_{\partial_- W_i\setminus s_i}(\Phi_i(\alpha),h_i(\partial\Delta''))\le 1.
$$

Moreover, if $\alpha=\beta$,
%$\alpha\in \mathcal{D}^0(V_i^{\ast}\setminus t_i^{\ast})$, then we regard $\beta$ above as $\alpha$. 
then $h_i(\partial\Delta'')=\Phi_i(\partial\Delta')=\Phi_i(\partial D_{\beta})=\Phi_i(\beta)=\Phi_i(\alpha)\in \mathcal{D}^0(V_i\setminus t_i)$, and hence we have
$$
d_{\partial_- W_i\setminus s_i}(\Phi_i(\alpha),h_i(\mathcal{D}^0(V_i\setminus t_i)))\le 
d_{\partial_- W_i\setminus s_i}(\Phi_i(\alpha),h_i(\partial\Delta''))=0.
$$

(2) Assume that $\alpha$ bounds a twice-punctured disk in $\partial_+ W_i\setminus s_i$ and  that $|\alpha\cap\beta|\le 2$.
Since $\alpha$ is separating, either $|\alpha\cap\beta|=0$ or $|\alpha\cap\beta|=2$ holds.
If $|\alpha\cap\beta|=0$, then $d_{\partial_- W_i\setminus s_i}(\Phi_i(\alpha),h_i(\partial\Delta''))\le 1<2$ holds by the arguments in the above (1).
Hence, we assume that $|\alpha\cap\beta|=2$ in the rest of the proof.

Let $\beta'$ be the closure of the component of $\partial \Delta'\setminus \alpha$ that is not contained in the twice-punctured disk bounded by $\alpha$ (see Figure~\ref{fig-delta2}).
It is easy to see that $\beta'$ together with at least one of the two components of $\alpha\setminus \beta'$ forms an element of $\mathcal{C}^0(\partial_+W_i\setminus s_i)$.
Let $\beta''$ be the element of $\mathcal{C}^0(\partial_+W_i\setminus s_i)$.
%Let $\alpha'$ be the closures of either of the components of $\alpha\setminus \partial \Delta'$, and $\beta'$ be the component of $\partial \Delta'\setminus \alpha$ that is not contained in the twice-punctured disk bounded by $\alpha$ (see Figure~\ref{fig-delta2}).
%Let $\beta'':=\beta'\cup\alpha'$.
\begin{figure}[tb]
 \begin{center}
 \includegraphics[width=90mm]{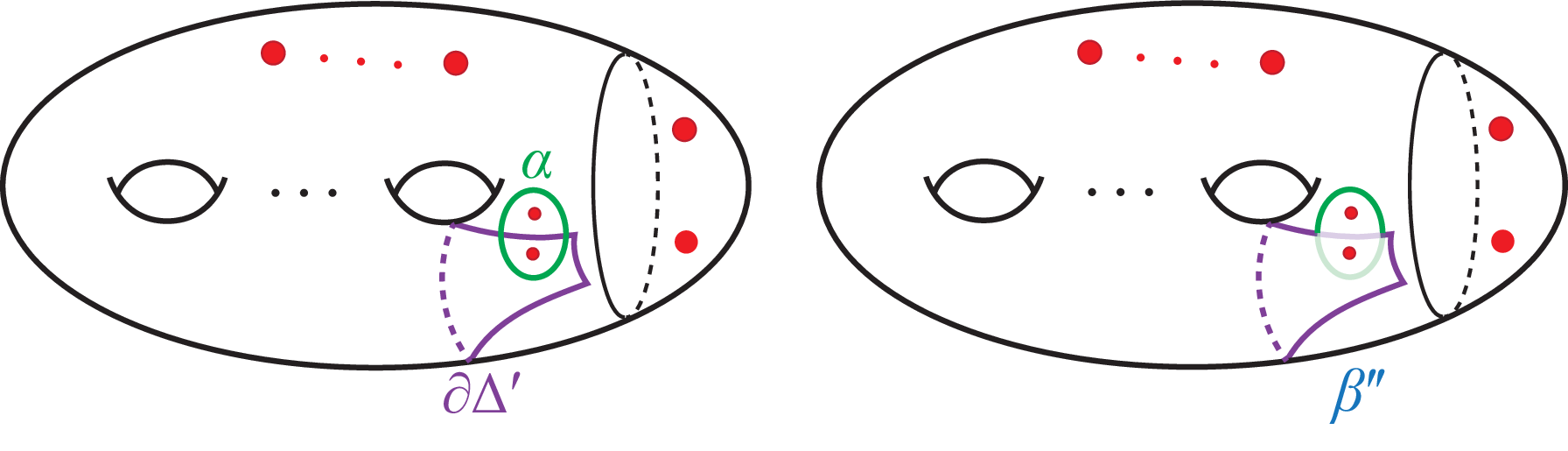}
 \end{center}
 \caption{$\beta''$.}
\label{fig-delta2}
\end{figure}
%
%We see that $P_i(\beta'')$ is essential in $\partial_- W_i$ as follows.
%Assume on the contrary that $P_i(\beta'')$ is inessential.
%Then $P_i(\beta'')$ bounds a disk in $\partial_- W_i$ with at most one puncture.
%By the minimality of $|D_{\beta}\cap D_i|$, we see that the disk must contain a puncture.
%Hence, $\partial \Delta\cap F_i$ cuts off a once-punctured disk from $F_i$.
%This implies that $h_i^{-1}\circ P_i(\partial \Delta')(=\partial \Delta'')$ bounds a once-punctured disk in $\partial V_i$.
%However, this is impossible since $\Delta''\cap t_i=\emptyset$,
%and this fact implies that each component of $\partial V_i\setminus \partial \Delta'$ contains even number of points of $\partial t_i$.
%
We note that $[\Phi_i(\alpha)(=P_i(\alpha)),P_i(\beta''), h_i(\partial \Delta'')]$ is a path in $\mathcal{C}(\partial_-W_i\setminus s_i)$.
Since $\partial\Delta''\in\mathcal{D}^0(V_i\setminus t_i)$, we have
%Recall that $\alpha$ cobounds a vertical annulus with $\Phi_1(\alpha)$ in $W_1^1(\cong \Sigma_g\times I)$.
%We have $|h_1(\partial\Delta'')\cap \Phi_1(\alpha)|=|\Phi_1(\partial\Delta')\cap \Phi_1(\alpha)|=|\partial\Delta'\cap \alpha|\le |\beta\cap \alpha|\le 2$.
%Hence, 
$$
d_{\partial_- W_i\setminus s_i}(\Phi_i(\alpha),h_i(\mathcal{D}^0(V_i\setminus t_i)))\le 
d_{\partial_- W_i\setminus s_i}(\Phi_i(\alpha),h_i(\partial\Delta''))\le 2.
$$
\end{proof}

%The following proposition follows from the proof of Proposition~\ref{prop-b>1}.

The next proposition will be used in Sections~\ref{sec-proof-7} and \ref{sec-proof-8}.

\begin{proposition}\label{prop-b>1-1}
Let $D$ be an essential disk in $V_i^{\ast}\setminus t_i^{\ast}$ such that $D\ne D_i$ and $|D\cap D_i|$ is minimal (hence, no component of $D\cap D_i$ is a loop), and let $\Delta$ be a disk defined as follows:
\begin{itemize}
\item If $D\cap D_i=\emptyset$, let $\Delta:=D$. 
\item If $D\cap D_i\ne\emptyset$, let $\Delta$ be the closure of a component of $D\setminus D_i$ that is outermost in $D$. 
\end{itemize}
Then the following hold.
\begin{itemize}
\item[{\rm (1)}] If there is an element $\alpha$ of $\mathcal{C}^0(\partial_+ W_i\setminus s_i)$ such that $\alpha\cap \partial D_i=\emptyset$, $\Phi_i(\alpha)\ne \emptyset$ and $\alpha\cap\Delta=\emptyset$,
then $d_{\partial_- W_i\setminus s_i}(\Phi_i(\alpha),h_i(\mathcal{D}^0(V_i\setminus t_i)))\le 1$.
\item[{\rm (2)}] If there is an element $\alpha=\gamma_1\cup \gamma_2$ of $\mathcal{C}^0(\partial_+ W_i\setminus s_i)$ such that $\gamma_1$ is an essential arc in $F_i$ and $\gamma_2$ is a subarc of $\partial D_i$ and that $|\alpha\cap\Delta|=|\gamma_2\cap\Delta|=1$,
then $d_{\partial_- W_i\setminus s_i}(\Phi_i(\alpha),h_i(\mathcal{D}^0(V_i\setminus t_i)))\le 2$. 
%If $|\alpha\cap\Delta|\le 2$, then $d_{\partial_- W_i\setminus s_i}(\Phi_i(\alpha),h_i(\mathcal{D}^0(V_i\setminus t_i)))\le 2$. 
\end{itemize}
\end{proposition}

\begin{proof}
By the proof of Proposition~\ref{prop-b>1}, we see that $\Delta$ is an essential disk in $(W_i^1\cup_{\overline{h}_i} V_i)\setminus t_i^{\ast}$. Then:

(1) follows from the proof of Proposition~\ref{prop-b>1} (1).

(2) %Let $\Delta'$ and $\Delta''$ be as in the proof of Proposition~\ref{prop-b>1}.
%We give a proof for the case when $i=1$.
Let $\alpha=\gamma_1\cup \gamma_2$ be as in the proposition. 
Define the disks $\Delta'$ and $\Delta''$ as in the proof of Proposition~\ref{prop-b>1}.
Note that $|\alpha\cap\Delta'|=|\alpha\cap\Delta|=1$ (see Figure~\ref{fig-delta3}).
\begin{figure}[tb]
 \begin{center}
 \includegraphics[width=90mm]{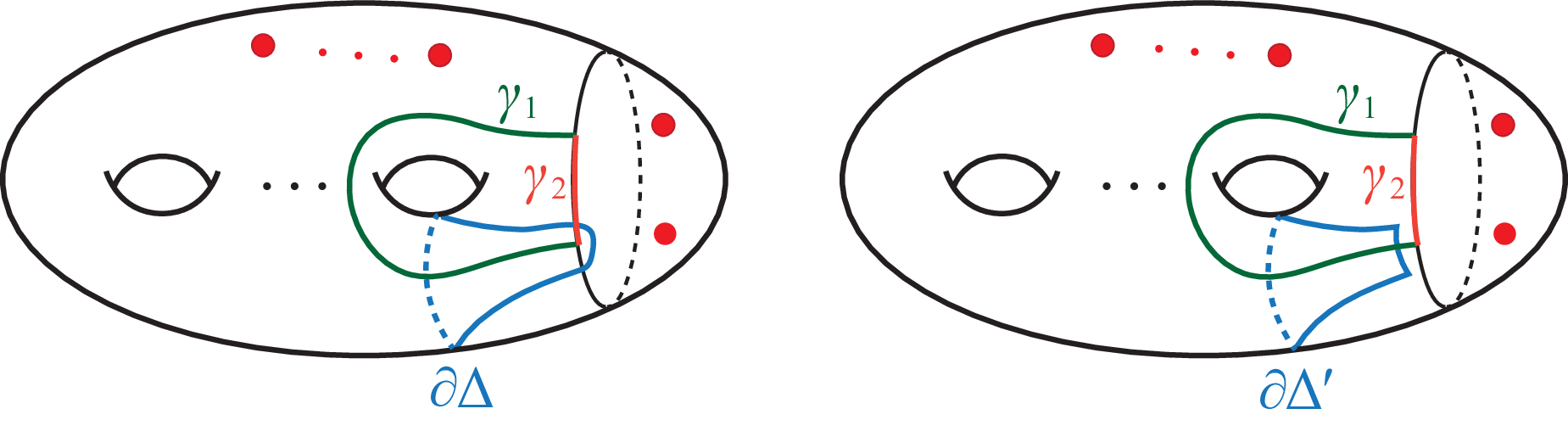}
 \end{center}
 \caption{$\alpha=\gamma_1\cup\gamma_2$, $\partial\Delta$ and $\partial\Delta'$.}
\label{fig-delta3}
\end{figure}
Thus, both $\alpha$ and $\partial \Delta'$ are non-separating in $F_i$.
Let $\delta$ be the boundary of a regular neighborhood $N_{F_i}(\alpha\cup \partial \Delta')$ (see Figure~\ref{fig-delta4}).
\begin{figure}[tb]
 \begin{center}
 \includegraphics[width=48mm]{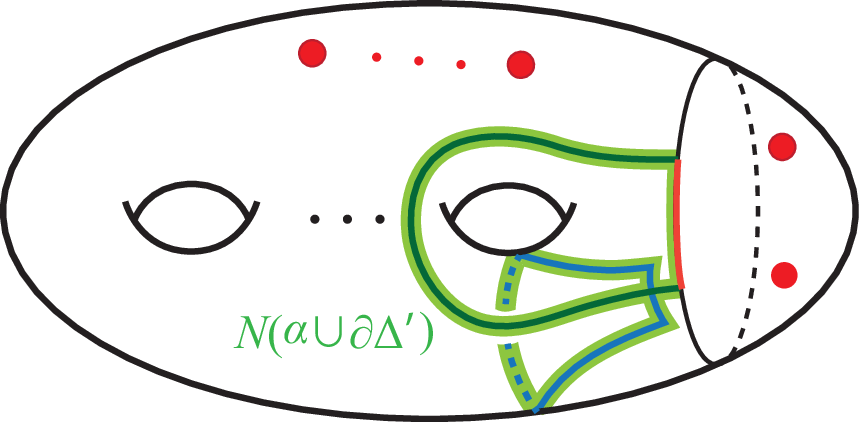}
 \end{center}
 \caption{$N(\alpha\cup \partial \Delta')$.}
\label{fig-delta4}
\end{figure}
Note that $\delta$ is essential in $F_i$ since it cuts off a one-holed torus $N_{F_i}(\alpha\cup \partial \Delta')$ with no punctures and $b\ge 2$ (by the assumption of Subsection~\ref{sec-b>1}).
Hence, $[\alpha,\delta,\partial\Delta']$ is a path in $\mathcal{C}(F_i\setminus s_i)$.
Since $W_i^1\cong \Sigma_g\times I$ and $s_i\cap W_i^1$ is the union of $I$-fibers, $[\Phi_i(\alpha)(=P_i(\alpha)), P_i(\delta), h_i(\partial\Delta'')(=P_i(\partial\Delta'))]$ is a path of length 2 in $\mathcal{C}(\partial _-W_1\setminus s_i)$.
Hence, we have 
$d_{\partial_- W_i\setminus s_i}(\Phi_i(\alpha),h_i(\mathcal{D}^0(V_i\setminus t_i)))\le 2$. 
\end{proof}

\subsection{When $g\ge 2$}\label{sec-g>1}

Assume that $g\ge 2$ and $b\ge 1$.
For $i=1,2$, let $V_i (\subset V_i^{\ast,0})$ be a genus-$(g-1)$ handlebody such that
\begin{itemize}
\item $t_i:=t_i^{\ast,0}\cap V_i$ is the union of $b$ arcs which is parallel to $\partial V_i$,
\item $W_i:={\rm cl} (V_i^{\ast,0}\setminus V_i)\cong (\Sigma\times I)\cup (1\text{-handle})$, where $\Sigma$ is a genus-$(g-1)$ closed orientable surface, and the $1$-handle is attached to $\Sigma\times \{1\}$, and
\item $s_i:=t_i^{\ast, 0}\cap W_i$ is the union of $I$-fibers in $\Sigma\times I$.
\end{itemize}
%Note that $s_i:=t_i^{\ast, 0}\cap W_i$ is the union of $2b$ $I$-fibers ($\subset \Sigma\times I$).
Let $D_i$ be the co-core of the $1$-handle ($\cong D_i\times I$), hence $(D_i\times\{0\}) \cup (D_i\times\{1\})\,(\subset\Sigma\times\{1\})$ is the attaching disks of the $1$-handle.
See Figure \ref{fig-wi-si3-1}.
\begin{figure}[tb]
 \begin{center}
 \includegraphics[width=50mm]{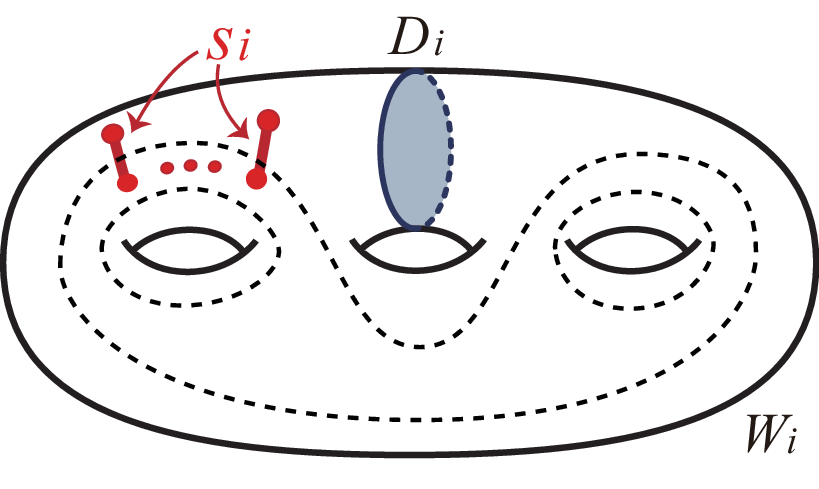}
 \end{center}
 \caption{$(W_i, s_i)$ and $D_i$.}
\label{fig-wi-si3-1}
\end{figure}
In the rest of this subsection, $W_i'$ denotes the submanifold of $W_i$ corresponding to $\Sigma\times I$.
Further, $\partial_-W_i'$ (resp. $\partial_+ W_i'$) denotes the component of $\partial W_i'$ corresponding to $\Sigma\times\{0\}$ (resp. $\Sigma\times\{1\}$). 
In the remainder of this paper, $\partial_- W_i$ denotes the surface corresponding to $\partial_- W_i'$ if it is regarded as a boundary component of $W_i$.
Then let $\partial_+ W_i:=\partial W_i\setminus \partial_-W_i$.
Let $F_i:= \partial_+W_i'\cap \partial_+W_i$. 
(Note that $F_i$ is the closure of $\partial_+W_i'\setminus (D_i\times \{0,1\})$.)
Let $\pi_{F_i\setminus s_i}:\mathcal{C}^0(\partial_+ W_i\setminus s_i)\rightarrow \mathcal{P}(\mathcal{C}^0(F_i\setminus s_i))$ be the subsurface projection, and let $P_i:F_i\setminus s_i\rightarrow (F_i\setminus s_i)\cup (D_i\times\{0,1\})\rightarrow \partial_- W_i\setminus s_i$ be the natural projection.
%We note that, since $\partial F_i$ is connected, $P_i(a)\ne\emptyset$ for any $a\in\mathcal{C}^0(F_i)$
Let $\Phi_i:\mathcal{C}^0(\partial_+ W_i\setminus s_i)\rightarrow \mathcal{P}(\mathcal{C}^0(\partial_- W_i\setminus s_i))$ be the composition $P_i\circ \pi_{F_i\setminus s_i}$.
Let $h_i:\partial V_i\setminus t_i\rightarrow \partial_- W_i\setminus s_i$ be a homeomorphism, and let $\overline{h}_i:(\partial V_i,\partial t_i)\rightarrow (\partial_- W_i,s_i\cap \partial_- W_i)$ be the homeomorphism of pairs naturally induced from $h_i$, and let $(V_i^{\ast}, t_i^{\ast}):=(W_i,s_i)\cup_{\overline{h}_i}(V_i,t_i)$.
Then $V_i^{\ast}$ is a genus-$g$ handlebody and $t_i^{\ast}$ is the union of $b$ arcs parallel to $\partial V_i^{\ast}$.

Recall that $\mathcal{D}(V_i^{\ast}\setminus t_i^{\ast})$ is the disk complex of $V_i^{\ast}\setminus t_i^{\ast}$.
The following can be proved by \cite[Claim 2]{IS} and Lemma \ref{lem-2-3-1}.

\begin{proposition}\label{prop-g>1-hi}
For $\alpha\in\mathcal{C}^0(\partial_+ W_i\setminus s_i)$ such that $\Phi_i(\alpha)\ne \emptyset$ and any positive integer $k$, there exists a homeomorphism $h_i:\partial V_i\setminus t_i\rightarrow \partial_- W_i\setminus s_i$ such that
$$
d_{\partial_- W_i\setminus s_i}(\Phi_i(\alpha),h_i(\mathcal{D}^0(V_i\setminus t_i)))>k.
$$
%and that $\overline{h}_i|_{\partial t_i}:\partial t_i\rightarrow \partial s_i\cap \partial_- W_i$ corresponds to ${\rm id}_{\partial t_i}$.
\end{proposition}

The following propositions %can be proved by arguments similar to those for the proofs of Propositions~\ref{prop-b>1} and \ref{prop-b>1-1}, and 
will be used in Section~\ref{sec-proof-6}.

\begin{proposition}\label{prop-g>1}
Let $\alpha$ be an element of $\mathcal{C}^0(\partial_+ W_i\setminus s_i)\setminus\{\partial D_i\}$ such that $\alpha$ is non-separating, $\alpha\cap D_i=\emptyset$, and $\alpha\cup \partial D_i$ is separating in $\partial_+W_i\setminus s_i$. %, and $\beta$ an element of $\mathcal{D}^0(V_i^{\ast}\setminus t_i^{\ast})$.
%Then the following hold.
%\begin{itemize}
%\item[{\rm (1)}] 
%If $\alpha\cap\beta=\emptyset$, 
If there is an element $\beta$ of $\mathcal{D}^0(V_i^{\ast}\setminus t_i^{\ast})$ such that $\alpha\cap \beta=\emptyset$ and $\beta\ne \partial D_i$,
then $d_{\partial_- W_i\setminus s_i}(\Phi_i(\alpha),h_i(\mathcal{D}^0(V_i\setminus t_i)))\le 1$. %Moreover, if $\alpha=\beta$ then $d_{\partial_- W_i\setminus s_i}(\Phi_i(\alpha),h_i(\mathcal{D}^0(V_i\setminus t_i)))=0$.
%\item[{\rm (2)}] If $|\alpha\cap\beta|=1$, then $d_{\partial_- W_i\setminus s_i}(\Phi_i(\alpha),h_i(\mathcal{D}^0(V_i\setminus t_i)))\le 2$. 
%\end{itemize}
\end{proposition}

\begin{proof}
Let $D_{\beta}$ be a disk in $V_i^{\ast}\setminus t_i^{\ast}$ bounded by $\beta$.
We may assume that $|D_{\beta}\cap D_i|$ is minimal (hence, no component of $D_{\beta}\cap D_i$ is an loop).
%If $|D_{\beta}\cap D_i|>0$, we see that $D_{\beta}\cap D_i$ has no loop component by the minimality of $|D_{\beta}\cap D_i|$.
We may suppose that each component of $D_{\beta}\cap \text{\rm (1-handle)}(=D_{\beta}\cap (D_i\times I))$ is a product disk in $D_i\times I$.
Let $\Delta$ be the closure of a component of $D_{\beta}\setminus\text{\rm (1-handle)}$ that is outermost in $D_{\beta}$.
Then $\Delta$ is a disk properly embedded in $(W_i'\cup_{\overline{h}_i} V_i)\setminus t_i^{\ast}$.
Note that $\alpha$ is a simple closed curve in $\partial_+W_i'$ which separates $D_i\times\{0\}$ and $D_i\times\{1\}$.
Since $\Delta\cap\alpha=\emptyset$ by the assumption $\alpha\cap\beta=\emptyset$, it is easy to see that $\Delta$ is an essential disk in $(W_i'\cup_{\overline{h}_i} V_i)\setminus t_i^{\ast}$.
Then the arguments in the proof of Proposition~\ref{prop-b>1} using $\Delta'$, $\Delta''$ work in this setting to show $d_{\partial_- W_i\setminus s_i}(\Phi_i(\alpha),h_i(\mathcal{D}^0(V_i\setminus t_i)))\le 1$.
%This can be proved by arguments similar to those for Proposition~\ref{prop-b>1}.
%Let $D$ be a disk in $V_i^{\ast}\setminus t_i^{\ast}$ bounded by $\beta$, and let $\Delta$ be a disk as in Proposition~\ref{prop-g>1-1}.
%We note that, by using the facts that $\alpha\cap\partial D_i=\emptyset$, that each of $\alpha$ and $\partial D_i$ is non-separating and that their union is separating, we can see that $\Phi_i(\alpha)\ne\emptyset$ and that $\Delta$ %defined as in the proof of Proposition~\ref{prop-b>1} must be essential in $(W_i'\cup_{\overline{h}_i} V_i)\setminus t_i^{\ast}$.
\end{proof}

\begin{proposition}\label{prop-g>1-1}
Let $D$ be an essential disk in $V_i^{\ast}\setminus t_i^{\ast}$ such that $D\ne D_i$ and $|D\cap D_i|$ is minimal (hence, no component of $D\cap D_i$ is a loop), and let
%$D$ intersects $D_i$ transversely, and that no component of $D\cap D_i$ is a simple closed curve.
%Let 
$\Delta$ be a disk defined as follows:
\begin{itemize}
\item If $D\cap D_i=\emptyset$, let $\Delta:=D$. 
\item If $D\cap D_i\ne\emptyset$, let $\Delta$ be the closure of a component of $D\setminus N(D_i)$ that is outermost in $D$. 
\end{itemize}
%Assume that $\Delta$ is essential in $(W_i'\cup_{\overline{h}_i} V_i)\setminus t_i^{\ast}$.
Then one of the following {\rm (A)} or {\rm (B)} holds.
\begin{itemize}
%\item[{\rm (A)}] $D=D_i$,
\item[{\rm (A)}] $D\cap D_i=\emptyset$ and $D$ is a band sum of two copies of $D_i$,
\item[{\rm (B)}] $\Delta$ is essential in $(W_i'\cup_{\overline{h}_i} V_i)\setminus t_i^{\ast}$, and the following {\rm (B1)} and {\rm (B2)} hold.
\begin{itemize}
\item[{\rm (B1)}] If there is an element $\alpha$ of $\mathcal{C}^0(\partial_+ W_i\setminus s_i)$ such that $\alpha\cap \partial D_i=\emptyset$, $\Phi_i(\alpha)\ne \emptyset$ and $\alpha\cap\Delta=\emptyset$,
then $d_{\partial_- W_i\setminus s_i}(\Phi_i(\alpha),h_i(\mathcal{D}^0(V_i\setminus t_i)))\le 1$.
\item[{\rm (B2)}] If there is an element $\alpha=\gamma_1\cup \gamma_2$ of $\mathcal{C}^0(\partial_+ W_i\setminus s_i)$ such that $\gamma_1$ is an essential arc in $F_i$ and $\gamma_2$ is a subarc of $\partial (D_i\times\{0,1\})$ and that $|\alpha\cap\Delta|=|\gamma_2\cap\Delta|=1$,
then  $d_{\partial_- W_i\setminus s_i}(\Phi_i(\alpha),h_i(\mathcal{D}^0(V_i\setminus t_i)))\le 2$. 
% If $|\alpha\cap\Delta|=1$, then $d_{\partial_- W_i\setminus s_i}(\Phi_i(\alpha),h_i(\mathcal{D}^0(V_i\setminus t_i)))\le 2$. 
\end{itemize}
\end{itemize}
\end{proposition}

\begin{proof}
We note that Proposition~\ref{prop-g>1-1} corresponds to Proposition~\ref{prop-b>1-1}.
In Proposition~\ref{prop-b>1-1}, the minimality of $|D\cap D_i|$ implies that $\Delta$ is essential in $(W_i'\cup_{\overline{h}_i} V_i)\setminus t_i^{\ast}$.
However, when $D_i$ is non-separating, the minimality of $|D\cap D_i|$ does not necessarily imply the fact that $\Delta$ is essential in $(W_i'\cup_{\overline{h}_i} V_i)\setminus t_i^{\ast}$.
In case when $\Delta$ is essential in $(W_i'\cup_{\overline{h}_i} V_i)\setminus t_i^{\ast}$, the arguments in the proof of Proposition~\ref{prop-b>1-1} completely work and we can see that (B1) and (B2) hold.
%Hence, in the remaining of this proof, we show that (A) holds if $\Delta$ is inessential in $(W_i'\cup_{\overline{h}_i} V_i)\setminus t_i^{\ast}$.
Hence, in the remainder of this proof, we assume that $\Delta$ is inessential in $(W_i'\cup_{\overline{h}_i} V_i)\setminus t_i^{\ast}$.
%
%Assume that $\Delta$ is inessential in $(W_i'\cup_{\overline{h}_i} V_i)\setminus t_i^{\ast}$.
Let $D_i^0:=D_i\times\{0\}$ and $D_i^1:=D_i\times\{1\}$.
%If $D\cap D_i=\emptyset$ and $\Delta(=D)$ is essential in $(W_i'\cup_{\overline{h}_i} V_i)\setminus t_i^{\ast}$, it is done.
%Hence, we need to consider the two cases, in one of which $D\cap D_i=\emptyset$ and $\Delta(=D)$ is inessential in $(W_i'\cup_{\overline{h}_i} V_i)\setminus t_i^{\ast}$, and in the other $D\cap D_i\ne\emptyset$.

\begin{claim}\label{claim_D_D_i}
$D\cap D_i=\emptyset$.
\end{claim}

\begin{proof}
Assume, for a contradiction, that $D\cap D_i\ne\emptyset$.
We may assume that $D\cap (D_i\times I)$ consists of product disks in $D_i\times I$.
We may also assume that $\Delta\cap D_i^0\ne \emptyset$ and $\Delta\cap D_i^1=\emptyset$ without loss of generality.
Since $\Delta$ is inessential in $(W_i'\cup_{\overline{h}_i} V_i)\setminus t_i^{\ast}$, $\partial\Delta$ is inessential in $(F_i\setminus s_i)\cup D_i^0\cup D_i^1$.
Note that the arc $\partial \Delta\cap F_i$ cuts off an annulus $A$ with at most one point of $s_i\cap F_i$.
However, since $\Delta\cap t_i^{\ast}=\emptyset$, each component of $((F_i\setminus s_i)\cup D_i^0\cup D_i^1)\setminus \partial\Delta$ contains even number of points of $s_i\cap F_i$, and this shows that $A$ does not contain a point of $s_i\cap F_i$.
Then by the minimality of $|D\cap D_i|$, we see that each component of $\partial D\cap F_i$ intersecting $\partial D_i^1$ is an essential arc in $A$, hence the other endpoint is on $\partial D_i^0$.
This implies that $|\partial D\cap \partial D_i^0|\ge |\partial D\cap \partial D_i^1|+2$ since the both endpoints of the arc $\partial\Delta\cap F_i$ are contained in $\partial D_i^0$.
However, by the assumption that $D\cap (D_i\times I)$ consists of product disks, we have $|\partial D\cap \partial D_i^0|=|\partial D\cap \partial D_i^1|$, a contradiction.
This completes the proof of Claim~\ref{claim_D_D_i}.
%Note that $\partial\Delta\cap F_i$ is an essential arc in $F_i$ by the minimality of $|D\cap N(D_i)|$.
%These imply that $\partial \Delta\cap F_i$ cuts off an annulus $A$ from $F_i$ such that $\partial D_i^1$ is a component of $\partial A$.
%Then each component of $\partial D\cap F_i$ intersecting $\partial D_i^1$ is an essential arc in $A$ and the other endpoint is on $\partial D_i^0$. 
%This implies that $|\partial D\cap \partial D_i^0|\ge |\partial D\cap \partial D_i^1|+2$ since the arc $\partial \Delta\cap F_i$ has the both endpoints on $\partial D_i^0$.
%Let $n_0:=|\partial E_i\cap \partial D_i^0|$ and $n_1:=|\partial E_i\cap \partial D_i^1|$.
%Then we have $n_0\ge n_1+2$ by the above observation.
%This is a contradiction since $|\partial D\cap \partial D_i^0|=|\partial D\cap \partial D_i^1|$ by the assumption at the beginning of this paragraph.
%Hence, $\Delta$ is essential in $(W_i'\cup_{\overline{h}_i} V_i)\setminus t_i^{\ast}$.
%
\end{proof}

By Claim~\ref{claim_D_D_i}, we have $D=\Delta$.
Recall that $D$ is inessential in $(W_i'\cup_{\overline{h}_i} V_i)\setminus t_i^{\ast}$.
This implies that $\partial D$ bounds a disk in $(F_i\setminus s_i)\cup D_i^0\cup D_i^1$ with at most one point of $s_i\cap F_i$.
However, since $D\cap t_i^{\ast}=\emptyset$, each component $((F_i\setminus s_i)\cup D_i^0\cup D_i^1)\setminus \partial D$ contains even number of points of $s_i\cap F_i$, and this shows that the disk $D^{\ast}$ in $(F_i\setminus s_i)\cup D_i^0\cup D_i^1$ bounded by $\partial D$ does not contain a point of $s_i\cap F_i$.
%On the other hand, $\partial D_i$ is essential in $F_i\setminus s_i$, 
%These show that $D^{\ast}$ contain at least one of $D_i^0$, $D_i^1$.
Then $D^{\ast}$ must contain at least one of $D_i^0$, $D_i^1$, since, otherwise, $\partial D$ bounds the disk $D^{\ast} (\subset F_i\setminus s_i\subset \partial_+ W_i)$, which contradicts the assumption that $D$ is essential in $V^{\ast}\setminus t^{\ast}$.
However, since $D\ne D_i$ by the assumption of Proposition~\ref{prop-g>1-1}, we see that both $D_i^0$ and $D_i^1$ are contained in $D^{\ast}$, and this shows that (A) holds.

This completes the proof of Proposition~\ref{prop-g>1-1}.
%First assume that $D\cap D_i=\emptyset$. %and $\Delta(=D)$ is inessential in $(W_i'\cup_{\overline{h}_i} V_i)\setminus t_i^{\ast}$.
%Then we can see that (A) holds as follows.
%If $D=D_i$, then (A) holds. Hence, next assume that $D\ne D_i$.
%Then Since $E_i$ is an essential disk in $V_i^{\ast}\setminus t_i^{\ast}$ disjoint from $D_i$, not isotopic to $D_i$ and inessential in $(W_i'\cup_{\overline{h}_i} V_i)\setminus t_i^{\ast}$, 
%Note that $\partial D$ is an essential simple closed curve on $F_i$ and is inessential on $F_i\cup D_i^0\cup D_i^1$.
%Hence, either $D$ can be obtained by a band-sum of $D_i^0$ and $D_i^1$ along an arc (on $F_i$) or $D$ bounds a disk in $F_i\cup D_i^0\cup D_i^1$ which contains exactly one of $D_i^0$ and $D_i^1$ and exactly one point of $s_i\cap F_i$.
%However, the latter case cannot occur, since $D\cap t_i^{\ast}=\emptyset$ and hence the both components of $(F_i\cup D_i^0\cup D_i^1)\setminus \partial D$ contains even number of the points $s_i\cap F_i$.
%Hence, (A) holds.
\end{proof}

%\begin{remark}
%{\rm
%In Proposition~\ref{prop-b>1-1}, we assumed the minimality of $|D\cap D_i|$.
%Note that, in Proposition~\ref{prop-g>1-1}, we assume that $\Delta$ is essential for the condition.
%This is because the minimality of $|D\cap D_i|$ does not necessarily imply the fact that $\Delta$ is an essential disk in $(W_i'\cup_{\overline{h}_i} V_i)\setminus t_i^{\ast}$ when $D_i$ is non-separating.
%}
%\end{remark}

%
%The following proposition follows from the proof of Proposition~\ref{prop-g>1}.
%
%
%\begin{proposition}\label{prop-g>1-1}
%Let $\alpha$ be an element of $\mathcal{C}^0(\partial_+ W_i\setminus s_i)$ such that $\alpha$ is non-separating,  $\alpha\cap \partial D_i=\emptyset$, and $\alpha\cup\partial D_i$ is separating in $\partial_+W_i\setminus s_i$, and $D$ is an essential disk in $V_i^{\ast}\setminus t_i^{\ast}$ such that $D\cap D_i\ne\emptyset$.
%Let $\Delta$ be the closure of one of the components of $D\setminus D_i$.
%Then the following hold.
%\begin{itemize}
%\item[{\rm (1)}] If $\alpha\cap\Delta=\emptyset$, then $d_{\partial_- W_i\setminus s_i}(\Phi_i(\alpha),h_i(\mathcal{D}^0(V_i\setminus t_i)))\le 1$.
%\item[{\rm (2)}] If $|\alpha\cap\Delta|=1$, then $d_{\partial_- W_i\setminus s_i}(\Phi_i(\alpha),h_i(\mathcal{D}^0(V_i\setminus t_i)))\le 2$. 
%\end{itemize}
%\end{proposition}

\part{Proof of Theorem \ref{thm-1} when $n\ge 2$}\label{part2}

\section{Proof of Theorem \ref{thm-1} when $n\ge 3$ and $b\ge 2$}\label{sec-proof-1}

In this section, we give a proof of Theorem~\ref{thm-1} for the case when $n\ge 3$ and $b\ge 2$. 
(Note that $(g,b)\ne (0,2)$ by the assumption of Theorem~\ref{thm-1}.)
%Recall that $c$ is an integer satisfying $1\le c\le b$ which corresponds to the number of the components of the target link.
%We prove there exists a strongly keen $(g,b)$-splitting of a link with distance $n$ for any given integer $n\,(\ge 3)$.

%
Let $F$ be a closed orientable surface of genus $g$ and let $P$ be the union of $2b$ points on $F$.
Let $[\alpha_0,\alpha_1,\dots,\alpha_{n-1}]$ be a geodesic in $\mathcal{C}(F\setminus P)$ constructed as in Proposition \ref{prop-unique-geodesic}.
By Remark \ref{rem-sec3}, we may assume that 
\begin{equation}\label{eqn-0ton-2}
{\rm diam}_{X_1}(\pi_{X_1}(\alpha_0),\pi_{X_1}(\alpha_{n-1}))>2n+6
\end{equation} 
holds, where $X_1$ is the subsurface of $F\setminus P$ associated with $\alpha_1$.
Let $\alpha_n'$ be a simple closed curve in $X_{n-1}$ that cuts off a twice-punctured disk from $X_{n-1}$, where $X_{n-1}$ is the subsurface of $F\setminus P$ associated with $\alpha_{n-1}$.
By Lemma \ref{lem-2-3-1}, there exists a homeomorphism $h:F\setminus P\rightarrow F\setminus P$ such that 
$h(\alpha_{n-1})=\alpha_{n-1}$ and ${\rm diam}_{X_{n-1}}(\pi_{X_{n-1}}(\alpha_0),\pi_{X_{n-1}}(h(\alpha_{n}')))>2n+16$.
Let $\alpha_{n}=h(\alpha_n')$.
Then, by Propositions \ref{prop-extending-geodesic} and \ref{prop-unique-geodesic}, $[\alpha_0,\alpha_1,\dots,\alpha_n]$ is the unique geodesic in $\mathcal{C}(F\setminus P)$ connecting $\alpha_0$ and $\alpha_n$.
Moreover, every $\alpha_i$ cuts off a twice-punctured disk from $F\setminus P$, and 
\begin{equation}\label{eqn-n-2ton}
{\rm diam}_{X_{n-1}}(\pi_{X_{n-1}}(\alpha_{0}),\pi_{X_{n-1}}(\alpha_{n}))>2n+16
\end{equation} 
holds.

For $i=1,2$, let $V_i^{\ast,0}$, $t_i^{\ast,0}$, $V_i$, $t_i$, $W_i$, $s_i$, $D_i$, %$W_i^1$, $s_i^1$, $W_i^2$, $s_i^2$, 
%$\partial_- W_i$, $\partial_+ W_i$, 
$F_i$, $\Phi_i$ be as in Subsection~\ref{sec-b>1}.
Identify $(\partial_+W_1, s_1\cap\,\partial_+W_1)$ and $(\partial_+W_2, s_2\cap\,\partial_+W_2)$ with $(F,P)$ so that $\partial D_1=\alpha_0$ and $\partial D_2=\alpha_n$. % and that $t_1^{\ast,0}\cup t_2^{\ast,0}$ consists of $c$ components.
%We identify $(\partial_+ W_1,s_1\cap\,\partial_+ W_1)$ and $(\partial_+ W_2,s_2\cap\,\partial_+ W_2)$ by $f$.
%Let $\mathcal{D}(W_i)$ $(i=1,2)$ denote the subset of $\mathcal{C}^0(\partial_+ W_i)$ consisting of the vertices with representatives bounding disks in $W_i\setminus s_i$.
%It can be easily seen that $\mathcal{D}(W_i)$ consists of a single vertex $\partial D_i$, and hence we have $d_F(\mathcal{D}(W_1),\mathcal{D}(W_2))=n$.
%
%Let $V_1$ and $t_1$ be as in Subsection~\ref{sec-b>1}. 
By Proposition~\ref{prop-b>1-hi}, there is a homeomorphism $h_1:\partial V_1\setminus t_1\rightarrow \partial_- W_1\setminus s_1$ such that
\begin{equation}\label{eqn-v1-2}
d_{\partial_- W_1\setminus s_1}(\Phi_1(\alpha_1),h_1(\mathcal{D}^0(V_1\setminus t_1)))>2.
\end{equation}
%and that $\overline{h}_1|_{\partial V_1\cap t_1}$ corresponds to ${\rm id}|_{\partial V_1\cap t_1}$, where 
Let $\overline{h}_1:(\partial V_1,\partial t_1)\rightarrow (\partial_- W_1,s_1\cap \partial_- W_1)$ be the homeomorphism of pairs naturally induced from $h_1$. 
Let $(V_1^{\ast}, t_1^{\ast}):=(W_1,s_1)\cup_{\overline{h}_1}(V_1,t_1)$.

\begin{claim}\label{claim-1}
$\alpha_1$ intersects every element of $\mathcal{D}^0(V_1^{\ast}\setminus t_1^{\ast})\setminus \{\alpha_0\,(=\partial D_1)\}$, that is, $[\alpha_0,\alpha_1]$ is the unique geodesic realizing the distance $d_{F\setminus P}(\mathcal{D}^0(V_1^{\ast}\setminus t_1^{\ast}), \{\alpha_1\})=1$. 
\end{claim}

\begin{proof}
Assume on the contrary that there exists an element $\beta$ of $\mathcal{D}^0(V_1^{\ast}\setminus t_1^{\ast})\setminus \{\alpha_0\}$ such that $\beta\cap \alpha_1=\emptyset$.
%Let $D_{\beta}$ be a disk in $V_1^{\ast}\setminus t_1^{\ast}$ bounded by $\beta$.
%We may assume that $|D_{\beta}\cap D_1|$ is minimal.
%(($|D_{\beta}\cap D_1|=0$))
%By using outermost disk arguments, we see that $D_{\beta}\cap D_1$ has no loop components. 
%Let $\Delta$ be a component of $D_{\beta}\setminus D_1$ that is outermost in $D_{\beta}$.
%Note that $D_1$ cuts $(V_1^{\ast},t_1^{\ast})$ into $(W_1^1,s_1^1)\cup_{h_1} (V_1,t_1)$ and $(W_1^2,s_1^2)$.
%Note also that there is no essential disk in $W_1^2\setminus s_1^2$ since $W_1^2$ is a $3$-ball and $s_1^2$ is an arc parallel to $\partial W_1^2$.
%By the minimality of $|D_{\beta}\cap D_1|$, we see that $\Delta$ must be an essential disk in $(W_1^1\cup_{\overline{h}_1} V_1)\setminus t_1^{\ast}$. 
Then, by Proposition~\ref{prop-b>1} (1), we have 
%\begin{eqnarray*}
%\begin{array}{rcl}
%d_{\partial_- W_1\setminus s_1}(\Phi_1(\alpha_1),h_1(\mathcal{D}^0(V_1\setminus t_1)))&\le& {\rm diam}_{\partial_- W_1\setminus s_1}(\Phi_1(\alpha_1),\Phi_1(a))\\&&+ d_{\partial_- W_1\setminus s_1}(\Phi_1(a),h_1(\mathcal{D}^0(V_1\setminus t_1)))\\
%&\le &2+0=2,
%\end{array}
%\end{eqnarray*}
$$
d_{\partial_- W_1\setminus s_1}(\Phi_1(\alpha_1),h_1(\mathcal{D}^0(V_1\setminus t_1)))\le 1,
$$
a contradiction to the inequality (\ref{eqn-v1-2}).
\end{proof}

\begin{claim}\label{claim-1-1}
For any element $a\in\mathcal{D}^0(V_1^{\ast}\setminus t_1^{\ast})\setminus \{\alpha_0\}$, $a\cap \alpha_1$ consists of at least $4$ points.
\end{claim}

\begin{proof}
Assume on the contrary that $|a\cap \alpha_1|<4$.
Then we have $|a\cap \alpha_1|\le 2$, since $\alpha_1$ cuts off a twice-punctured disk from $F\setminus P$ and hence is separating in $F$.
By Proposition~\ref{prop-b>1} (2), we have
$$
d_{\partial_- W_1\setminus s_1}(\Phi_1(\alpha_1),h_1(\mathcal{D}^0(V_1\setminus t_1)))\le 2,
$$
a contradiction to the inequality (\ref{eqn-v1-2}).
\end{proof}

%The following claim can be proved by using arguments similar to that for the proof of \cite[Claim 4.2]{IJK3}.
Then we have:

\begin{claim}\label{claim-1-2}
For any element $a\in\mathcal{D}^0(V_1^{\ast}\setminus t_1^{\ast})$, we have $\pi_{X_1}(a)\ne\emptyset$ and 
${\rm diam}_{X_1}(\{\alpha_0\}, \pi_{X_1}(a))\le 4$.
\end{claim}

\begin{proof}
Note that by Claim \ref{claim-1}, we have $\pi_{X_1}(a)\ne\emptyset$.

If $a=\alpha_0$ or $a\cap\alpha_0=\emptyset$, that is, $d_{F\setminus P}(\alpha_0,a)\le 1$, then we have ${\rm diam}_{X_1}(\{\alpha_0\},\pi_{X_1}(a))\le 2$ by Lemma \ref{subsurface distance}.

Next, we suppose that $a\ne \alpha_0$ and $a\cap \alpha_0\ne\emptyset$.
Let $D_a$ be a disk in $V_1^{\ast}\setminus t_1^{\ast}$ bounded by $a$. 
Recall that $\alpha_0$ bounds the disk $D_1$.
We may assume that $|D_a\cap D_1|$ is minimal (hence, each component of $D_a\cap D_1$ is an arc).
Let $\Delta$ be the closure of a component of $D_a\setminus D_1$ that is outermost in $D_a$.
%If $\Delta\cap\alpha_1=\emptyset$, then we can lead to a contradiction by arguments in the proof of Claim \ref{claim-1}. 
%Hence, $\Delta\cap \alpha_1\ne\emptyset.$
Let $\Delta'$ be the disk obtained from $\Delta$ as in the proof of Proposition~\ref{prop-b>1}.
Then we see by the proof that $\partial \Delta'\in \mathcal{D}^0(V_1^{\ast}\setminus t_1^{\ast})$.
Further we may suppose that $\Delta'\cap \alpha_1=\Delta\cap \alpha_1$.
By Claim~\ref{claim-1-1}, we see that $\Delta\cap \alpha_1(=\Delta'\cap \alpha_1)$ consists of at least 4 points.
Note that $\Delta\cap F_1$ is an arc properly embedded in $F_1$.
Hence, there is a subarc $\gamma$ of $\partial \Delta\setminus D_1$ such that $\gamma\cap N(\alpha_1)=\partial\gamma$ and $\gamma$ is an arc properly embedded in $X_1$.
These imply that $d_{\mathcal{AC}(X_1)}(\alpha_0, \gamma)=1$.
%
%Since $\alpha_1$ is separating in $F\setminus \alpha_0$ by the construction, there exists a component $\gamma$ of $\partial\Delta\setminus(D_1\cup \alpha_1)$ such that $\partial \gamma\subset \alpha_1$.
%Note that $\gamma$ is disjoint from $\alpha_0$, that is, $d_{\mathcal{AC}(X_1)}(\alpha_0, \gamma)=1$, since $\alpha_0\cap\Delta=\emptyset$ and $\gamma$ is a subarc of $\partial\Delta$.
%Since $\gamma\in\pi_A(a)\,(\subset \mathcal{AC}^0(X_1))$, we have $d_{\mathcal{AC}(X_1)}(\alpha_0, \pi_A(a))\le 1$.
Note that $\gamma\in \pi_{AC}(a)$, where $\pi_{AC}$ is the map from $\mathcal{C}^{0}(F\setminus P)$ to $\mathcal{P}(\mathcal{AC}^{0}(X_1))$ defined as in Subsection~\ref{sec-subsurface}. 
Hence, we have
\begin{eqnarray*}
\begin{array}{rcl}
{\rm diam}_{\mathcal{AC}(X_1)}(\{\alpha_0\}, \pi_{AC}(a))&\le& d_{\mathcal{AC}(X_1)}(\alpha_0, \gamma)+{\rm diam}_{\mathcal{AC}(X_1)}(\pi_{AC}(a))\\
&\le& 1+1=2.
\end{array}
\end{eqnarray*}
Hence, by Lemma \ref{lem_mm},
%\ref{subsurface distance}, 
we have ${\rm diam}_{X_1}(\{\alpha_0\}, \pi_{X_1}(a))\le 4$.
\end{proof}

\begin{claim}\label{claim-1-3}
$d_{F\setminus P}(\mathcal{D}^0(V_1^{\ast}\setminus t_1^{\ast}),\{\alpha_n\})=n$. Moreover, $[\alpha_0, \alpha_1,\dots,\alpha_n]$ is the unique geodesic realizing the distance.
\end{claim}

\begin{proof}
Since there is a geodesic $[\alpha_0,\alpha_1,\dots,\alpha_n]$, we see that 
$d_{F\setminus P}(\mathcal{D}^0(V_1^{\ast}\setminus t_1^{\ast}),\{\alpha_n\})\le n$.
Let $[\beta_0,\beta_1,\dots,\beta_m]$ be a geodesic in $\mathcal{C}(F\setminus P)$ such that $\beta_0\in\mathcal{D}^0(V_1^{\ast}\setminus t_1^{\ast})$, $\beta_m=\alpha_n$ and $m\le n$.

%Assume on the contrary that
%$d_{F\setminus P}(\mathcal{D}^0(V_1^{\ast}\setminus t_1^{\ast}),\{\alpha_n\})<n$.
%Then, there is a geodesic $[\beta_0,\beta_1,\dots,\beta_m]$ in $\mathcal{C}(F\setminus P)$ such that $\beta_0\in\mathcal{D}^0(V_1^{\ast}\setminus t_1^{\ast})$, $\beta_m=\alpha_n$ and $m<n$.

We claim that there exists $i\in\{0,1,\dots,m\}$ such that $\beta_i=\alpha_1$.
In fact, if $\beta_i\ne \alpha_1$ for every $i\in\{0,1,\dots,m\}$, then every $\beta_i$ cuts $X_1$. 
Then, by Lemma~\ref{subsurface distance}, we have
\begin{equation*}%\label{beta}
{\rm diam}_{X_1}(\pi_{X_1}(\beta_0), \pi_{X_1}(\beta_m))\le 2m\le 2n.
\end{equation*}
Similarly we have
$$
{\rm diam}_{X_1}(\pi_{X_1}(\alpha_n), \pi_{X_1}(\alpha_{n-1}))\le 2.
$$
On the other hand, by Claim~\ref{claim-1-2}, we have 
\begin{equation*}%\label{beta2}
{\rm diam}_{X_1}(\{\alpha_0\},\pi_{X_1}(\beta_0))\le 4.
\end{equation*}
These show
\begin{eqnarray*}
\begin{array}{rcl}
{\rm diam}_{X_1}(\pi_{X_1}(\alpha_0), \pi_{X_1}(\alpha_{n-1}))%&\le& {\rm diam}_{X_1}(\alpha_0, \pi_{X_1}(\beta_p))+{\rm diam}_{X_1}(\pi_{X_1}(\alpha_n)), \pi_{X_1}(\alpha_{n-1}))\\
&\le&{\rm diam}_{X_1}(\{\alpha_0\}, \pi_{X_1}(\beta_0))+{\rm diam}_{X_1}(\pi_{X_1}(\beta_0), \pi_{X_1}(\beta_m))\\&&+{\rm diam}_{X_1}(\pi_{X_1}(\alpha_n), \pi_{X_1}(\alpha_{n-1}))\\
&\le& 4+2n+2=2n+6,
\end{array}
\end{eqnarray*}
a contradiction to the inequality (\ref{eqn-0ton-2}).
Hence, there exists $i\in\{0,1,\dots,m\}$ such that $\beta_i=\alpha_1$.

We have $i\ne 0$ by Claim~\ref{claim-1}.
If $i\ge 2$, then 
$$
n=d_{F\setminus P}(\alpha_0, \alpha_{n})\le d_{F\setminus P}(\alpha_0, \alpha_{1})+d_{F\setminus P}(\beta_i, \beta_m)\le 1+(m-i)\le 1+(n-2),
$$
a contradiction.
Hence, $i=1$, that is, $\beta_1=\alpha_1$.
Then, by Claim~\ref{claim-1}, we have $\beta_0=\alpha_0$.
Since $[\alpha_1,\dots,\alpha_n]$ is the unique geodesic connecting $\alpha_1$ and $\alpha_n$ (because $[\alpha_0, \alpha_1,\dots,\alpha_n]$ is a unique geodesic), this imiplies $m=n$ and $[\beta_1,\dots,\beta_m]=[\alpha_1,\dots,\alpha_n]$.
%On the other hand, $[\beta_1,\dots,\beta_m]=[\alpha_1,\dots,\alpha_n]$ since $[\alpha_1,\dots,\alpha_n]$ is the unique geodesic connecting $\alpha_1$ and $\alpha_n$.
%Hence, we have $[\beta_0,\beta_1,\dots,\beta_m]=[\alpha_0,\alpha_1,\dots,\alpha_n]$.
%However, this shows that $[\beta_0(=\alpha_0),\beta_1,\dots,\beta_m(=\alpha_n)]$ is a path of length $m(<n)$ connecting $\alpha_0$ and $\alpha_n$, which contradicts the fact that $[\alpha_0,\alpha_1,\dots,\alpha_n]$ is a geodesic.
%
\end{proof}

%By using Claims \ref{claim-1}, \ref{claim-1-2} and arguments similar to those for \cite[Lemma 4.3]{IJK3}, we have $d_F(\mathcal{D}(V_1^{\ast}),\partial D_2)=n$.

%By \cite[Claims 1 and 2]{IJK2} together with the inequality (\ref{eqn-v1-1}), we have $d_F(\mathcal{D}(V_1^{\ast}),\partial D_2)=n$. and ${\rm diam}_{F_2}(\pi_{F_2}(\mathcal{D}(V_1^{\ast})))\le 12$.

%Let $V_2$ and $t_2$ be as in Subsection~\ref{sec-b>1}. 
By Proposition~\ref{prop-b>1-hi}, there is a homeomorphism $h_2:\partial V_2\setminus t_2\rightarrow \partial_- W_2\setminus s_2$ such that
\begin{equation}\label{eqn-v2-2}
d_{\partial_- W_2\setminus s_2}(\Phi_2(\alpha_{n-1}),h_2(\mathcal{D}^0(V_2\setminus t_2)))>2.
\end{equation}
%and that $\overline{h}_2|_{\partial V_2\cap t_2}$ corresponds to ${\rm id}_{\partial V_2\cap t_2}$, where
Let $\overline{h}_2:(\partial V_2,\partial t_2)\rightarrow (\partial_- W_2,s_2\cap \partial_- W_2)$ be the homeomorphism of pairs naturally induced from $h_2$.
Let $(V_2^{\ast}, t_2^{\ast}):=(W_2,s_2)\cup_{\overline{h}_2}(V_2,t_2)$.
%Then $V_2^{\ast}$ is a genus-$g$ handlebody and $t_2^{\ast}$ is the union of $b$ arcs parallel to $\partial V_2^{\ast}$.
Then $(V_1^{\ast}, t_1^{\ast})\cup_{(F,P)} (V_2^{\ast}, t_2^{\ast})$ is a $(g,b)$-splitting of a link.
%Recall that $\mathcal{D}(V_2^{\ast}\setminus t_2^{\ast})$ is the disk complex of $V_2^{\ast}\setminus t_2^{\ast}$.

The following two claims can be proved by arguments similar to those for Claims \ref{claim-1} and \ref{claim-1-2}, and the proofs are left to the reader.

\begin{claim}\label{claim-2}
$\alpha_{n-1}$ intersects every element of $\mathcal{D}^0(V_2^{\ast}\setminus t_2^{\ast})\setminus \{\alpha_n\,(=\partial D_2)\}$, that is, $[\alpha_{n-1},\alpha_n]$ is the unique geodesic realizing the distance $d_{F\setminus P}(\{\alpha_{n-1}\},\mathcal{D}^0(V_2^{\ast}\setminus t_2^{\ast}))=1$.
\end{claim}

%Let $X_1$ and $X_{n-1}$ be the subsurfaces of $F$ associated with $\alpha_1$ and $\alpha_{n-1}$, respectively.
%Then t
%The following claim can be proved by using arguments similar to that for the proof of \cite[Claim 4.2]{IJK3}.

\begin{claim}\label{claim-3}
For any element $b\in\mathcal{D}^0(V_2^{\ast}\setminus t_2^{\ast})$, we have $\pi_{X_{n-1}}(b)\ne\emptyset$ and ${\rm diam}_{X_{n-1}}(\{\alpha_n\}, \pi_{X_{n-1}}(b))\le 4$.
\end{claim}

%Then, these claims together with arguments similar to those for Claim~\ref{claim-1-3} imply the following.

%\begin{claim}\label{claim-3-3}
%$d_{F\setminus P}(\mathcal{D}^0(V_2^{\ast}\setminus t_2^{\ast}),\{\partial D_1\})=n$.
%\end{claim}

%Also, the following claim can be proved by using arguments similar to that for the proof of \cite[Claim 4.8]{IJK3} and \cite[Proposition 5.1]{IJK2}.

Since $\alpha_{n-1}$ cuts off a twice-punctured disk, we have the next claim by \cite[Proposition 5.1]{IJK2}.

\begin{claim}\label{claim-4}
%${\rm diam}_{X_1}(\pi_{X_1}(\mathcal{D}^0(V_2^{\ast}\setminus t_2^{\ast})))\le 12$ and 
${\rm diam}_{X_{n-1}}(\pi_{X_{n-1}}(\mathcal{D}^0(V_1^{\ast}\setminus t_1^{\ast})))\le 12$.
\end{claim}

%\begin{proof}
%s
%\end{proof}

Then we have:

\begin{claim}\label{claim-4-1}
$d_{F\setminus P}(\mathcal{D}^0(V_1^{\ast}\setminus t_1^{\ast}),\mathcal{D}^0(V_2^{\ast}\setminus t_2^{\ast}))=n$, and the bridge splitting $(V_1^{\ast}, t_1^{\ast})\cup_{(F,P)} (V_2^{\ast}, t_2^{\ast})$ is strongly keen.
\end{claim}

\begin{proof}
%
%Since there is a geodesic $[\alpha_0,\alpha_1,\dots,\alpha_n]$, we see that 
%$d_{F\setminus P}(\mathcal{D}(V_1^{\ast}\setminus t_1^{\ast}),\partial D_2)\le n$.
%
%Assume on the contrary that
%$d_{F\setminus P}(\mathcal{D}^0(V_1^{\ast}\setminus t_1^{\ast}),\mathcal{D}^0(V_2^{\ast}\setminus t_2^{\ast}))<n$.
%Then, there is a geodesic $[\beta_0,\beta_1,\dots,\beta_p]$ in $\mathcal{C}(F\setminus P)$ such that $\beta_0\in\mathcal{D}^0(V_1^{\ast}\setminus t_1^{\ast})$, $\beta_p\in\mathcal{D}^0(V_2^{\ast}\setminus t_2^{\ast})$ and $p<n$.
Since there is a geodesic $[\alpha_0,\alpha_1,\dots,\alpha_n]$, we have $d_{F\setminus P}(\mathcal{D}^0(V_1^{\ast}\setminus t_1^{\ast}),\mathcal{D}^0(V_2^{\ast}\setminus t_2^{\ast}))\le n$.
Let $[\beta_0,\beta_1,\dots,\beta_m]$ be a geodesic in $\mathcal{C}(F\setminus P)$ such that $\beta_0\in\mathcal{D}^0(V_1^{\ast}\setminus t_1^{\ast})$, $\beta_m\in\mathcal{D}^0(V_2^{\ast}\setminus t_2^{\ast})$ and $m\le n$.

We claim that there exists $i\in\{0,1,\dots,m\}$ such that $\beta_i=\alpha_{n-1}$. 
In fact, if $\beta_i\ne \alpha_{n-1}$ for every $i\in\{0,1,\dots,m\}$, then every $\beta_i$ cuts $X_{n-1}$.
Then by Lemma \ref{subsurface distance}, we have ${\rm diam}_{X_{n-1}}(\pi_{X_{n-1}}(\beta_0),\pi_{X_{n-1}}(\beta_m))\le 2m\le 2n$.
This together with Claims~\ref{claim-3} and \ref{claim-4} implies that
\begin{eqnarray*}
\begin{array}{rcl}
{\rm diam}_{X_{n-1}}(\pi_{X_{n-1}}(\alpha_{0}),\pi_{X_{n-1}}(\alpha_{n}))&\le& 
{\rm diam}_{X_{n-1}}(\pi_{X_{n-1}}(\alpha_0),\pi_{X_{n-1}}(\beta_0))\\
&&+{\rm diam}_{X_{n-1}}(\pi_{X_{n-1}}(\beta_0),\pi_{X_{n-1}}(\beta_m))\\
&&+{\rm diam}_{X_{n-1}}(\pi_{X_{n-1}}(\beta_m),\pi_{X_{n-1}}(\alpha_{n}))\\
&\le&12+2n+4\\&=&2n+16,
\end{array}
\end{eqnarray*}
which contradicts the inequality (\ref{eqn-n-2ton}).

By Claim~\ref{claim-2}, the fact $\beta_i=\alpha_{n-1}$ implies that $[\beta_i,\dots,\beta_m]=[\alpha_{n-1},\alpha_n]$, and hence, both $[\beta_0,\beta_1,\dots,\beta_m]$ and $[\alpha_0,\alpha_1,\dots,\alpha_n]$ are geodesics connecting $\mathcal{D}^0(V_1^{\ast}\setminus t_1^{\ast})$ and $\alpha_n$.
Then, by Claim~\ref{claim-1-3}, we see that $m=n$, that is, $d_{F\setminus P}(\mathcal{D}^0(V_1^{\ast}\setminus t_1^{\ast}),\mathcal{D}^0(V_2^{\ast}\setminus t_2^{\ast}))=n$, 
and that $[\beta_0,\beta_1,\dots,\beta_n]=[\alpha_0,\alpha_1,\dots,\alpha_n]$.
%Then, we have $[\beta_0,\dots,\beta_i]=[\alpha_0,\dots,\alpha_{n-1}]$, since $[\alpha_0,\dots,\alpha_{n-1}]$ is the unique geodesic realizing the distance $d_{F\setminus P}(\mathcal{D}^0(V_1^{\ast}\setminus t_1^{\ast}),\{\alpha_{n-1}\})=n-1$ by Claim~\ref{claim-1-3}.
%This also implies that $i=n-1$.
%Note that $\alpha_{n-1}\not\in\mathcal{D}^0(V_2^{\ast}\setminus t_2^{\ast})$ by Claim \ref{claim-2}, and hence $\alpha_{n-1}(=\beta_i)\ne\beta_m$.
%Hence, we have $i(=n-1)\le m-1.$
%This implies $m=n$, that is, $d_{F\setminus P}(\mathcal{D}^0(V_1^{\ast}\setminus t_1^{\ast}),\mathcal{D}^0(V_2^{\ast}\setminus t_2^{\ast}))=n$.
%
\end{proof}

%By arguments similar to those in the proof of \cite[Lemma 4.9]{IJK3} together with the above claims, we can show that $d_{F\setminus P}(\mathcal{D}(V_1^{\ast}\setminus t_1^{\ast}),\mathcal{D}(V_2^{\ast}\setminus t_2^{\ast}))=n$.

%\begin{claim}
%The bridge splitting $(V_1^{\ast}, t_1^{\ast})\cup_{(F,P)} (V_2^{\ast}, t_2^{\ast})$ is strongly keen.
%\end{claim}

%
%\begin{proof}
%Let $[\beta_0,\beta_1,\dots,\beta_n]$ be any geodesic connecting $\mathcal{D}^0(V_1^{\ast}\setminus t_1^{\ast})$ and $\mathcal{D}^0(V_2^{\ast}\setminus t_2^{\ast})$.
%By the proof of  Claim~\ref{claim-4-1}, we have $[\beta_0,\dots,\beta_{n-1}]=[\alpha_0,\dots,\alpha_{n-1}]$. 
%Also, we have $\beta_n=\alpha_n$ by Claim~\ref{claim-2}.
%These imply that $[\beta_0,\beta_1,\dots,\beta_n]=[\alpha_0,\alpha_1,\dots,\alpha_n]$.
%\end{proof}

This completes the proof of Theorem~\ref{thm-1} for the case when $n\ge 3$ and $b\ge 2$.

\begin{remark}\label{rem}
{\rm 
We remark that there exist keen bridge splittings with distance $n(\ge 4)$ each of which is not strongly keen.

We can construct such examples for the case when $n\ge 5$ as follows.

Firstly, let $[\alpha_3,\alpha_4,\dots,\alpha_{n-1}]$ be a geodesic constructed as in Proposition \ref{prop-unique-geodesic}, and let $[\gamma_0,\gamma_1,\gamma_2]$ be a geodesic as illustrated in Figure \ref{fig-rem}.
That is, $\gamma_0$ and $\gamma_2$ are the boundaries of twice-punctured disks in $F\setminus P$ such that they intersect in $4$ points and that one of the components of $(F\setminus P)\setminus (\gamma_0\cup \gamma_2)$ is a twice-punctured disk whose boundary is parallel to $\gamma_1$.
\begin{figure}[tb]
 \begin{center}
 \includegraphics[width=46mm]{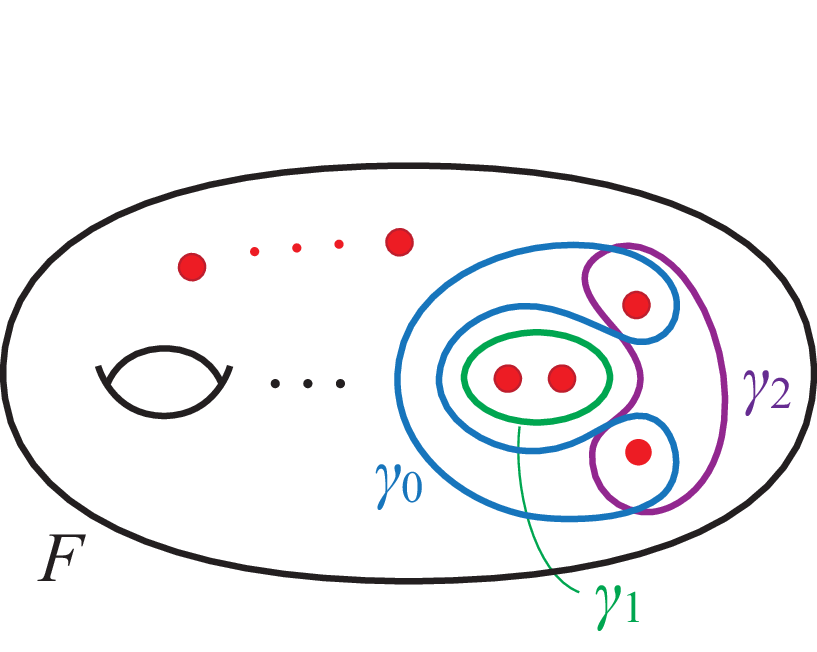}
 \end{center}
 \caption{$\gamma_0$, $\gamma_1$ and $\gamma_2$.}
\label{fig-rem}
\end{figure}
Note that there exists an essential simple closed curve $\gamma_1'(\ne \gamma_1)$ disjoint from $\gamma_0\cup \gamma_2$ since $(g,b)\ne (0,2)$. 
By Lemma \ref{lem-2-3-1}, there exists a homeomorphism $h:F\setminus P\rightarrow F\setminus P$ such that $h(\gamma_2)=\alpha_{3}$ and ${\rm diam}_{X_{3}}(\pi_{X_3}(\alpha_{n-1}),\pi_{X_3}(h(\gamma_0)))>2(n-2)$, where $X_{3}$ is the subsurface of $F\setminus P$ associated with $\alpha_{3}$.
Let $\alpha_{2}:=h(\gamma_1)$ and $\alpha_{1}:=h(\gamma_0)$.
Note that the above inequality ${\rm diam}_{X_{3}}(\pi_{X_3}(\alpha_{n-1}),\pi_{X_3}(h(\gamma_0)))>2(n-2)$ implies that every geodesic connecting $\alpha_1$ and $\alpha_{n-1}$ passes $\alpha_3$.

Secondly, let $\alpha_0'$ be a simple closed curve in $F\setminus P$ that is disjoint from $\alpha_{1}$ and that cuts off a twice-punctured disk from $F\setminus P$.
By Lemma \ref{lem-2-3-1}, there exists a homeomorphism $h':F\setminus P\rightarrow F\setminus P$ such that 
$h'(\alpha_{1})=\alpha_{1}$ and ${\rm diam}_{X_{1}}(\pi_{X_{1}}(\alpha_{n-1}),\pi_{X_{1}}(h'(\alpha_{0}')))>2n+2$,
where $X_{1}$ is the subsurface of $F\setminus P$ associated with $\alpha_{1}$.
Let $\alpha_{0}:=h'(\alpha_0')$.
Then, by using Lemma~\ref{subsurface distance}, it can be seen that $[\alpha_0,\alpha_1,\dots,\alpha_{n-1}]$ is a geodesic, and every geodesic connecting $\alpha_0$ and $\alpha_{n-1}$ passes through $\alpha_{1}$.

Similarly, we can take $\alpha_n$ such that ${\rm diam}_{X_{n-1}}(\pi_{X_{n-1}}(\alpha_0), \pi_{X_{n-1}}(\alpha_n))>2n+16$, where $X_{n-1}$ is the subsurface of $F\setminus P$ associated with $\alpha_{n-1}$. Then it can be seen that $[\alpha_0,\alpha_1,\dots,\alpha_{n-1},\alpha_n]$ is a geodesic, and every geodesic connecting $\alpha_0$ and $\alpha_n$ passes through $\alpha_{n-1}$.
We apply the construction of a $(g,b)$-splitting $(V_1^{\ast},t_1^{\ast})\cup_{(F,P)} (V_2^{\ast},t_2^{\ast})$ in this section, with the above geodesic $[\alpha_0,\alpha_1,\dots,\alpha_{n-1},\alpha_n]$.

Then we claim that the $(g,b)$-splitting $(V_1^{\ast},t_1^{\ast})\cup_{(F,P)} (V_2^{\ast},t_2^{\ast})$ is keen. 
Let $[\beta_0,\beta_1,\dots,\beta_m]$ be a shortest geodesic in $\mathcal{C}(F\setminus P)$ such that $\beta_0\in\mathcal{D}^0(V_1^{\ast}\setminus t_1^{\ast})$ and $\beta_m\in\mathcal{D}^0(V_2^{\ast}\setminus t_2^{\ast})$.
Note that $m\le n$.
%Then, by the arguments in the second and the third paragraphs in the proof of Claim~\ref{claim-1-3}, we can see that $\beta_1=\alpha_1$.
By the arguments in the second paragraph of the proof of Claim~\ref{claim-4-1}, we can see that there exists $i\in\{0,1,\dots,m\}$ such that $\beta_i=\alpha_{n-1}$.
Note that $i\ne m$ since $\beta_m\in\mathcal{D}^0(V_2^{\ast}\setminus t_2^{\ast})$ and $\alpha_{n-1}\not\in\mathcal{D}^0(V_2^{\ast}\setminus t_2^{\ast})$ by Claim~\ref{claim-2}.
In fact, we have $i=m-1$ since, otherwise, $[\beta_0,\beta_1,\dots,\beta_i=\alpha_{n-1},\alpha_n]$ is a geodesic in $\mathcal{C}(F\setminus P)$ connecting $\mathcal{D}^0(V_1^{\ast}\setminus t_1^{\ast})$ and $\mathcal{D}^0(V_2^{\ast}\setminus t_2^{\ast})$, which is shorter than $[\beta_0,\beta_1,\dots,\beta_m]$.
Similarly, we can apply arguments similar to those in the second and the third paragraphs of the proof of Claim~\ref{claim-1-3} to see that there exists $j\in\{0,1,\dots,m-2\}$ such that $\beta_j=\alpha_1$, since otherwise we have
\begin{eqnarray*}
\begin{array}{rcl}
{\rm diam}_{X_1}(\pi_{X_1}(\alpha_0), \pi_{X_1}(\alpha_{n-1}))%&\le& {\rm diam}_{X_1}(\alpha_0, \pi_{X_1}(\beta_p))+{\rm diam}_{X_1}(\pi_{X_1}(\alpha_n)), \pi_{X_1}(\alpha_{n-1}))\\
&\le&{\rm diam}_{X_1}(\{\alpha_0\}, \pi_{X_1}(\beta_0))\\&&+{\rm diam}_{X_1}(\pi_{X_1}(\beta_0), \pi_{X_1}(\beta_{m-1}))\\
&\le& 4+2(n-1)=2n+2,
\end{array}
\end{eqnarray*}
a contradiction.
In fact, we have $j=1$ by Claim~\ref{claim-1} together with the assumption that $[\beta_0,\beta_1,\dots,\beta_m]$ is a shortest geodesic in $\mathcal{C}(F\setminus P)$ connecting $\mathcal{D}^0(V_1^{\ast}\setminus t_1^{\ast})$ and $\mathcal{D}^0(V_2^{\ast}\setminus t_2^{\ast})$.
Then, we have %$$i-1= (m-1)-1=m-2.$$
%On the other hand, we have 
$$(m-1)-1=d_{F\setminus P}(\beta_1,\beta_{m-1})=d_{F\setminus P}(\alpha_1,\alpha_{n-1})=n-2.$$
Hence, we have $m=n$, and $\beta_1=\alpha_1$ and $\beta_{n-1}=\alpha_{n-1}$. %$i=n-1$. %$(j,i)=(1,n-1)$.
By Claims \ref{claim-1} and \ref{claim-2}, we have $\beta_0=\alpha_0$ and $\beta_n=\alpha_n$, and hence the $(g,b)$-splitting is keen.

However, we have another geodesic $[\alpha_0,\alpha_1,h(\gamma_1'),\alpha_{3},\dots,\alpha_n]$ connecting $\alpha_0$ and $\alpha_{n}$, where $\gamma_1'(\ne \gamma_1)\in \mathcal{C}^0(F\setminus P)$ with $\gamma_1'\cap (\gamma_0\cup \gamma_2)=\emptyset$, and hence the $(g,b)$-splitting is not strongly keen.

Examples for the case when $n=4$ are constructed similarly.
The difference in the construction is as follows.
We start with the geodesic $[\alpha_1, \alpha_2,\alpha_3]:=[\gamma_0,\gamma_1,\gamma_2]$.
Then we apply the above argument to obtain $\alpha_0(:=h'(\alpha_0'))$ and $\alpha_4$ with %a homeomorphism $h':F\setminus P\rightarrow F\setminus P$ such that $h'(\alpha_{1})=\alpha_{1}$ and 
${\rm diam}_{X_{1}}(\pi_{X_{1}}(\alpha_{3}),\pi_{X_{1}}(\alpha_{0}))>6$ and ${\rm diam}_{X_{3}}(\pi_{X_{3}}(\alpha_{0}),\pi_{X_{3}}(\alpha_{4}))>24$.
Then $[\alpha_0,\alpha_1,\alpha_2,\alpha_3,\alpha_4]$ is a geodesic, and every geodesic connecting $\alpha_0$ and $\alpha_4$ passes through $\alpha_1$ and $\alpha_3$.
We apply the construction of a $(g,b)$-splitting in this section, with the above geodesic $[\alpha_0,\alpha_1,\alpha_2,\alpha_3,\alpha_4]$.
Then, by using the same arguments as above, we can show that the $(g,b)$-splitting is keen, but not strongly keen.
}
\end{remark}

\section{Proof of Theorem \ref{thm-1} when $n=2$ and $b\ge 2$}\label{sec-proof-2}

In this section, we give a proof of Theorem \ref{thm-1} for the case when $n=2$ and $b\ge 2$.
(Note that $(g,b)\ne (0,2)$ by the assumption of Theorem~\ref{thm-1}.)

Let $F$ be a closed orientable surface of genus $g$ and let $P$ be the union of $2b$ points on $F$.
Let $[\alpha_0,\alpha_1,\alpha_2]$ be a geodesic in $\mathcal{C}(F\setminus P)$ constructed as in Proposition \ref{prop-unique-geodesic}.
By Remark~\ref{rem-sec3}, we may assume 
\begin{equation}\label{eqn-0to2}
d_{X_1}(\alpha_0,\alpha_2)>12
\end{equation}
holds, where $X_1$ is the subsurface of $F\setminus P$ associated with $\alpha_1$.

For $i=1,2$, let $V_i^{\ast,0}$, $t_i^{\ast,0}$, $V_i$, $t_i$, $W_i$, $s_i$, %$W_i^j$, 
$D_i$, %$s_i^j$, 
$F_i$, %$\pi_{F_i\setminus s_i}$, 
%$P_i$, 
$\Phi_i$ be as in Subsection~\ref{sec-b>1}.
Identify $(\partial_+W_1, s_1\cap\,\partial_+W_1)$ and $(\partial_+W_2, s_2\cap\,\partial_+W_2)$ with $(F,P)$ so that $\partial D_1=\alpha_0$ and $\partial D_2=\alpha_2$. %and that $t_1^{\ast,0}\cup t_2^{\ast,0}$ consists of $c$ components.
%Note that $\Phi_1(\alpha_1)$ and $\Phi_2(\alpha_1)$ are essential simple closed curves on $\partial_- W_1\setminus s_1$ and $\partial_- W_2\setminus s_2$, respectively.
By Proposition~\ref{prop-b>1-hi}, there is a homeomorphism $h_i:\partial V_i\setminus t_i\rightarrow \partial_- W_i\setminus s_i$ such that $d_{\partial_-W_i\setminus s_i}(P_i(\alpha_1), h_i(\mathcal{D}(V_i\setminus t_i)))> 2$.
Let $\overline{h}_i:(\partial V_i,\partial t_i)\rightarrow (\partial_- W_i,s_i\cap \partial_- W_i)$ be the homeomorphism of pairs naturally induced from $h_i$.
Let $(V_i^{\ast},t_i^{\ast}):=(W_i,s_i)\cup_{\overline{h}_i}(V_i,t_i)$.
Then $(V_1^{\ast},t_1^{\ast})\cup_{(F,P)}(V_2^{\ast},t_2^{\ast})$ is a $(g,b)$-splitting of a link.
%Recall that $\mathcal{D}(V_i^{\ast}\setminus t_i^{\ast})$ is the disk complex of $V_i^{\ast}\setminus t_i^{\ast}$ for $i=1,2$.
By arguments similar to those for Claims \ref{claim-1}, \ref{claim-1-2}, \ref{claim-2} and \ref{claim-3}, we have the following.
\begin{claim}\label{claim-6}
{\rm (1)} $\alpha_1$ intersects every element of $\mathcal{D}^0(V_1^{\ast}\setminus t_1^{\ast})\setminus\{\alpha_0\}$ and every element of $\mathcal{D}^0(V_2^{\ast}\setminus t_2^{\ast})\setminus\{\alpha_2\}$.

{\rm (2)} For any $a\in\mathcal{D}^0(V_1^{\ast}\setminus t_1^{\ast})$, we have $\pi_{X_1}(a)\ne\emptyset$ and ${\rm diam}_{X_1}(\{\alpha_0\},\pi_{X_1}(a))\le 4$.

{\rm (3)} For any $b\in\mathcal{D}^0(V_2^{\ast}\setminus t_2^{\ast})$, we have $\pi_{X_1}(b)\ne\emptyset$ and ${\rm diam}_{X_1}(\{\alpha_2\},\pi_{X_1}(b))\le 4$.
\end{claim}

\begin{lemma}
$(V_1^{\ast},t_1^{\ast})\cup_{(F,P)}(V_2^{\ast},t_2^{\ast})$ is a strongly keen bridge splitting whose distance is $2$.
\end{lemma}

\begin{proof}
We have $d_{F\setminus P}(\mathcal{D}^0(V_1^{\ast}\setminus t_1^{\ast}),\mathcal{D}^0(V_2^{\ast}\setminus t_2^{\ast}))\le 2$ since $\alpha_0\in\mathcal{D}^0(V_1^{\ast}\setminus t_1^{\ast})$ and $\alpha_2\in\mathcal{D}^0(V_2^{\ast}\setminus t_2^{\ast})$.

Let $[\beta_0,\beta_1,\beta_2]$ be a geodesic in $\mathcal{C}(F\setminus P)$ such that $\beta_0\in\mathcal{D}^0(V_1^{\ast}\setminus t_1^{\ast})$ and $\beta_2\in\mathcal{D}^0(V_2^{\ast}\setminus t_2^{\ast})$.
(Possibly, $\beta_1\in\mathcal{D}^0(V_1^{\ast}\setminus t_1^{\ast})$ or $\beta_1\in\mathcal{D}^0(V_2^{\ast}\setminus t_2^{\ast})$.)
By Claim \ref{claim-6} (1), both $\beta_0$ and $\beta_2$ cut $X_1$.
If $\beta_1$ also cuts $X_1$, then we have ${\rm diam}_{X_1}(\pi_{X_1}(\beta_0),\pi_{X_1}(\beta_2))\le 4$ by Lemma \ref{subsurface distance}, which together with Claim \ref{claim-6} (2) and (3) implies that
\begin{eqnarray*}
\begin{array}{rcl}
d_{X_1}(\alpha_0,\alpha_2)&\le&{\rm diam}_{X_1}(\{\alpha_0\},\pi_{X_1}(\beta_0))+{\rm diam}_{X_1}(\pi_{X_1}(\beta_0), \pi_{X_1}(\beta_2))\\
&&+{\rm diam}_{X_1}(\pi_{X_1}(\beta_2),\{\alpha_2\})\\
&\le&4+4+4=12.
\end{array}
\end{eqnarray*}
This contradicts the inequality (\ref{eqn-0to2}).
Hence, $\beta_1$ misses $X_1$, that is, $\beta_1=\alpha_1$.
By Claim \ref{claim-6} (1), we have $\beta_0=\alpha_0$ and $\beta_2=\alpha_2$, and we obtain the desired result.
\end{proof}

This completes the proof of Theorem \ref{thm-1} for the case when $n=2$ and $b\ge 2$.

\section{Proof of Theorem \ref{thm-1} when $n\ge 2$, $g\ge 2$ and $b=1$}\label{sec-proof-4}

In this section, we give a proof of Theorem \ref{thm-1} for the case when $n\ge 2$, $g\ge 2$ and $b=1$.

Let $F$ be a closed orientable surface of genus $g$ and let $P$ be the union of $2$ points on $F$.
For $i=1,2$, let $V_i^{\ast,0}$, $t_i^{\ast,0}$, $V_i$, $t_i$, $W_i$, $s_i$, %$W_i^j$, 
$D_i$, %$s_i^j$, 
$F_i$, %$\pi_{F_i\setminus s_i}$, 
%$P_i$, 
$\Phi_i$ be as in Subsection~\ref{sec-g>1}.

\setcounter{case}{0}

\begin{case}\label{case1}
$n\ge 3$.
\end{case}

Let $[\alpha_0,\alpha_1,\dots,\alpha_{n-1}]$ be a geodesic in $\mathcal{C}(F\setminus P)$ constructed as in Proposition \ref{prop-unique-geodesic2}.
By Remark~\ref{rmk-unique-geodesic2}, we may assume that $\alpha_0$, $\alpha_1$ are non-separating and $\alpha_0\cup \alpha_1$ is separating in $S$.
By Remark~\ref{rem-sec3}, we may assume that 
%\begin{eqnarray*}
${\rm diam}_{X_1}(\pi_{X_1}(\alpha_0),\pi_{X_1}(\alpha_{n-1}))>2n+6$
%\end{eqnarray*}
holds, where $X_1$ is the subsurface of $F\setminus P$ associated with $\alpha_1$.
Let $\alpha_n'$ be a non-separating simple closed curve in $F\setminus P$ such that $\alpha_n'\ne\alpha_{n-1}$ and that $\alpha_n'\cup \alpha_{n-1}$ is separating in $F$.
By Lemma~\ref{lem-2-3-1}, there exists a homeomorphism $h:F\setminus P\rightarrow F\setminus P$ such that $h(\alpha_{n-1})=\alpha_{n-1}$ and ${\rm diam}_{X_{n-1}}(\pi_{X_{n-1}}(\alpha_0),\pi_{X_{n-1}}(h(\alpha_{n}')))>2n+16$,
where $X_{n-1}$ is the subsurface of $F\setminus P$ associated with $\alpha_{n-1}$.
Let $\alpha_n:=h(\alpha_n')$.
Then $[\alpha_0, \alpha_1,\dots,\alpha_{n}]$ is the unique geodesic connecting $\alpha_0$ and $\alpha_n$ (see the proof of Proposition~\ref{prop-unique-geodesic}), and the following hold:
\begin{itemize}
%\item every $\alpha_i$ is non-separating in $F$,
%$\alpha_0\cup \alpha_1$ is separating in $F$, and hence 
\item $\alpha_0\cup \alpha_1$ is separating in $F$,
\item $\alpha_{n-1}\cup \alpha_n$ is separating in $F$,
\item ${\rm diam}_{X_1}(\pi_{X_1}(\alpha_0),\pi_{X_1}(\alpha_{n-1}))>2n+6$,
\item ${\rm diam}_{X_{n-1}}(\pi_{X_{n-1}}(\alpha_{0}),\pi_{X_{n-1}}(\alpha_{n}))>2n+16$.
\end{itemize}
%
%Since $\alpha_0$ is non-separating in $F$, we may assume that $\alpha_0=\partial D_1$.
%Since $\alpha_n$ is non-separating in $F$, there exists a homeomorphism $f:(\partial_+ W_1,s_1\cap \partial_+W_1)\rightarrow (\partial_+ W_2,s_2\cap \partial_+ W_2)$ such that $f(\alpha_n)=\partial D_2$.

Identify $(\partial_+W_1, s_1\cap\,\partial_+W_1)$ and $(\partial_+W_2, s_2\cap\,\partial_+W_2)$ with $(F,P)$ so that $\partial D_1=\alpha_0$ and $\partial D_2=\alpha_n$.
%Note that $t_1^{\ast, 0}\cup t_2^{\ast,0}$ consists of one component since $b=1$.
%We identify $(\partial_+ W_1,s_1\cap \partial_+W_1)\,(=(F,P))$ and $(\partial_+ W_2,s_2\cap \partial_+W_2)$ by $f$.
%Let $\mathcal{D}(W_i)$ $(i=1,2)$ denote the subset of $\mathcal{C}^0(\partial_+ W_i)$ consisting of the vertices with representatives bounding disks in $W_i\setminus s_i$.
%It can be easily seen that $\mathcal{D}(W_i)$ consists of a single vertex $\partial D_i$, and hence we have $d_F(\mathcal{D}(W_1),\mathcal{D}(W_2))=n$.
By Proposition~\ref{prop-g>1-hi}, there exist homeomorphisms $h_i:\partial V_i\setminus t_i\rightarrow \partial_- W_i\setminus s_i$ such that
\begin{equation*}\label{eqn-g-1-1}
d_{\partial_- W_1\setminus s_1}(\Phi_1(\alpha_1),h_1(\mathcal{D}(V_1\setminus t_1)))>2,
\end{equation*}
\begin{equation*}\label{eqn-g-1-2}
d_{\partial_- W_2\setminus s_2}(\Phi_2(\alpha_{n-1}),h_2(\mathcal{D}(V_2\setminus t_2)))>2.
\end{equation*}
%and that $\overline{h}_i|_{\partial V_i\cap s_i}$ corresponds to ${\rm id}_{\partial V_i\cap s_i}$, where 
Let $\overline{h}_i:(\partial V_i,\partial t_i)\rightarrow (\partial_- W_i,s_i\cap \partial_- W_i)$ be the homeomorphism of the pairs induced from $h_i$.
Let $(V_i^{\ast}, t_i^{\ast}):=(W_i,s_i)\cup_{\overline{h}_i}(V_i,t_i)$.
Then $(V_1^{\ast}, t_1^{\ast})\cup_{(F,P)} (V_2^{\ast}, t_2^{\ast})$ is a $(g,1)$-splitting of a knot.
%Let $\mathcal{D}(V_i^{\ast})$ be the subset of $\mathcal{C}^0(F)$ consisting of the vertices with representatives bounding disks in $V_i^{\ast}\setminus t_i^{\ast}$.

%By arguments similar to those for Claims \ref{claim-1}, \ref{claim-1-2}, \ref{claim-1-3}, \ref{claim-2} and \ref{claim-3}, we have the next claim.

\begin{claim}\label{claim-7-1}
$\alpha_1$ intersects every element of $\mathcal{D}^0(V_1^{\ast}\setminus t_1^{\ast})\setminus\{\alpha_0\}$.
\end{claim}

\begin{proof}
Assume on the contrary that there exists an element $\beta$ of $\mathcal{D}^0(V_1^{\ast}\setminus t_1^{\ast})\setminus \{\alpha_0\}$ such that $\beta\cap \alpha_1=\emptyset$.
Then, by Proposition~\ref{prop-g>1}, we have 
$$
d_{\partial_- W_1\setminus s_1}(\Phi_1(\alpha_1),h_1(\mathcal{D}^0(V_1\setminus t_1)))\le 1,
$$
a contradiction. % to the inequality (\ref{eqn-g-1-1}).
\end{proof}

\begin{claim}\label{claim-7-2}
For any $a\in\mathcal{D}^0(V_1^{\ast}\setminus t_1^{\ast})$, we have $\pi_{X_1}(a)\ne\emptyset$ and ${\rm diam}_{X_1}(\{\alpha_0\},\pi_{X_1}(a))\le 4$.
\end{claim}

\begin{proof}
Note that by Claim \ref{claim-7-1}, we have $\pi_{X_1}(a)\ne\emptyset$.

If $a=\alpha_0$ or $a\cap\alpha_0=\emptyset$, that is, $d_{F\setminus P}(\alpha_0,a)\le 1$, then we have ${\rm diam}_{X_1}(\{\alpha_0\},\pi_{X_1}(a))\le 2$ by Lemma \ref{subsurface distance}.

Next, we suppose that $a\ne \alpha_0$ and $a\cap \alpha_0\ne\emptyset$.
Let $D_a$ be a disk in $V_1^{\ast}\setminus t_1^{\ast}$ bounded by $a$. 
Recall that $\alpha_0$ bounds the disk $D_1$.
We may assume that $|D_a\cap D_1|$ is minimal (hence, each component of $D_a\cap D_1$ is an arc).
Let $\Delta$ be the closure of a component of $D_a\setminus N(D_1)$ that is outermost in $D_a$.
Then by Proposition~\ref{prop-g>1-1}, we see that $\Delta$ is an essential disk in $(W_1'\cup_{\overline{h}_1} V_1)\setminus t_1^{\ast}$.
(Note that since $a\cap a_0\ne\emptyset$, we cannot have conclusion (A) of Proposition~\ref{prop-g>1-1} for $D=D_a$.)
%$\Delta'':=\Delta'\cap V_1$ is an essential disk in $V_1\setminus t_1$, 
%where $\Delta'$ is a disk that is isotopic to $\Delta$ in $V_1^{\ast}\setminus t_1^{\ast}$ such that $\partial \Delta'\subset F_1$.

Suppose $\partial \Delta\cap \alpha_1=\emptyset$.
Then by Proposition~\ref{prop-g>1-1} (B1), we have 
$$
d_{\partial_- W_1\setminus s_1}(\Phi_1(\alpha_1),h_1(\mathcal{D}^0(V_1\setminus t_1)))\le 1,
$$
a contradiction.
%Suppose $\partial \Delta'\cap \alpha_1=\emptyset$.
%Then $P_1(\partial \Delta')\,(=h_1(\partial \Delta''))\cap \Phi_1(\alpha_1)=\emptyset$.
%This implies 

Suppose $\partial \Delta\cap \alpha_1\ne\emptyset$.
In this case, there is a subarc $\gamma$ of $\partial \Delta\cap \partial_+ W_1'$ such that $\gamma\cap N(\alpha_1)=\partial \gamma$, 
hence $\gamma$ can be regarded as an arc properly embedded in $X_1$.
These imply that $d_{\mathcal{AC}(X_1)}(\alpha_0, \gamma)=1$.
Note that $\gamma\in \pi_{AC}(a)$, where $\pi_{AC}$ is the map from $\mathcal{C}^{0}(F\setminus P)$ to $\mathcal{P}(\mathcal{AC}^{0}(X_1))$ defined as in Subsection~\ref{sec-subsurface}. 
Hence, we have
\begin{eqnarray*}
\begin{array}{rcl}
{\rm diam}_{\mathcal{AC}(X_1)}(\{\alpha_0\}, \pi_{AC}(a))&\le& d_{\mathcal{AC}(X_1)}(\alpha_0, \gamma)+{\rm diam}_{\mathcal{AC}(X_1)}(\pi_{AC}(a))\\
&\le& 1+1=2.
\end{array}
\end{eqnarray*}
Hence, by Lemma \ref{lem_mm},
%\ref{subsurface distance}, 
we have ${\rm diam}_{X_1}(\{\alpha_0\}, \pi_{X_1}(a))\le 4$.
\end{proof}

Then the arguments in the proof of Claim~\ref{claim-1-3} works to show:

\begin{claim}\label{claim-7-3}
$[\alpha_0, \alpha_1,\dots,\alpha_{n}]$ is the unique geodesic realizing the distance $d_{F\setminus P}(\mathcal{D}^0(V_1^{\ast}\setminus t_1^{\ast}), \{\alpha_n\})=n$.
\end{claim}

The following two claims can be proved by arguments similar to those for Claims~\ref{claim-7-1} and \ref{claim-7-2}.

\begin{claim}\label{claim-7-4}
$\alpha_{n-1}$ intersects every element of $\mathcal{D}^0(V_2^{\ast}\setminus t_2^{\ast})\setminus\{\alpha_n\}$.
\end{claim}

\begin{claim}\label{claim-7-5}
For any $b\in\mathcal{D}^0(V_2^{\ast}\setminus t_2^{\ast})$, we have $\pi_{X_{n-1}}(b)\ne\emptyset$ and ${\rm diam}_{X_{n-1}}(\{\alpha_n\},\pi_{X_{n-1}}(b))\le 4$.
\end{claim}

Further, we have the next claim.

\begin{claim}\label{claim-d-cplx}
%${\rm diam}_{X_1}(\pi_{X_1}(\mathcal{D}^0(V_2^{\ast}\setminus t_2^{\ast})))\le 12$ and 
${\rm diam}_{X_{n-1}}(\pi_{X_{n-1}}(\mathcal{D}^0(V_1^{\ast}\setminus t_1^{\ast})))\le 12$.
\end{claim}

\begin{proof}
Assume that ${\rm diam}_{X_{n-1}}(\pi_{X_{n-1}}(\mathcal{D}^0(V_1^{\ast}\setminus t_1^{\ast})))> 12$ on the contrary.
By Proposition~\ref{prop-image-d} in Appendix~\ref{appendix-disk-complex}, 
$(V_1^{\ast}, \alpha_{n-1})$ is homeomorphic to the twisted $I$-bundle $\Omega\tilde{\times} I$ over a non-orientable surface $\Omega$, where $t_1^{\ast}$ is an $I$-fiber and $\alpha_{n-1}$ is the core curve of the annulus $\partial\Omega\tilde{\times}I$.
Let $\epsilon$ be an essential arc on $\Omega$ such that $(\epsilon\tilde{\times}I)\cap t_1^{\ast}=\emptyset$, and let $E:=\epsilon\tilde{\times}I$.
Then $E$ is an essential disk in $V_1^{\ast}\setminus t_1^{\ast}$.
Note that $|E\cap \alpha_{n-1}|=2$. 
Let $\beta$ be an essential simple closed curve on $\partial V_1^{\ast}\setminus t_1^{\ast}$ disjoint from both $\partial E$ and $\alpha_{n-1}$.
There is another essential simple closed curve $\beta'$ on $\partial V_1^{\ast}\setminus t_1^{\ast}$ disjoint from both $\partial E$ and $\alpha_{n-1}$ (for example, since $t_1^{\ast}$ is an $I$-fiber, we can choose $\beta'$ so that $\beta\cup\beta'$ bounds a once-punctured annulus disjoint from $\partial E\cup\alpha_{n-1}$). 
Then $[\partial E,\beta,\alpha_{n-1}]$ and $[\partial E,\beta',\alpha_{n-1}]$ are distinct geodesics connecting $\mathcal{D}^0(V_1^{\ast}\setminus t_1^{\ast})$ and $\alpha_{n-1}$, which contradicts Claim~\ref{claim-7-3}.
Hence, we have the desired inequality.
\end{proof}

Then the above claims together with the arguments similar to those in the proof of Claim~\ref{claim-4-1}, we can see that the $(g,1)$-splitting is of distance $n$ and is strongly keen.

%\begin{case}\label{case2}
%$n=3$.
%\end{case}

%.....

\begin{case}\label{case3}
$n=2$.
\end{case}

By using arguments similar to those in \cite[Section 5]{IJK3}, we have a geodesic $[\alpha_0,\alpha_1,\alpha_2]$ in $\mathcal{C}(F\setminus P)$ such that
\begin{itemize}
\item every $\alpha_i$ is non-separating in $F$,
\item $\alpha_0\cup \alpha_1$ is separating in $F$, 
\item $\alpha_{1}\cup \alpha_2$ is separating in $F$,
\item $d_{X_1}(\alpha_0,\alpha_{2})>12$,
where $X_1$ is the subsurface of $F\setminus P$ associated with $\alpha_1$.
\end{itemize}
Identify $(\partial_+W_1, s_1\cap\,\partial_+W_1)$ and $(\partial_+W_2, s_2\cap\,\partial_+W_2)$ with $(F,P)$ so that $\partial D_1=\alpha_0$ and $\partial D_2=\alpha_2$.
By Proposition~\ref{prop-g>1-hi}, there exist homeomorphisms $h_i:\partial V_i\setminus t_i\rightarrow \partial_- W_i\setminus s_i$ $(i=1,2)$ such that
\begin{equation*}\label{eqn-g-1-5}
d_{\partial_- W_1\setminus s_1}(\Phi_1(\alpha_1),h_1(\mathcal{D}^0(V_1\setminus t_1)))>2,
\end{equation*}
\begin{equation*}\label{eqn-g-1-5-2}
d_{\partial_- W_2\setminus s_2}(\Phi_2(\alpha_1),h_2(\mathcal{D}^0(V_2\setminus t_2)))>2.
\end{equation*}
%and that $\overline{h}_i|_{\partial V_i\cap t_i}$ corresponds to ${\rm id}_{\partial V_i\cap t_i}$, where 
Let $\overline{h}_i:(\partial V_i,\partial t_i)\rightarrow (\partial_- W_i,s_i\cap \partial_- W_i)$ be the homeomorphism induced from $h_i$.
Let $(V_i^{\ast}, t_i^{\ast}):=(W_i,s_i)\cup_{\overline{h}_i}(V_i,t_i)$.
Then $(V_1^{\ast}, t_1^{\ast})\cup_{(F,P)} (V_2^{\ast}, t_2^{\ast})$ is a $(g,1)$-splitting of a knot.
%Let $\mathcal{D}(V_i^{\ast})$ be the subset of $\mathcal{C}^0(F)$ consisting of the vertices with representatives bounding disks in $V_i^{\ast}\setminus t_i^{\ast}$.
By arguments similar to those in Section \ref{sec-proof-2}, we can see that the $(g,1)$-splitting is of distance $2$ and is strongly keen.

This completes the proof of Theorem \ref{thm-1} for the case when $n\ge 2$, $g\ge 2$ and $b=1$.

\section{Proof of Theorem \ref{thm-1} when $n\ge 2$, $g=1$ and $b=1$}\label{sec-proof-5}

%In this section, we show that there exists a strongly keen $(1,1)$-splitting of a knot with distance $n$
%for any given integer $n\,(\ge 2)$.

In this section, we give a proof of Theorem \ref{thm-1} for the case when $n\ge 2$ and $(g,b)=(1,1)$.
We note that the settings of Subsections~\ref{sec-b>1} and \ref{sec-g>1} are not applicable to this case.

For $i=1,2$, let $V_i$ be a solid torus, $t_i$ an arc properly embedded in $V_i$, and $D_i$ the essential disk in $V_i\setminus t_i$ that cuts $V_i$ into a solid torus and a component containing $t_i$, as in Figure \ref{fig-vi-ti}.
\begin{figure}[tb]
 \begin{center}
 \includegraphics[width=40mm]{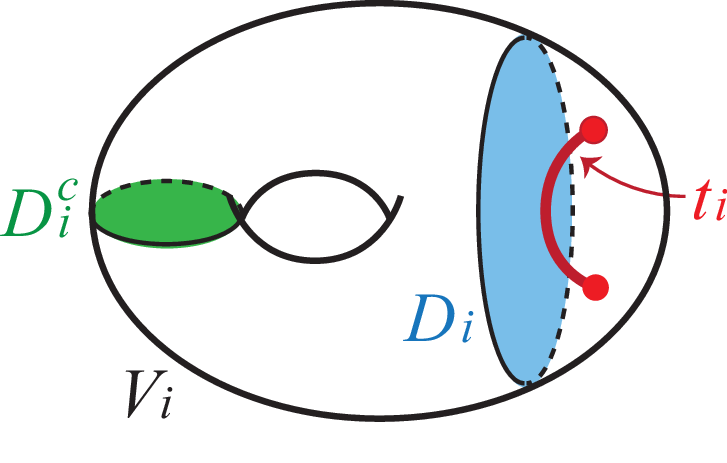}
 \end{center}
 \caption{$V_i$, $t_i$ and $D_i$.}
\label{fig-vi-ti}
\end{figure}
Recall that $\mathcal{D}(V_i\setminus t_i)$ is the disk complex of $V_i\setminus t_i$.
By \cite[Proposition 3.8]{Sai}, $\mathcal{D}(V_i\setminus t_i)$ is the join $\{\partial D_i^c\}\ast \mathcal{A}_i$, where $D_i^c$ is the unique non-separating disk in $V_i\setminus t_i$ and $\mathcal{A}_i$ consists of the countably infinite vertices corresponding to essential separating disks in $V_i\setminus t_i$.
In particular, %the diameter of $\mathcal{D}(V_i\setminus t_i)$ is $2$, 
\begin{equation}\label{eqn-diam}
{\rm diam}_{\partial V_i\setminus t_i} (\mathcal{D}^0(V_i\setminus t_i))=2,
\end{equation}
and there are no edges in $\mathcal{C}(\partial V_i\setminus t_i)$ connecting distinct elements of $\mathcal{A}_i$.

\setcounter{case}{0}

\begin{case}\label{case1-0}
$n\ge 3$.
\end{case}

Let $F$ be a torus and $P$ be the union of $2$ points on $F$.
Let $[\alpha_1,\dots,\alpha_{n-1}]$ be a geodesic in $\mathcal{C}(F\setminus P)$ such that
every $\alpha_i$ $(1\le i\le n-1)$ is non-separating in $F$ and that $[\alpha_1,\dots,\alpha_{n-1}]$ is the unique geodesic connecting $\alpha_1$ and $\alpha_{n-1}$.
(We have such a geodesic by Proposition \ref{prop-unique-geodesic2} when $n\ge 4$. When $n=3$, we may choose $\alpha_1$ and $\alpha_2$ to be non-separating simple closed curves in $F\setminus P$ that are mutually disjoint and non-isotopic. See Figure~\ref{fig-vi-ti2}.)
\begin{figure}[tb]
 \begin{center}
 \includegraphics[width=35mm]{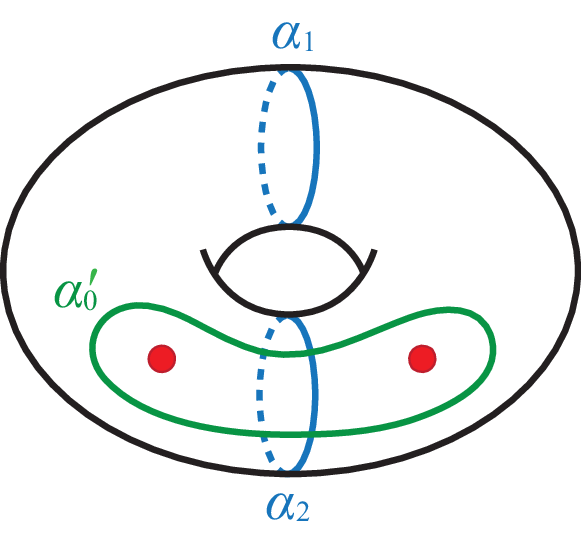}
 \end{center}
 \caption{$\alpha_1$, $\alpha_2$ and $\alpha_0'$ for the case when $n=3$.}
\label{fig-vi-ti2}
\end{figure}
Let $X_1$ and $X_{n-1}$ be the subsurfaces of $F\setminus P$ associated with $\alpha_1$ and $\alpha_{n-1}$, respectively.
Let $\alpha_0'$ be a simple closed curve in $F\setminus P$ which is disjoint from $\alpha_1$ and cuts off a twice-punctured disk from $F\setminus P$.
By Lemma \ref{lem-2-3-1}, there exists a homeomorphism $g_1:F\setminus P \rightarrow F\setminus P$ such that $g_1(\alpha_1)=\alpha_1$ and ${\rm diam}_{X_1}(\pi_{X_1}(g_1(\alpha_0')), \pi_{X_1}(\alpha_{n-1}))>2n+10$.
Let $\alpha_0:=g_1(\alpha_0')$. 
Let $\alpha_n'$ be a simple closed curve in $F\setminus P$ which is disjoint from $\alpha_{n-1}$ and cuts off a twice-punctured disk from $F\setminus P$.
By Lemma \ref{lem-2-3-1}, there exists a homeomorphism $g_{n-1}:F\setminus P\rightarrow F\setminus P$ such that $g_{n-1}(\alpha_{n-1})=\alpha_{n-1}$ and ${\rm diam}_{X_{n-1}}(\pi_{X_{n-1}}(g_{n-1}(\alpha_n')), \pi_{X_{n-1}}(\alpha_0))>2n+8$.
Let $\alpha_n:=g_{n-1}(\alpha_n')$. 
By applying Proposition~\ref{prop-extending-geodesic} for $[\alpha_0,\alpha_1,\dots,\alpha_{n-1},\alpha_n]$, we see that every geodesic connecting $\alpha_0$ and $\alpha_n$ passes through $\alpha_{n-1}$.
Then by applying Proposition~\ref{prop-extending-geodesic} for $[\alpha_0,\alpha_1,\dots,\alpha_{n-1}]$, we see that every geodesic connecting $\alpha_0$ and $\alpha_{n-1}$ passes through $\alpha_{1}$.
These facts together with the uniqueness of $[\alpha_1,\dots,\alpha_{n-1}]$ show that the geodesic $[\alpha_0,\alpha_1,\dots,\alpha_{n}]$ is the unique geodesic connecting $\alpha_0$ and $\alpha_n$.
We remark that each of $\alpha_0$ and $\alpha_n$ cuts off a twice-punctured disk from $F\setminus P$, each of $\alpha_1$ and $\alpha_{n-1}$ is non-separating in $F$, and the following inequalities hold:
\begin{equation}\label{eqn-g1-1}
{\rm diam}_{X_1}(\pi_{X_1}(\alpha_0), \pi_{X_1}(\alpha_{n-1}))>2n+10,
\end{equation}
\begin{equation}\label{eqn-g1-2}
{\rm diam}_{X_{n-1}}(\pi_{X_{n-1}}(\alpha_0), \pi_{X_{n-1}}(\alpha_n))>2n+8.
\end{equation}

Identify $(\partial V_1,\partial t_1)$ and $(\partial V_2,\partial t_2)$ with $(F,P)$ so that $\partial D_1=\alpha_0$, $\partial D_2=\alpha_n$.
Further, by Appendix~\ref{app-b}, we may suppose that 
\begin{equation}\label{eqn-8-case1}
d_{X_0}(\partial D_1^c, \alpha_1)>2\ \text{and}\ d_{X_n}(\partial D_2^c, \alpha_{n-1})>2, 
\end{equation}
where $X_0$ and $X_n$ are the subsurfaces of $F\setminus P$ associated with $\alpha_0$ and $\alpha_n$, respectively.
Then $(V_1,t_1)\cup_{(F,P)}(V_2,t_2)$ is a $(1,1)$-splitting of a knot.

\begin{claim}\label{claim-g1-1}
{\rm (1)} $\alpha_1$ intersects every element of $\mathcal{D}^0(V_1\setminus t_1)\setminus \{\alpha_0\,(=\partial D_1)\}$.

{\rm (2)} $\alpha_{n-1}$ intersects every element of $\mathcal{D}^0(V_2\setminus t_2)\setminus \{\alpha_n\,(=\partial D_2)\}$.
\end{claim}

\begin{proof}
We give a proof for (1) only, since (2) can be proved similarly. 
Assume on the contrary that there exists an element $a$ of $\mathcal{D}^0(V_1\setminus t_1)\setminus \{\partial D_1\}$ such that $a\cap \alpha_1=\emptyset$.
Let $D_a$ be a disk in $V_1\setminus t_1$ bounded by $a$.
We may assume that $|D_a\cap D_1|$ is minimal (hence, no component of $D_a\cap D_1$ is a loop).
%By the minimality of $|D_a\cap D_1|$, we see that $D_a\cap D_1$ has no loop components. 
Let $\Delta$ be the closure of a component of $D_a\setminus D_1$ that is outermost in $D_a$.
Note that $D_1$ cuts $V_1$ into a solid torus $W_1^1$ and a $3$-ball $W_1^2$ containing $t_1$ that is parallel to $\partial W_1^2$.
By the minimality of $|D_a\cap D_1|$, we see that $\Delta$ must be a non-separating disk in $W_1^1$.
Let $\Delta'$ be a disk properly embedded in $W_1^1$, parallel to the union of $\Delta$ and one of the two components of $D_1\setminus\Delta$.
Since $W_1^1$ is a solid torus, $\Delta'$ is isotopic to the disk $D_1^c$, which implies $\partial D_1^c\in \pi_{X_0}(a)$. 
Recall $a\cap \alpha_1=\emptyset$ by the assumption.
This fact together with Lemma~\ref{subsurface distance} implies:
$$
d_{X_0}(\partial D_1^c,\alpha_1)\le {\rm diam}_{X_0}(\pi_{X_0}(a),\{\alpha_1\})\le 2\cdot d_{F\setminus P}(a,\alpha_1)=2,
$$
%\begin{eqnarray*}
%\begin{array}{rcl}
%d_{X_0}(\partial D_1^c,\alpha_1)&\le& {\rm diam}_{X_0}(\pi_{X_0}(a),\alpha_1)\\
%&\le &2\times d_F(a,\alpha_1)=2,
%\end{array}
%\end{eqnarray*}
contradicting the inequality (\ref{eqn-8-case1}).
\end{proof}

%The following claim can be proved by using arguments similar to that for the proof of \cite[Claim 4.2]{IJK3}.

\begin{claim}\label{claim-g1-2}
${\rm diam}_{X_1}(\pi_{X_1}(\mathcal{D}^0(V_1\setminus t_1)))\le 4$ and ${\rm diam}_{X_{n-1}}(\pi_{X_{n-1}}(\mathcal{D}^0(V_2\setminus t_2)))\le 4$.
%{\rm (1)} ${\rm diam}_{X_1}(\mathcal{D}^0(V_1))\le 4$.\\
%{\rm (2)} ${\rm diam}_{X_{n-1}}(\mathcal{D}^0(V_2))\le 4$.
\end{claim}

\begin{proof}
By Claim \ref{claim-g1-1} (1), every element of $\mathcal{D}^0(V_1\setminus t_1)$ cuts $X_1$.
Since the diameter of $\mathcal{D}(V_1\setminus t_1)$ is $2$ as mentioned before, we have ${\rm diam}_{X_1}(\pi_{X_1}(\mathcal{D}^0(V_1\setminus t_1)))\le 4$ by Lemma \ref{subsurface distance}.
Similarly, we have ${\rm diam}_{X_{n-1}}(\pi_{X_{n-1}}(\mathcal{D}^0(V_2\setminus t_2)))\le 4$.
\end{proof}

\begin{claim}\label{claim-g1-3}
${\rm diam}_{X_1}(\pi_{X_{1}}(\mathcal{D}^0(V_2\setminus t_2)))\le 4$ and ${\rm diam}_{X_{n-1}}(\pi_{X_{n-1}}(\mathcal{D}^0(V_1\setminus t_1)))\le 4$.
\end{claim}

\begin{proof}
%When $n=2$, the inequalities follow directly from Claim \ref{claim-g1-2}.
%Hence, we assume $n\ge 3$ in the rest of the proof.
%
Note that $\alpha_n$ cuts $X_1$.
Also, $\partial D_2^c$ cuts $X_1$, since otherwise, we have $\partial D_2^c=\alpha_1$, which implies 
$$1=d_{F\setminus P}(\partial D_2^c,\partial D_2)=d_{F\setminus P}(\alpha_1,\alpha_n)=n-1,$$ 
and hence $n=2$, a contradiction.

Let $a$ be any element of $\mathcal{D}^0(V_2\setminus t_2)\setminus \{\alpha_n,\partial D_2^c\}$.
If $a$ misses $X_1$, that is, $a=\alpha_1$, then 
$$d_{F\setminus P}(a,\partial D_2)=d_{F\setminus P}(\alpha_1,\alpha_n)=n-1.$$ 
However, $d_{F\setminus P}(a,\partial D_2)=2$, since $\mathcal{D}(V_2\setminus t_2)=\{\partial D_2^c\}\ast \mathcal{A}_2$ as mentioned above.
These give $n-1=2$, which is a contradiction when $n\ge 4$.
Suppose $n=3$.
By the fact that $\mathcal{D}(V_2\setminus t_2)=\{\partial D_2^c\}\ast \mathcal{A}_2$ again, we see that $[a=\alpha_1,\partial D_2^c,\partial D_2=\alpha_3]$ is a geodesics connecting $\alpha_1$ and $\alpha_3$.
By the uniqueness of the geodesic $[\alpha_1,\alpha_2,\alpha_3]$, we have $\partial D_2^c=\alpha_2$, which contradicts Claim \ref{claim-g1-1} (2).
Hence, $a$ also cuts $X_1$.

Since the diameter of $\mathcal{D}(V_2\setminus t_2)$ is $2$ and every element of $\mathcal{D}^0(V_2\setminus t_2)$ cuts $X_1$ as shown above, we have ${\rm diam}_{X_1}(\pi_{X_1}(\mathcal{D}^0(V_2\setminus t_2)))\le 4$ by Lemma \ref{subsurface distance}.

Similarly, we have ${\rm diam}_{X_{n-1}}(\pi_{X_{n-1}}(\mathcal{D}^0(V_1\setminus t_1)))\le 4$.
\end{proof}

%\begin{claim}\label{claim-g1-4}
%The distance of the $(1,1)$-splitting $(V_1,t_1)\cup_F(V_2,t_2)$ is $n$.
%\end{claim}

To prove that the $(1,1)$-splitting $(V_1,t_1)\cup_{(F,P)}(V_2,t_2)$ has distance $n$ and is strongly keen, 
let $[\beta_0,\beta_1,\dots,\beta_m]$ be a shortest geodesic in $\mathcal{C}(F\setminus P)$ such that $\beta_0\in\mathcal{D}^0(V_1\setminus t_1)$ and $\beta_m\in\mathcal{D}^0(V_2\setminus t_2)$.
Since $d_{F\setminus P}(\mathcal{D}^0(V_1\setminus t_1),\mathcal{D}^0(V_2\setminus t_2))\le d_{F\setminus P}(\alpha_0,\alpha_n)\le n$, we may assume $m\le n$.

\begin{claim}\label{claim-g1-5-0}
$\beta_i=\alpha_1$ and $\beta_{j}=\alpha_{n-1}$ for some $i,j\in\{0,1,\dots,m\}$.
\end{claim}

\begin{proof}
%First assume that $n\ge 3$.
Assume that $\beta_i\ne \alpha_1$ for every $i\in\{0,1,\dots,m\}$.
Namely, every $\beta_i$ cuts $X_1$.
By Lemma \ref{subsurface distance}, we have ${\rm diam}_{X_1}(\pi_{X_1}(\beta_0),\pi_{X_1}(\beta_m))\le 2m\le 2n$.
Similarly, we have ${\rm diam}_{X_1}(\pi_{X_1}(\alpha_{n-1}),\pi_{X_1}(\alpha_n))\le 2$, since each of $\alpha_{n-1}$ and $\alpha_n$ cuts $X_1$. %for every $j\in\{2,\dots,n\}$.
By using the above inequalities together with Claims \ref{claim-g1-2} and \ref{claim-g1-3}, we have
\begin{eqnarray*}
\begin{array}{rcl}
{\rm diam}_{X_1}(\pi_{X_1}(\alpha_0),\pi_{X_1}(\alpha_{n-1}))&\le& 
{\rm diam}_{X_1}(\pi_{X_1}(\alpha_0),\pi_{X_1}(\beta_0))\\
&&+{\rm diam}_{X_1}(\pi_{X_1}(\beta_0),\pi_{X_1}(\beta_m))\\
&&+{\rm diam}_{X_1}(\pi_{X_1}(\beta_m),\pi_{X_1}(\alpha_n))\\
&&+{\rm diam}_{X_1}(\pi_{X_1}(\alpha_n),\pi_{X_1}(\alpha_{n-1}))\\
&\le&4+2n+4+2\\&=&2n+10,
\end{array}
\end{eqnarray*}
which contradicts the inequality (\ref{eqn-g1-1}).
Hence, we have $\beta_i=\alpha_1$ for some $i\in\{0,1,\dots,m\}$.
Similarly, we have $\beta_j=\alpha_{n-1}$ for some $j\in\{0,1,\dots,m\}$, since otherwise
\begin{eqnarray*}
\begin{array}{rcl}
{\rm diam}_{X_{n-1}}(\pi_{X_{n-1}}(\alpha_0),\pi_{X_{n-1}}(\alpha_{n}))&\le& 
{\rm diam}_{X_{n-1}}(\pi_{X_{n-1}}(\alpha_0),\pi_{X_{n-1}}(\beta_0))\\
&&+{\rm diam}_{X_{n-1}}(\pi_{X_{n-1}}(\beta_0),\pi_{X_{n-1}}(\beta_m))\\
&&+{\rm diam}_{X_{n-1}}(\pi_{X_{n-1}}(\beta_m),\pi_{X_{n-1}}(\alpha_{n}))\\
&\le&4+2n+4\\&=&2n+8,
\end{array}
\end{eqnarray*}
which contradicts the inequality (\ref{eqn-g1-2}).
%
%Next assume that $n=2$.
%We have $\beta_i=\alpha_1$ for some $i$, since otherwise, we have
%\begin{eqnarray*}
%\begin{array}{rcl}
%{\rm diam}_{X_{1}}(\pi_{X_{1}}(\alpha_0),\pi_{X_{1}}(\alpha_{2}))&\le& 
%{\rm diam}_{X_{1}}(\pi_{X_{1}}(\alpha_0),\pi_{X_{1}}(\beta_0))\\
%&&+{\rm diam}_{X_{1}}(\pi_{X_{1}}(\beta_0),\pi_{X_{1}}(\beta_m))\\
%&&+{\rm diam}_{X_{1}}(\pi_{X_{1}}(\beta_m),\pi_{X_{1}}(\alpha_{2}))\\
%&\le&4+2n+4\\&=&2n+8,
%\end{array}
%\end{eqnarray*}
%which contradicts the inequality (\ref{eqn-g1-2}).
\end{proof}

Let $i, j$ be as in Claim~\ref{claim-g1-5-0}.
Since $[\beta_0,\beta_1,\dots,\beta_m]$ is a shortest geodesic, we see by Claim~\ref{claim-g1-1} that $i=1$ and $j=m-1$.
By the uniqueness of the geodesic, we have $[\beta_1,\dots,\beta_{m-1}]=[\alpha_1,\dots,\alpha_{n-1}]$.
This implies that $m=n$, and $\beta_{n-1}=\alpha_{n-1}$.
%Note that $\beta_i\ne \alpha_0$ and $\beta_j\ne \alpha_n$ by Claim~\ref{claim-g1-1}. 
%These together with the assumption $m\le n$ implies that $m=n$, $i=1$ and $j=n-1$.
By Claim~\ref{claim-g1-1}, we have $\beta_0=\alpha_0$ and $\beta_m=\alpha_n$, and this shows that the $(1,1)$-splitting has distance $n$ and  is strongly keen.
%which completes the proof of Theorem \ref{thm-1} when $n\ge 2$ and $(g,b)=(1,1)$.

\begin{case}\label{case2-0}
$n=2$.
\end{case}

Let $F$ be a torus and $P$ be the union of $2$ points on $F$.
Let $\alpha_1$ be a non-separating simple closed curve in $F\setminus P$, 
and let $\alpha_0$ and $\alpha_2$ be simple closed curves each of which is disjoint from $\alpha_1$ and cuts off a twice-punctured disk from $F\setminus P$.
For $i=0,1,2$, let $X_i$ be the subsurface of $F\setminus P$ associated with $\alpha_i$.
By Appendix~\ref{app-b}, we may suppose that
\begin{equation}\label{eqn-g1-n2-0}
{\rm diam}_{X_1}(\pi_{X_1}(\alpha_0), \pi_{X_1}(\alpha_2))>12.
\end{equation}
Note that $\alpha_0\cap\alpha_2\ne\emptyset$, and hence, $[\alpha_0,\alpha_1,\alpha_2]$ is a geodesic in $\mathcal{C}(F\setminus P)$.
%By Lemma \ref{lem-2-3-1}, there exists a homeomorphism $g:F\setminus P \rightarrow F\setminus P$ such that $g(\alpha_1)=\alpha_1$ and ${\rm diam}_{X_1}(\pi_{X_1}(\pi_{X_1}(\alpha_0), \pi_{X_1}(g(\alpha_{2})))>12$.
%Let $\alpha_2:=g(\alpha_2')$. 
%Then $[\alpha_0,\alpha_1,\alpha_{2}]$ is the unique geodesic connecting $\alpha_0$ and $\alpha_2$ by Proposition \ref{prop-extending-geodesic}.
%We remark that each of $\alpha_0$ and $\alpha_2$ cuts off a twice-punctured disk from $F\setminus P$, $\alpha_1$ is non-separating in $F$, and the following inequality holds:
%
%\begin{equation}\label{eqn-g1-n2}
%{\rm diam}_{X_1}(\pi_{X_1}(\alpha_0), \pi_{X_1}(\alpha_{2}))>12.
%\end{equation}

Identify $(\partial V_1,\partial t_1)$ and $(\partial V_2,\partial t_2)$ with $(F,P)$ so that $\partial D_1=\alpha_0$, $\partial D_2=\alpha_2$ 
and the following two inequalities hold.
\begin{equation}\label{eqn-g1-n2-1}
{d}_{X_0}(\partial D_1^c, \alpha_1)\ge 2,
\end{equation}
\begin{equation}\label{eqn-g1-n2-2}
{d}_{X_2}(\partial D_2^c, \alpha_1)\ge 2.
\end{equation}
%\begin{equation}\label{eqn-g1-n2-3}
%{\rm diam}_{X_1}(\pi_{X_1}(\partial D_1^c), \pi_{X_1}(\partial D_2^c))>8.
%\end{equation}
%
%Further we may suppose that $d_{X_0}(\partial D_1^c, \{\alpha_1\})>2$ and $d_{X_n}(\partial D_2^c, \{\alpha_{n-1}\})>2$, where $X_0$ and $X_n$ are the subsurfaces of $F\setminus P$ associated with $\alpha_0$ and $\alpha_n$, respectively.
Then $(V_1,t_1)\cup_{(F,P)}(V_2,t_2)$ is a $(1,1)$-splitting of a knot.

\begin{claim}\label{claim-g1-1-n2}
{\rm (1)} $\alpha_1$ intersects every element of $\mathcal{D}^0(V_1\setminus t_1)\setminus \{\alpha_0\,(=\partial D_1)\}$.

{\rm (2)} $\alpha_1$ intersects every element of $\mathcal{D}^0(V_2\setminus t_2)\setminus \{\alpha_2\,(=\partial D_2)\}$.
\end{claim}

\begin{proof}
Recall that $\mathcal{D}^0(V_i\setminus t_i)=\{\partial D_i^c\}\ast \mathcal{A}_i$, and note that (the boundary of) every element of $\mathcal{A}_i$ cuts off a twice-punctured disk from $F\setminus P$.
By the inequality~(\ref{eqn-g1-n2-1}), $\alpha_1$ intersects $\partial D_1^c$.
Hence, for the proof of the conclusion~(1) of the claim, it is enough to show that $\alpha_1$ intersects every element of
$(\mathcal{D}^0(V_1\setminus t_1)\setminus \{\alpha_0\})\setminus \{\partial D_1^c\}=\mathcal{A}_1\setminus \{\alpha_0\}$.
Assume on the contrary that there exists $\gamma\in\mathcal{A}_1\setminus\{\alpha_0\}$ such that $\gamma\cap \alpha_1=\emptyset$.
It is easy to see that at most one component of $\partial V_1\setminus(\alpha_0\cup \gamma)$ is not simply connected,
and the non-simply connected component is an open annulus.
Here we note that $\alpha_1$ and $\partial D_1^c$ must be contained in the non-simply connected component, and this shows that $\alpha_1=\partial D_1^c$.
However, this contradicts the inequality (\ref{eqn-g1-n2-1}).

The conclusion (2) of the claim can be proved similarly.
%Hence, $\alpha_1$ intersects every element of $\mathcal{D}^0(V_1\setminus t_1)\setminus \{\alpha_0, \partial D_1^c\}$.
%Similarly, $\alpha_1$ intersects every element of $\mathcal{D}^0(V_2\setminus t_2)\setminus \{\alpha_2, \partial D_2^c\}$.
%Also, by the inequality (\ref{eqn-g1-n2-1}) (resp. (\ref{eqn-g1-n2-2})), $\alpha_1$ intersects $\partial D_1^c$ (resp. $\partial D_2^c)$.
\end{proof}

\begin{claim}\label{claim-g1-4-n2}
The distance of $(V_1,t_1)\cup_{(F,P)}(V_2,t_2)$ is not $1$.
\end{claim}

\begin{proof}
Assume on the contrary that the distance of $(V_1,t_1)\cup_{(F,P)}(V_2,t_2)$ is $1$.
Then there exist $\beta_0\in\mathcal{D}^0(V_1\setminus t_1)$ and $\beta_1\in\mathcal{D}^0(V_2\setminus t_2)$ such that $\beta_0\cap \beta_1=\emptyset$.
By Claim~\ref{claim-g1-1-n2}, we see that each of $\beta_0$ and $\beta_1$ cuts $X_1$.
%Note that $d_{F\setminus P}(\beta_0,\partial D_1^c)\le 1$, $d_{F\setminus P}(\beta_1,\partial D_2^c)\le 1$, and each of $\partial %D_1^c$, $\beta_0$, $\beta_1$ and $\partial D_2^c$ cuts $X_1$.
%Hence, %by Lemma~\ref{subsurface distance}, we have 
%$$
%{\rm diam}_{X_1}(\pi_{X_1}(\partial D_1^c), \pi_{X_1}(\partial D_2^c))\le 6,
%$$
This together with the facts ${\rm diam}_{\partial V_i\setminus t_i} (\mathcal{D}^0(V_i\setminus t_i))=2$ $(i=1,2)$ (the equality (\ref{eqn-diam})) and Lemma~\ref{subsurface distance} shows:
\begin{eqnarray*}
\begin{array}{rcl}
{\rm diam}_{X_1}(\pi_{X_1}(\alpha_0),\pi_{X_1}(\alpha_2))&\le&{\rm diam}_{X_1}(\pi_{X_1}(\alpha_0),\pi_{X_1}(\beta_0))\\
&&+{\rm diam}_{X_1}(\pi_{X_1}(\beta_0),\pi_{X_1}(\beta_1))\\&&+{\rm diam}_{X_1}(\pi_{X_1}(\beta_1),\pi_{X_1}(\alpha_2))\\
&\le&4+2+4=10,
\end{array}
\end{eqnarray*}
which contradicts the inequality (\ref{eqn-g1-n2-0}).
\end{proof}

By Claim~\ref{claim-g1-4-n2}, the distance of $(V_1,t_1)\cup_{(F,P)}(V_2,t_2)$ is $2$.
Let $[\beta_0,\beta_1,\beta_2]$ be a geodesic realizing the distance.

\begin{claim}\label{claim-g1-5}
$\beta_1=\alpha_1$.
\end{claim}

\begin{proof}
Assume on the contrary that $\beta_1\ne \alpha_1$.
Since $\beta_0\in \mathcal{D}^0(V_1\setminus t_1)$ (resp. $\beta_2\in \mathcal{D}^0(V_2\setminus t_2)$), $\beta_0\ne \alpha_1$ (resp. $\beta_2\ne \alpha_1$).
Then each of $\beta_0$, $\beta_1$ and $\beta_2$ cuts $X_1$.
This together with the facts ${\rm diam}_{\partial V_i\setminus t_i} (\mathcal{D}^0(V_i\setminus t_i))=2$ $(i=1,2)$ (the equality~(\ref{eqn-diam})) and Lemma~\ref{subsurface distance} shows:
%$$
%{\rm diam}_{X_1}(\pi_{X_1}(\partial D_1^c), \pi_{X_1}(\partial D_2^c))\le 8,
%$$
\begin{eqnarray*}
\begin{array}{rcl}
{\rm diam}_{X_1}(\pi_{X_1}(\alpha_0),\pi_{X_1}(\alpha_2))&\le&{\rm diam}_{X_1}(\pi_{X_1}(\alpha_0),\pi_{X_1}(\beta_0))\\
&&+{\rm diam}_{X_1}(\pi_{X_1}(\beta_0),\pi_{X_1}(\beta_2))\\&&+{\rm diam}_{X_1}(\pi_{X_1}(\beta_2),\pi_{X_1}(\alpha_2))\\
&\le&4+4+4=12,
\end{array}
\end{eqnarray*}
which contradicts the inequality (\ref{eqn-g1-n2-0}).
\end{proof}

By Claims~\ref{claim-g1-1-n2} and \ref{claim-g1-5}, we have $[\beta_0,\beta_1,\beta_2]=[\alpha_0,\alpha_1,\alpha_2]$, i.e., $(V_1,t_1)\cup_{(F,P)}(V_2,t_2)$ is strongly keen.

This completes the proof of Theorem \ref{thm-1} for the case when $n\ge 2$, $g=1$ and $b=1$.

\part{Proof of Theorem \ref{thm-1} when $n=1$}

\section{Proof of Theorem \ref{thm-1} when $n=1$ and $g\ge 2$}\label{sec-proof-6}

In this section, we give a proof of Theorem \ref{thm-1} for the case when $n=1$ and $g\ge 2$.
We remark that the idea of the key part of the proof in this and the next sections is due to \cite{E}.

Let $F$ be a closed orientable surface of genus $g$ and let $P$ be the union of $2b$ points on $F$, where $b\ge 1$.
Let $\alpha_0$ and $\alpha_1$ be non-separating simple closed curves on $F\setminus P$ such that $\alpha_0\cap\alpha_1=\emptyset$ and that $\alpha_0\cup \alpha_1$ separates $F\setminus P$ into two components, one of which is an annulus with two punctures.
For $i=1,2$, let $V_i^{\ast, 0}$, $t_i^{\ast,0}$, $V_i$, $t_i$, $W_i'$, $\partial_+W_i'$, $W_i$, $\partial_- W_i$, $s_i$, $D_i$, %$W_i'$, 
$D_i\times \{\varepsilon\}$ ($\varepsilon=0,1$), $F_i$, $\Phi_i$ be as in Subsection~\ref{sec-g>1}.
Identify $(\partial_+W_1, s_1\cap\,\partial_+W_1)$ and $(\partial_+W_2, s_2\cap\,\partial_+W_2)$ with $(F,P)$ so that $\partial D_1=\alpha_0$ and $\partial D_2=\alpha_1$. %and that $t_1^{\ast,0}\cup t_2^{\ast,0}$ consists of $c$ components.
%Let $D_1^0$ and $D_1^1$ be the two disks $N(D_1)\cap W_1'$, and $D_2^0$ and $D_2^1$ be the two disks $N(D_2)\cap W_2'$, such that 
We adopt notations $N_{W_i}(D_i)=D_i\times I$, $D_i^{\varepsilon}:=D_i\times \{\varepsilon\}$ ($i=1,2$, $\varepsilon=0,1$), where
$\partial D_1^1\cup \partial D_2^1$ bounds an annulus with two punctures (in $F\setminus P$) disjoint from $\partial D_1^0\cup \partial D_2^0$. 
Let $F_A$ be the annulus with two punctures, and let $F_B$ be the genus-$(g-1)$ subsurface of $F\setminus P$ bounded by $D_1^0\cup D_2^0$. 
See Figure~\ref{fig-n1-1}.
\begin{figure}[tb]
 \begin{center}
 \includegraphics[width=50mm]{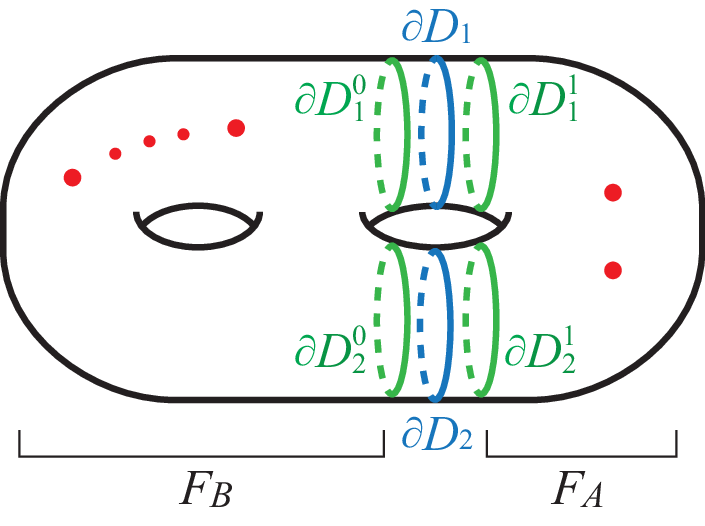}
 \end{center}
 \caption{$D_i^0$, $D_i^1$, $F_A$ and $F_B$.}
\label{fig-n1-1}
\end{figure}
Note that $F_B$ has positive genus.

%For $i=1,2$, let $V_i$ be a genus-$(g-1)$ handlebody and $t_i$ be the union of $b$ arcs properly embedded in $V_i$ which is parallel to $\partial V_i$.
Recall that $\mathcal{D}(V_i\setminus t_i)$ is the disk complex of $V_i\setminus t_i$.
By Proposition~\ref{prop-g>1-hi}, there exist homeomorphisms $h_i:\partial V_i\setminus t_i\rightarrow \partial_- W_i\setminus s_i$ such that
\begin{equation}\label{eqn-n1-1}
d_{\partial_- W_1\setminus s_1}(\Phi_1(\alpha_1),h_1(\mathcal{D}^0(V_1\setminus t_1)))>3,
\end{equation}
\begin{equation}\label{eqn-n1-2}
d_{\partial_- W_2\setminus s_2}(\Phi_2(\alpha_0),h_2(\mathcal{D}^0(V_2\setminus t_2)))>3.
\end{equation}
%and that $\overline{h}_i|_{\partial V_i\cap t_i}$ corresponds to ${\rm id}_{\partial V_i\cap t_i}$, where 
Let $\overline{h}_i:(\partial V_i,\partial t_i)\rightarrow (\partial_- W_i,s_i\cap \partial_- W_i)$ be the homeomorphism of the pairs induced from $h_i$.
Let $(V_i^{\ast}, t_i^{\ast}):=(W_i,s_i)\cup_{\overline{h}_i}(V_i,t_i)$.
Then $(V_1^{\ast}, t_1^{\ast})\cup_{(F,P)} (V_2^{\ast}, t_2^{\ast})$ is a $(g,b)$-splitting of a link.
Let $\mathcal{D}_i$ be the set of essential disks in $V_i^{\ast}\setminus t_i^{\ast}$ for $i=1,2$. 
To show that $(V_1^{\ast}, t_1^{\ast})\cup_{(F,P)} (V_2^{\ast}, t_2^{\ast})$ has distance $1$ and is strongly keen, we prove the following.

\begin{assertion}\label{e1e2}
$E_1\cap E_2\ne \emptyset$ for any $E_1\in \mathcal{D}_1$ and $E_2\in \mathcal{D}_2$ with $(E_1,E_2)\ne (D_1,D_2)$.
%$E_1\cap E_2\ne \emptyset$ for any $E_1\in \mathcal{D}^0(V_1^{\ast}\setminus t_1^{\ast})$ and $E_2\in \mathcal{D}^0(V_2^{\ast}\setminus t_2^{\ast})$ with $(E_1,E_2)\ne (D_1,D_2)$.
\end{assertion}

To prove the above assertion, we divide $\mathcal{D}_i$ $(i=1,2)$ into four sets $\mathcal{D}_i^1$, $\mathcal{D}_i^2$, $\mathcal{D}_i^3$, $\mathcal{D}_i^4$, where
\begin{itemize}
\item $\mathcal{D}_i^1$ consists of the single disk $D_i$,
\item $\mathcal{D}_i^2$ consists of disks which are disjoint from $D_i$, not isotopic to $D_i$ and inessential in $(W_i'\cup_{\overline{h}_i} V_i)\setminus t_i^{\ast}$,
\item $\mathcal{D}_i^3$ consists of disks which are disjoint from $D_i$, not isotopic to $D_i$ and essential in $(W_i'\cup_{\overline{h}_i} V_i)\setminus t_i^{\ast}$,
\item $\mathcal{D}_i^4$ consists of disks which are not isotoped to be disjoint from $D_i$,
\end{itemize}
%and prove the following claims first.

By the proof of Proposition~\ref{prop-g>1-1}, we have the following two claims.
\begin{claim}\label{lema}
Any disk $E_i\in\mathcal{D}_i^2$ can be obtained by a band-sum of $D_i^0$ and $D_i^1$ along an arc (on $F_i$) ($i=1,2$).
\end{claim}

%\begin{proof}
%Since $E_i$ is an essential disk in $V_i^{\ast}\setminus t_i^{\ast}$ disjoint from $D_i$, not isotopic to $D_i$ and inessential in $(W_i'\cup_{\overline{h}_i} V_i)\setminus t_i^{\ast}$, we see that $\partial E_i$ is an essential simple closed curve on $F_i$ and is inessential on $F_i\cup D_i^0\cup D_i^1$.
%Hence, either $E_i$ can be obtained by a band-sum of $D_i^0$ and $D_i^1$ along an arc (on $F_i$) or $E_i$ bounds a disk in $F_i\cup D_i^0\cup D_i^1$ which contains exactly one of $D_i^0$ and $D_i^1$ and exactly one point of $s_i\cap F_i$.
%However, the latter case cannot occur, since $E_i\cap t_i^{\ast}=\emptyset$ and hence the both components of $(F_i\cup D_i^0\cup D_i^1)\setminus \partial E_i$ contains even number of the points $s_i\cap F_i$.
%\end{proof}

\begin{claim}\label{lemc}
Let $E_i$ be a disk in $\mathcal{D}_i^4$ such that $|E_i\cap D_i|$ is minimal.
Let $\Delta$ be the closure of a component of $E_i\setminus N_{W_i}(D_i)$ that is outermost in $E_i$. 
Then $\Delta$ is an essential disk in $(W_i'\cup_{\overline{h}_i} V_i)\setminus t_i^{\ast}$.
%either of the following holds:
%\begin{itemize}
%\item $\Delta$ is an essential disk in $(W_i'\cup_{\overline{h}_i} V_i)\setminus t_i^{\ast}$,
%\item there exists a component of ...
%\end{itemize}
\end{claim}

We also prove the following.

\begin{claim}\label{lemb}
For any $E_1\in\mathcal{D}_1^3$ such that $E_1\cap D_2\ne\emptyset$, there exist a component $\gamma_1$ of $\partial E_1\cap F_B$ and a subarc $\gamma_2$ of $\partial D_2^0$ such that  $\gamma_1\cup\gamma_2$ is an essential simple closed curve in $\partial_+ W_2'\setminus s_2$, which implies $\Phi_2(\gamma_1\cup\gamma_2)\ne\emptyset$.
%which does not cut off a once-punctured disk from $F_B$, which implies that there exists a subarc $\gamma_2$ of $\partial D_2^0$ such that $\gamma_1\cup\gamma_2$ is an essential simple closed curve in $(F_2\cup D_2^0\cup D_2^1)\setminus P$
\end{claim}

\begin{proof}
Assume on the contrary that there does not exist a component of $\partial E_1\cap F_B$ which together with a subarc of $\partial D_2^0$ forms an essential simple closed curve in $\partial_+ W_2'\setminus s_2$.
Then every component of $\partial E_1\cap F_B$ cuts off an annulus, a once-punctured annulus, or a once-punctured disk from $F_B$.
\begin{figure}[tb]
 \begin{center}
 \includegraphics[width=50mm]{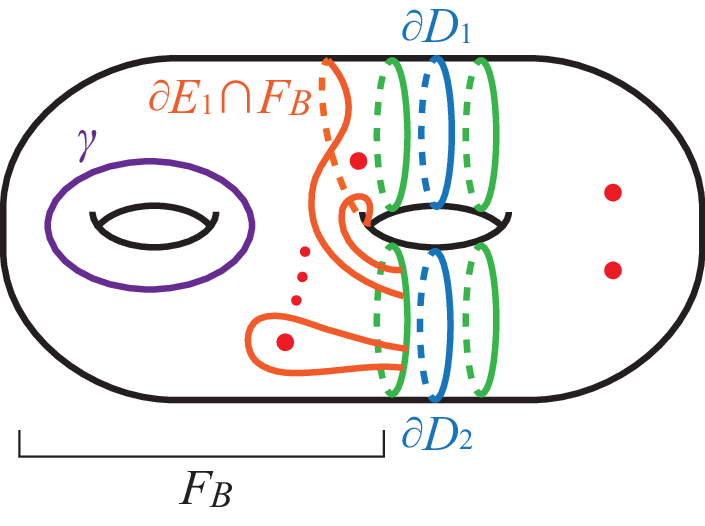}
 \end{center}
 \caption{$\partial E_1\cap F_B$ and $\gamma$.}
\label{fig-n1-2}
\end{figure}
Since the genus of $F_B$ is positive, there exists an essential simple closed curve $\gamma$ on $F_B$ disjoint from $\partial E_1\cap F_B$, and hence disjoint from $\partial E_1$ (see Figure~\ref{fig-n1-2}).
By Proposition~\ref{prop-g>1-1}, we have
$$
 d_{\partial_- W_1\setminus s_1}(\Phi_1(\gamma),h_1(\mathcal{D}^0(V_1\setminus t_1)))\le  1.
$$
Note also that $${\rm diam}_{\partial_- W_1\setminus s_1}(\Phi_1(\alpha_1),\Phi_1(\gamma))=d_{\partial_- W_1\setminus s_1}(\Phi_1(\alpha_1),\Phi_1(\gamma))=1$$ since $\alpha_1=\partial D_2$ and $\gamma$ are mutually disjoint simple closed curves on $F_1$. 
These imply
\begin{eqnarray*}
\begin{array}{rcl}
d_{\partial_- W_1\setminus s_1}(\Phi_1(\alpha_1),h_1(\mathcal{D}^0(V_1\setminus t_1)))&\le& {d}_{\partial_- W_1\setminus s_1}(\Phi_1(\alpha_1),\Phi_1(\gamma))\\&&+ d_{\partial_- W_1\setminus s_1}(\Phi_1(\gamma),h_1(\mathcal{D}^0(V_1\setminus t_1)))\\
&\le &1+1=2,
\end{array}
\end{eqnarray*}
a contradiction to the inequality (\ref{eqn-n1-1}).
\end{proof}

\begin{proof}[Proof of Assertion~\ref{e1e2}]
Suppose on the contrary that there exist $E_1\in \mathcal{D}_1$ and $E_2\in \mathcal{D}_2$ such that $(E_1,E_2)\ne (D_1,D_2)$ and $E_1\cap E_2= \emptyset$.
We may assume that $E_1\in \mathcal{D}_1^i$ and $E_2\in \mathcal{D}_2^j$ for some $i$ and $j$ such that $i\le j$ and $j\ne 1$, since the remaining cases can be treated similarly.
Assume that $|E_1\cap D_1|$ and $|E_2\cap D_2|$ are minimal.

\setcounter{case}{0}

\begin{case}\label{case8-1}
$E_1\in \mathcal{D}_1^1$, that is, $E_1=D_1$.
\end{case}

Then we divide Case~\ref{case8-1} into the following subcases.

\begin{subccase} 
$E_2\in \mathcal{D}_2^2$.
\end{subccase}

By Claim~\ref{lema}, $E_2$ is a band-sum of $D_2^0$ and $D_2^1$ along an arc on $F_2$.
Since $\partial D_1\cup \partial D_2$ is separating in $F$, the arc intersects $\partial D_1$.
Then we have $E_1\cap E_2=D_1\cap E_2\ne \emptyset$, a contradiction to the hypothesis.

\begin{subccase} 
$E_2\in \mathcal{D}_2^3$.
\end{subccase}

By Proposition~\ref{prop-g>1}, we have
$$
d_{\partial_- W_2\setminus s_2}(\Phi_2(\alpha_0),h_2(\mathcal{D}^0(V_2\setminus t_2))) \le 1,
$$
%\begin{eqnarray*}
%\begin{array}{rcl}
%d_{\partial_- W_2\setminus s_2}(\Phi_2(\alpha_1),h_2(\mathcal{D}^0(V_2\setminus t_2)))&\le& d_{\partial_- W_1\setminus s_1}(\Phi_2(\alpha_1),h_2(\partial E_2'))\\
%&\le &1,
%\end{array}
%\end{eqnarray*}
a contradiction to the inequality (\ref{eqn-n1-2}).

\begin{subccase} 
$E_2\in \mathcal{D}_2^4$.
\end{subccase}

Let $\Delta$ be the closure of a component of $E_2\setminus N_{W_2}(D_2)$ that is outermost in $E_2$.
We have $\Delta\cap \alpha_0=\emptyset$ since $E_2\cap \alpha_0=\emptyset$.
Also, $\Delta$ is essential in $(W_2'\cup_{\overline{h}_2}V_2)\setminus t_2^{\ast}$ by Claim~\ref{lemc}.
%Also, $\Delta$ is essential in $(W_2'\cap_{\overline{h}_2}V_2)\setminus t_2^{\ast}$ by the minimality of $|E_2\cap D_2|$ and the fact that $\alpha_1\cup \alpha_2$ is separating in $F$.
Then the union of $\Delta$ and one of the components of $D_2^0\setminus \Delta$ or $D_2^1\setminus \Delta$ is a disk which belongs to $\mathcal{D}_2^3$ and is disjoint from $D_1$.
This cannot occur as we have seen in the previous case.

\begin{case}\label{case8-2}
$E_1\in \mathcal{D}_1^2$.
\end{case}

By Claim~\ref{lema}, $E_1$ can be obtained by a band-sum of $D_1^0$ and $D_1^1$ along an arc $c$ on $F_1$.
Let $T_{E_1}$ be the 1-holed torus $(\partial D_1\times I)\cup N_{F_1}(c)$ bounded by $\partial E_1$ (see Figure~\ref{fig-n1-3}).
\begin{figure}[tb]
 \begin{center}
 \includegraphics[width=50mm]{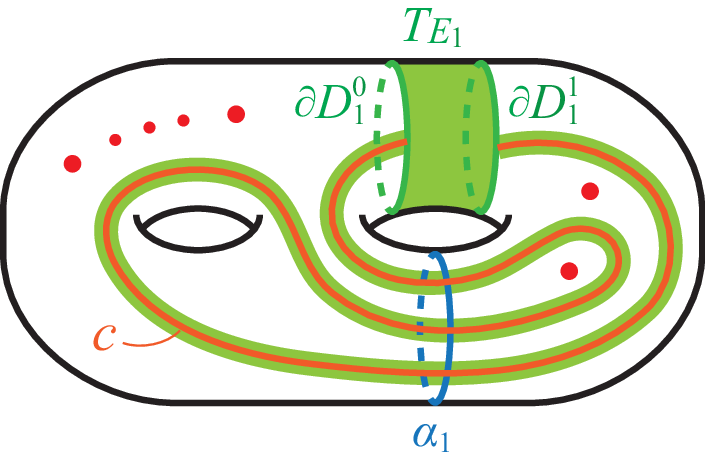}
 \end{center}
 \caption{$c$ and $T_{E_1}$.}
\label{fig-n1-3}
\end{figure}
Since $D_1\cap \alpha_1=\emptyset$ and $\alpha_0\cup \alpha_1$ is separating in $F$, we have $T_{E_1}\cap \alpha_1=N_{F_1}(c)\cap \alpha_1$.
%If $|N(c)\cap \alpha_2|=1$, then ${\rm cl}(T_{E_1}\setminus \alpha_2)$ is homeomorphic to $\partial D_1\times I$. If $|N(c)\cap \alpha_2|>1$, then ${\rm cl}(T_{E_1}\setminus \alpha_2)$ consists of a component homeomorphic to $\partial D_1\times I$ and disk components.
Hence, ${\rm cl}(T_{E_1}\setminus N_{F}(\alpha_1))$ consists of a component corresponding to $\partial D_1\times I$ and possibly some disk components.

Since $E_1\cap E_2=\emptyset$ by the hypothesis, either $\partial E_2\subset T_{E_1}$ or $\partial E_2\subset F\setminus T_{E_1}$ holds.
If $\partial E_2\subset F\setminus T_{E_1}$, then $E_2\cap D_1=\emptyset$, which is impossible by Case~\ref{case8-1}.
Hence, $\partial E_2\subset T_{E_1}$.

%Assume that $E_2\in \mathcal{D}_2^2$.
%Let $T_{E_2}\,(\subset F)$ be the 1-holed torus bounded by $\partial E_2$.
%Since $\partial E_2\subset T_{E_1}$ and $\partial E_2$ is separating, we have $\partial E_1=\partial E_2$.
%This implies $T_{E_1}=T_{E_2}$ since $T_{E_1}\cap P=T_{E_2}\cap P=\emptyset$.
%Note that $\alpha_1,\alpha_2\subset T_{E_1}(=T_{E_2})$, $\alpha_1\cap \alpha_2=\emptyset$, and $\alpha_1$, $\alpha_2$ are non-separating.
%These imply that $\alpha_1$ is isotopic to $\alpha_2$, a contradiction.

Then we divide Case~\ref{case8-2} into the following subcases.

\begin{subccase2} 
$E_2\in \mathcal{D}_2^2\cup \mathcal{D}_2^3$.
\end{subccase2}

In this case, $\partial E_2\cap \alpha_1=\emptyset$.
Since $\partial E_2\subset T_{E_1}$ by the argument in the previous paragraph, $\partial E_2$ is an essential simple closed curve on ${\rm cl}(T_{E_1}\setminus N_F(\alpha_1))$.
Recall that ${\rm cl}(T_{E_1}\setminus N_F(\alpha_1))$ consists of a component homeomorphic to $\partial D_1\times I$ and possibly disk components.
Hence, $\partial E_2$ is isotopic to $\partial D_1$, which is impossible by Case~\ref{case8-1}.

\begin{subccase2} 
$E_2\in \mathcal{D}_2^4$, that is $E_2\cap D_2\ne \emptyset$.
\end{subccase2}

Let $\Delta$ be the closure of a component of $E_2\setminus N_{W_2}(D_2)$ that is outermost in $E_2$.
We may assume that $\Delta\cap D_2^0\ne \emptyset$ and $\Delta\cap D_2^1=\emptyset$.
(The reader will see that the case when $\Delta\cap D_2^0=\emptyset$ and $\Delta\cap D_2^1\ne\emptyset$ can be treated in the same manner as below.)
Let $\overline{\Delta}$ be the union of $\Delta$ and one of the components of $D_2^0\setminus \Delta$.
Then $\partial \overline{\Delta}$ is the union of a subarc $\gamma_1$ of $\partial E_2$ and a subarc $\gamma_2$ of $\partial D_2^0$ (see Figure~\ref{fig-n1-4}).
\begin{figure}[tb]
 \begin{center}
 \includegraphics[width=105mm]{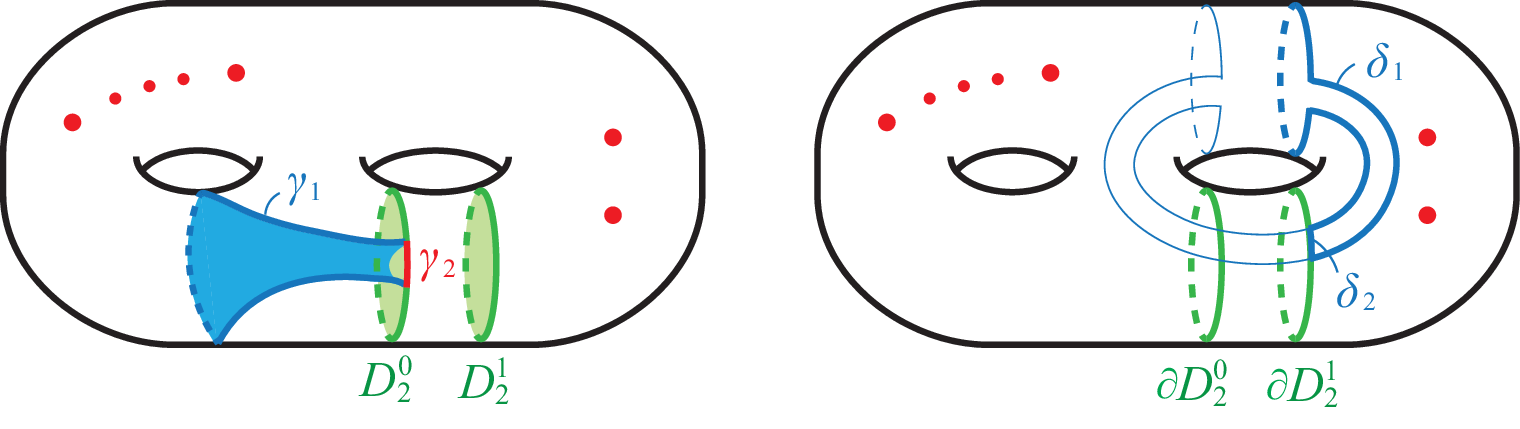}
 \end{center}
 \caption{$\gamma_1$, $\gamma_2$, $\delta_1$ and $\delta_2$.}
\label{fig-n1-4}
\end{figure}
Note that $\overline{\Delta}\in \mathcal{D}_2^3$ by Claim~\ref{lemc}.
%Note that $\overline{\Delta}\in \mathcal{D}_2^1\cup \mathcal{D}_2^2\cup \mathcal{D}_2^3$.
%
%Assume first that $\gamma_2$ is a subarc of $D_2^0$.
Recall that $E_1$ is a band-sum of $D_1^0$ and $D_1^1$.
Let $\delta_1$ be the closure of the component of $\partial E_1\setminus \partial D_2^1$ that contains a subarc of $\partial D_1^1$.
Then $\delta_1$ together with a subarc $\delta_2$ of $\partial D_2^1$ forms a simple closed curve isotopic to $\partial D_1$.
We have (i) $\gamma_1\cap \delta_1=\emptyset$ since $\gamma_1\subset\partial E_2$, $\delta_1\subset \partial E_1$ and $E_1\cap E_2=\emptyset$, 
(ii) $\gamma_1\cap \delta_2=\emptyset$ since $\gamma_1\cap D_2^1=\emptyset$ and $\delta_2\subset \partial D_2^1$, 
(iii) $\gamma_2\cap \delta_1=\emptyset$ since $\gamma_2\subset \partial D_2^0$ and $\delta_1\cap \partial D_2^0=\emptyset$, and 
(iv) $\gamma_2\cap \delta_2=\emptyset$ since $\gamma_2\subset \partial D_2^0$, $\delta_2\subset \partial D_2^1$ and $\partial D_2^0\cap \partial D_2^1=\emptyset$.
These imply that $\partial \overline{\Delta}\cap \partial D_1=\emptyset$, which is impossible by Case~\ref{case8-1}.

%Assume that $\gamma_2$ is a subarc of $D_2^1$.
%In this case, we can lead to a contradiction similarly by taking $\delta_1$ as the closure of a component of $\partial E_1\setminus \partial D_2^0$ disjoint from $\partial D_2^1$.

\begin{case}\label{case8-3}
$E_1\in \mathcal{D}_1^3$.
\end{case}

We may assume that $E_1\cap D_2\ne\emptyset$ since, otherwise, we may lead to a contradiction as in Case 1.2.
By Claim~\ref{lemb}, there exist a component $\gamma_1$ of $\partial E_1\cap F_B$ and a subarc $\gamma_2$ of $\partial D_2^0$ such that $\gamma_1\cup \gamma_2$ is an essential simple closed curve in $\partial_+ W_2'\setminus s_2$, which implies $\Phi_2(\gamma_1\cup \gamma_2)\ne\emptyset$.
Note that $\gamma_1\cup \gamma_2$ and $\alpha_0$ are simple closed curves on $F_2$, which are essential in $\partial_+ W_2\setminus s_2$. (See Figure~\ref{fig-n1-Phi}.)
Note also that $\alpha_0\cap \gamma_1\subset \partial D_1\cap \partial E_1=\emptyset$ and $\alpha_0\cap \gamma_2\subset \partial D_1\cap \partial D_2^0=\emptyset$.
Hence, $\Phi_2(\alpha_0)\cap \Phi_2(\gamma_1\cup\gamma_2)=\emptyset$ as seen in Figure~\ref{fig-n1-Phi}, and  we have
\begin{equation}\label{eq-case8-3-1}
d_{\partial_- W_2\setminus s_2}(\Phi_2(\alpha_0),\Phi_2(\gamma_1\cup\gamma_2))\le 1.
\end{equation}
\begin{figure}[tb]
 \begin{center}
 \includegraphics[width=110mm]{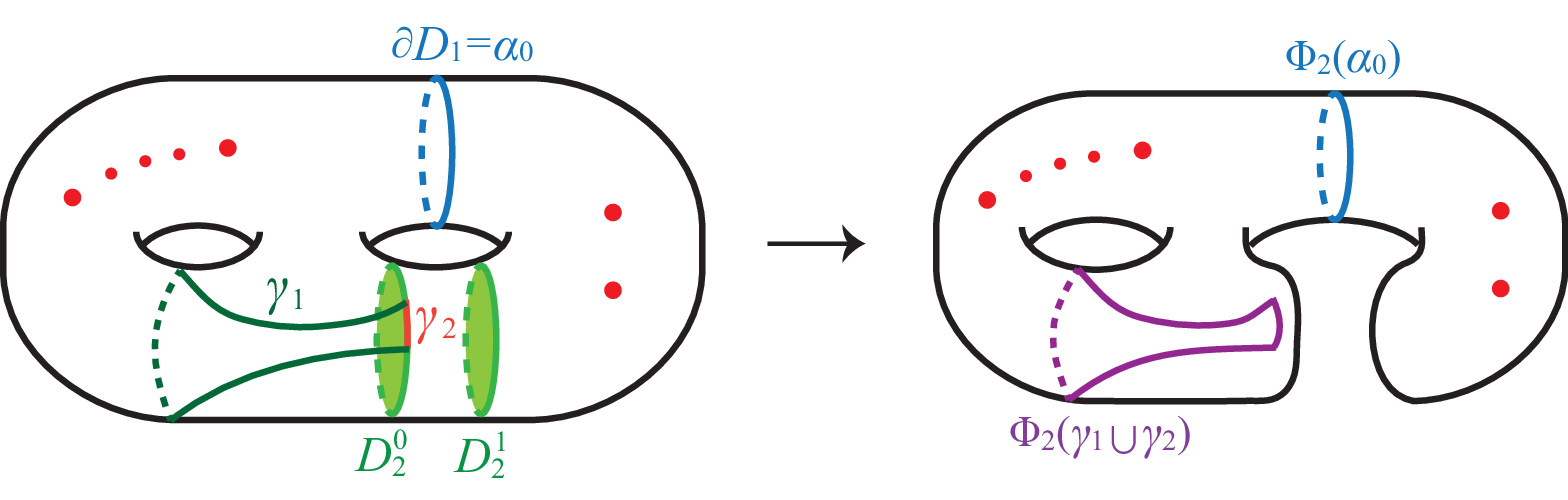}
 \end{center}
 \caption{$\Phi_2(\alpha_1)$ and $\Phi_2(\gamma_1\cup\gamma_2)$.}
\label{fig-n1-Phi}
\end{figure}

Then we divide Case~\ref{case8-3} into the following subcases.

\begin{subccase3}
$E_2\in \mathcal{D}_2^3$.
\end{subccase3}

Note that $\gamma_1\cap E_2\subset \partial E_1\cap E_2=\emptyset$ and $\gamma_2\cap E_2\subset \partial D_2^0\cap E_2=\emptyset$. Hence $(\gamma_1\cup \gamma_2)\cap E_2=\emptyset$.
By Proposition~\ref{prop-g>1-1}, we have
%\begin{equation}\label{eq-case8-3-2}
$$d_{\partial_- W_2\setminus s_2}(\Phi_2(\gamma_1\cup\gamma_2),h_2(\mathcal{D}^0(V_2\setminus t_2)))\le 1.$$
%\end{equation}
This together with the inequality (\ref{eq-case8-3-1}) implies that 
\begin{eqnarray*}
\begin{array}{rcl}
d_{\partial_- W_2\setminus s_2}(\Phi_2(\alpha_0), h_2(\mathcal{D}^0(V_2\setminus t_2)))&\le&d_{\partial_- W_2\setminus s_2}(\Phi_2(\alpha_0),\Phi_2(\gamma_1\cup\gamma_2))\\&&+ d_{\partial_- W_2\setminus s_2}(\Phi_2(\gamma_1\cup\gamma_2),h_2(\mathcal{D}^0(V_2\setminus t_2)))\\
&\le &1+1=2,
\end{array}
\end{eqnarray*}
a contradiction to the inequality (\ref{eqn-n1-2}).

\begin{subccase3} 
$E_2\in \mathcal{D}_2^4$.
\end{subccase3}

Let $\Delta$ be the closure of a component of $E_2\setminus N_{W_2}(D_2)$ that is outermost in $E_2$.
Then $\Delta$ is essential in $(W_2'\cup_{\overline{h}_2}V_2)\setminus t_2^{\ast}$ by Claim~\ref{lemc}.
%After boundary compressions toward $\partial_- W_2$ if necessary, we may assume that $\Delta':=\Delta\cap V_2$ consists of a single disk and $\Delta\cap W_2^1$ is a vertical annulus in $W_2^1(\cong \partial_- W_2\times I)$.

If $\Delta\cap D_2^0=\emptyset$, then $(\gamma_1\cup\gamma_2)\cap \Delta\subset (\gamma_1\cap\Delta)\cup(\gamma_2\cap\Delta)\subset (\partial E_1\cap E_2)\cup (\partial D_2^0\cap \Delta)=\emptyset$ (see Figure~\ref{fig-n1-5}).
Hence, by the inequality (\ref{eq-case8-3-1}) and Proposition~\ref{prop-g>1-1} (B1), we have
\begin{eqnarray*}
\begin{array}{rcl}
d_{\partial_- W_2\setminus s_2}(\Phi_2(\alpha_0),h_2(\mathcal{D}^0(V_2\setminus t_2)))&\le&d_{\partial_- W_2\setminus s_2}(\Phi_2(\alpha_0),\Phi_2(\gamma_1\cup\gamma_2))\\&&+ d_{\partial_- W_2\setminus s_2}(\Phi_2(\gamma_1\cup\gamma_2),h_2(\mathcal{D}^0(V_2\setminus t_2)))\\
&\le &1+1=2,
\end{array}
\end{eqnarray*}
a contradiction to the inequality (\ref{eqn-n1-2}).
\begin{figure}[tb]
 \begin{center}
 \includegraphics[width=50mm]{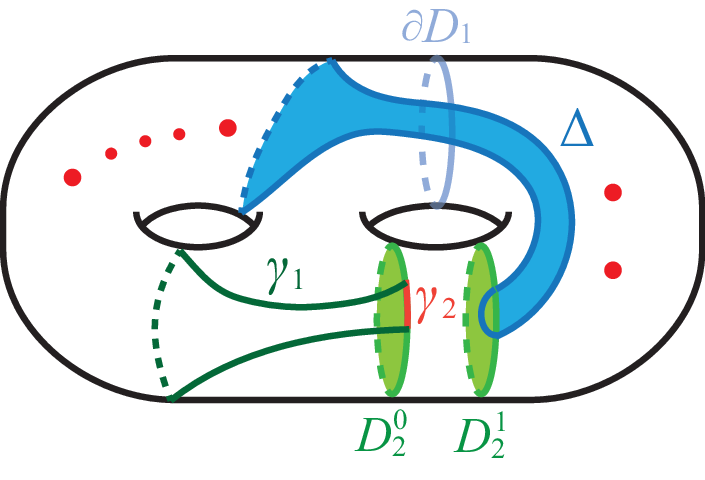}
 \end{center}
 \caption{$\gamma_1$, $\gamma_2$ and $\Delta$.}
\label{fig-n1-5}
\end{figure}

If $\Delta\cap D_2^0\ne\emptyset$, then $|(\gamma_1\cup\gamma_2)\cap \Delta|=|\gamma_2\cap \Delta|$ since $\gamma_1\subset \partial E_1$, $\Delta\subset E_2$ and $E_1\cap E_2=\emptyset$.
Also we may suppose $|\gamma_2\cap \Delta|\le 1$, by replacing the subarc $\gamma_2$ of $\partial D_2^0$ with the closure of $\partial D_2^0\setminus \gamma_2$ if necessary, since $\Delta\cap D_2^0$ is an arc properly embedded in the disk $D_2^0$ (see Figure~\ref{fig-n1-5-2}).
\begin{figure}[tb]
 \begin{center}
 \includegraphics[width=105mm]{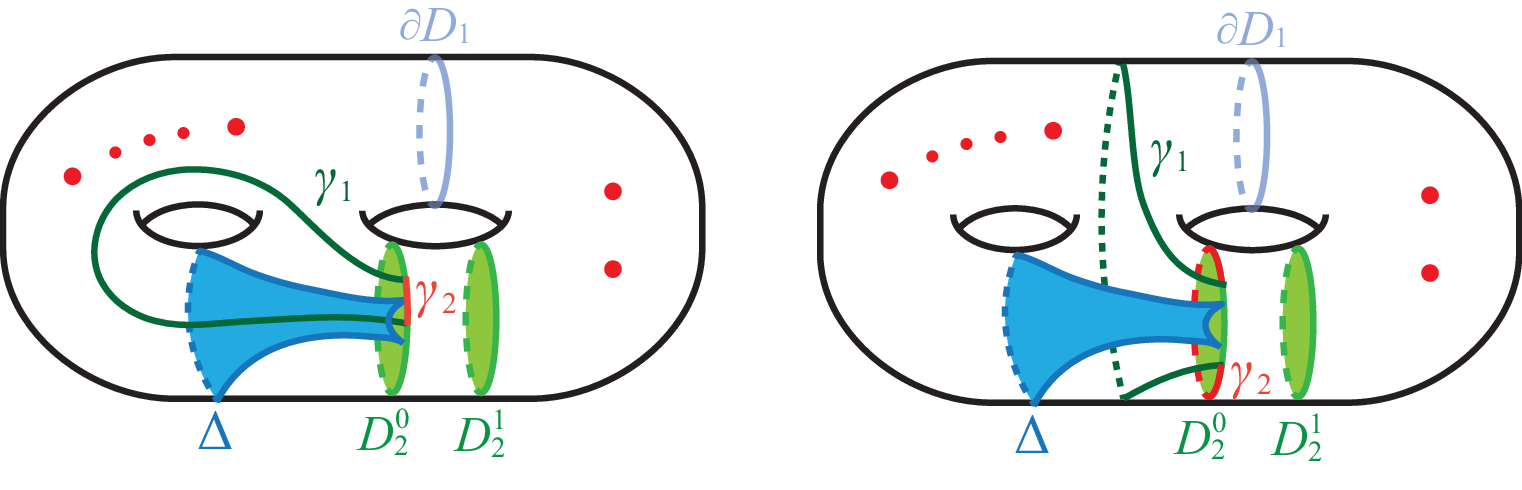}
 \end{center}
 \caption{$\gamma_1$, $\gamma_2$ and $\Delta$.}
\label{fig-n1-5-2}
\end{figure}
Hence, by the inequality (\ref{eq-case8-3-1}) and Proposition~\ref{prop-g>1-1} (B), we have
\begin{eqnarray*}
\begin{array}{rcl}
d_{\partial_- W_2\setminus s_2}(\Phi_2(\alpha_0),h_2(\mathcal{D}^0(V_2\setminus t_2)))&\le&d_{\partial_- W_2\setminus s_2}(\Phi_2(\alpha_0),\Phi_2(\gamma_1\cup\gamma_2))\\&&+ d_{\partial_- W_2\setminus s_2}(\Phi_2(\gamma_1\cup\gamma_2),h_2(\mathcal{D}^0(V_2\setminus t_2)))\\
&\le &1+2=3,
\end{array}
\end{eqnarray*}
a contradiction to the inequality (\ref{eqn-n1-2}).

\begin{case}\label{case8-4}
$E_1\in \mathcal{D}_1^4$.
\end{case}

In this case, $E_2\in \mathcal{D}_2^4$.

Let $\Delta_1$ be the closure of a component of $E_1\setminus D_1$ that is outermost in $E_1$, and let $\overline{\Delta}_1$ be the union of $\Delta_1$ and one of the component of $D_1\setminus \Delta_1$.
Then $\overline{\Delta}_1\in \mathcal{D}_1^3$ by Claim~\ref{lemc}.
If $\overline{\Delta}_1\cap D_2=\emptyset$, then we may lead to a contradiction as in Case 1.2.
Hence, we may assume that $\overline{\Delta}_1\cap D_2\ne\emptyset$.
By Claim~\ref{lemb}, there exist a component $\gamma_1$ of $\partial \overline{\Delta}_1\cap F_B$ and a subarc $\gamma_2$ of $\partial D_2^0$ such that $\gamma_1\cup \gamma_2$ is an essential simple closed curve in $\partial_+ W_2'\setminus s_2$, which implies $\Phi_2(\gamma_1\cup \gamma_2)\ne\emptyset$ (see Figure~\ref{fig-n1-6}).
\begin{figure}[tb]
 \begin{center}
 \includegraphics[width=105mm]{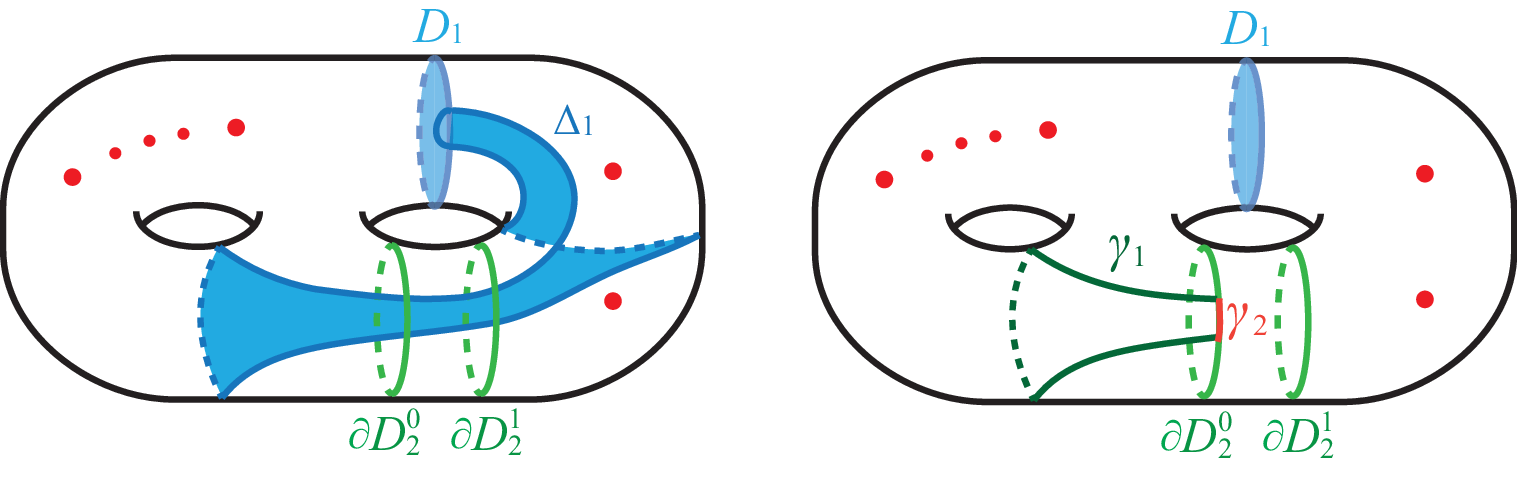}
 \end{center}
 \caption{$\Delta_1$, $\gamma_1$ and $\gamma_2$.}
\label{fig-n1-6}
\end{figure}
Note that $\gamma_1\cup \gamma_2$ and $\alpha_0$ are simple closed curves on $F_2$.
Note also that $\alpha_0\cap \gamma_1\subset \partial D_1\cap \partial F_B=\emptyset$ and $\alpha_0\cap \gamma_2\subset \partial D_1\cap \partial D_2^0=\emptyset$.
Hence, we have
\begin{equation}\label{eq-case8-4-1}
d_{\partial_- W_2\setminus s_2}(\Phi_2(\alpha_0),\Phi_2(\gamma_1\cup\gamma_2))\le 1.
\end{equation}
Let $\Delta_2$ be the closure of a component of $E_2\setminus N_{W_2}(D_2)$ that is outermost in $E_2$.
Then $\Delta_2$ is essential in $(W_2'\cup_{\overline{h}_2}V_2)\setminus t_2^{\ast}$ by Claim~\ref{lemc}.

If $\Delta_2\cap D_2^0=\emptyset$, then $(\gamma_1\cup\gamma_2)\cap \Delta_2\subset (\partial E_1\cap E_2)\cup (\partial D_2^0\cap \Delta_2)=\emptyset$.
Then, by the inequality (\ref{eq-case8-4-1}) and Proposition~\ref{prop-g>1-1} (B1), we have
\begin{eqnarray*}
\begin{array}{rcl}
d_{\partial_- W_2\setminus s_2}(\Phi_2(\alpha_0),h_2(\mathcal{D}^0(V_2\setminus t_2)))&\le&d_{\partial_- W_2\setminus s_2}(\Phi_2(\alpha_0),\Phi_2(\gamma_1\cup\gamma_2))\\&&+ d_{\partial_- W_2\setminus s_2}(\Phi_2(\gamma_1\cup\gamma_2),h_2(\mathcal{D}^0(V_2\setminus t_2)))\\
&\le &1+1=2,
\end{array}
\end{eqnarray*}
a contradiction to the inequality (\ref{eqn-n1-2}).

If $\Delta_2\cap D_2^0\ne\emptyset$, then $|(\gamma_1\cup\gamma_2)\cap \Delta_2|=|\gamma_2\cap \Delta_2|$ since $\gamma_1\subset \partial E_1$, $\Delta_2\subset E_2$ and $E_1\cap E_2=\emptyset$.
Also we may suppose $|\gamma_2\cap \Delta_2|\le 1$, by replacing the subarc $\gamma_2$ of $\partial D_2^0$ with the closure of $\partial D_2^0\setminus \gamma_2$ if necessary, since $\Delta_2\cap D_2^0$ is an arc properly embedded in the disk $D_2^0$.
Then, by the inequality (\ref{eq-case8-4-1}) and Proposition~\ref{prop-g>1-1} (B), we have
\begin{eqnarray*}
\begin{array}{rcl}
d_{\partial_- W_2\setminus s_2}(\Phi_2(\alpha_0),h_2(\mathcal{D}^0(V_2\setminus t_2)))&\le&d_{\partial_- W_2\setminus s_2}(\Phi_2(\alpha_0),\Phi_2(\gamma_1\cup\gamma_2))\\&&+ d_{\partial_- W_2\setminus s_2}(\Phi_2(\gamma_1\cup\gamma_2),h_2(\mathcal{D}^0(V_2\setminus t_2)))\\
&\le &1+2=3,
\end{array}
\end{eqnarray*}
a contradiction to the inequality (\ref{eqn-n1-2}).

This completes the proof of Assertion~\ref{e1e2}.

\end{proof}

\section{Proof of Theorem \ref{thm-1} when $n=1$ and $g=1$}\label{sec-proof-7}

We first show the next proposition, whose proof is due to Saito \cite{Sai}.

\begin{proposition}\label{prop-111}
Let $(V_1,t_1)\cup_{(F,P)}(V_2,t_2)$ be a $(1,1)$-splitting of a knot.
If the distance of $(V_1,t_1)\cup_{(F,P)}(V_2,t_2)$ is $1$, then it must be strongly keen.
\end{proposition}

\begin{proof}
Assume that the distance of $(V_1,t_1)\cup_{(F,P)}(V_2,t_2)$ is $1$, 
and let $x$ and $y$ be mutually disjoint essential simple closed curves in $F\setminus P$ which bound disks in $V_1\setminus t_1$ and $V_2\setminus t_2$, respectively.
By \cite[Proof of Theorem 2.3]{Sai}, $x$ and $y$ must bound so-called $\varepsilon_0$-disks (in fact, these disks are denoted by $D_i^c$ in Section~\ref{sec-proof-5}) in $(V_1,t_1)$ and $(V_2,t_2)$, respectively, which are unique up to isotopy by \cite[Lemma 3.4]{Sai}.
Hence, $(V_1,t_1)\cup_{(F,P)}(V_2,t_2)$ is strongly keen.
\end{proof}

In fact, it is shown that the distance of $(V_1,t_1)\cup_{(F,P)}(V_2,t_2)$ is $1$ if and only if the ambient manifold is $S^2\times S^1$ and the knot is a core knot (see \cite[Theorem 2.3]{Sai}).

In the remainder of this section, we give a proof of Theorem \ref{thm-1} for the case when $n=1$, $g=1$ and $b\ge 2$.

Let $F$ be a torus and let $P$ be the union of $2b$ points on $F$.
Let $\alpha_0$ and $\alpha_1$ be simple closed curves on $F\setminus P$ such that $\alpha_0\cap\alpha_1=\emptyset$ and that $\alpha_0\cup\alpha_1$ cuts off two twice-punctured disks from $F\setminus P$ which are disjoint to each other.
For $i=1,2$, let $V_i^{\ast, 0}$, $t_i^{\ast,0}$, $V_i$, $t_i$, $W_i$,  $W_i^1$, $\partial_- W_i$, $s_i$, $D_i$, $F_i$, $\Phi_i$ be as in Subsection~\ref{sec-b>1}.
Identify $(\partial_+W_1, s_1\cap\,\partial_+W_1)$ and $(\partial_+W_2, s_2\cap\,\partial_+W_2)$ with $(F,P)$ so that $\partial D_1=\alpha_0$ and $\partial D_2=\alpha_1$. %and that $t_1^{\ast,0}\cup t_2^{\ast,0}$ consists of $c$ components.

%For $i=1,2$, let $V_i$ be a genus-$(g-1)$ handlebody and $t_i$ be the union of $b$ arcs properly embedded in $V_i$ which is parallel to $\partial V_i$.
%Recall that $\mathcal{D}(V_i\setminus t_i)$ is the disk complex of $V_i\setminus t_i$.
%
By Proposition~\ref{prop-b>1-hi}, there exist homeomorphisms $h_i:\partial V_i\setminus t_i\rightarrow \partial_- W_i\setminus s_i$ such that
\begin{equation}\label{eqn-n1-3}
d_{\partial_- W_1\setminus s_1}(\Phi_1(\alpha_1),h_1(\mathcal{D}^0(V_1\setminus t_1)))>3,
\end{equation}
\begin{equation}\label{eqn-n1-4}
d_{\partial_- W_2\setminus s_2}(\Phi_2(\alpha_0),h_2(\mathcal{D}^0(V_2\setminus t_2)))>3.
\end{equation}
Let $\overline{h}_i:(\partial V_i,\partial t_i)\rightarrow (\partial_- W_i,s_i\cap \partial_- W_i)$ be the homeomorphism of the pairs induced from $h_i$, and let $(V_i^{\ast}, t_i^{\ast}):=(W_i,s_i)\cup_{\overline{h}_i}(V_i,t_i)$.
Then $(V_1^{\ast}, t_1^{\ast})\cup_{(F,P)} (V_2^{\ast}, t_2^{\ast})$ is a $(g,b)$-splitting of a link.
Let $\mathcal{D}_i$ be the set of essential disks in $V_i^{\ast}\setminus t_i^{\ast}$ for $i=1,2$. 
To show that $(V_1^{\ast}, t_1^{\ast})\cup_{(F,P)} (V_2^{\ast}, t_2^{\ast})$ has distance $1$ and is strongly keen, we prove the following.

\begin{assertion}\label{e1e2-2}
$E_1\cap E_2\ne \emptyset$ for any $E_1\in \mathcal{D}_1$ and $E_2\in \mathcal{D}_2$ with $(E_1,E_2)\ne (D_1,D_2)$.
\end{assertion}

To prove the above assertion, we divide $\mathcal{D}_i$ $(i=1,2)$ into three sets $\mathcal{D}_i^1$, $\mathcal{D}_i^2$, $\mathcal{D}_i^3$, where
\begin{itemize}
\item $\mathcal{D}_i^1$ consists of the single disk $D_i$,
\item $\mathcal{D}_i^2$ consists of disks which are disjoint from $D_i$, not isotopic to $D_i$,
\item $\mathcal{D}_i^3$ consists of disks which are not isotoped to be disjoint from $D_i$.
\end{itemize}
Since $D_i$ is separating in $V_i^{\ast}\setminus t_i^{\ast}$, we have the following claim.
\begin{claim}\label{claim_10_3}
{\rm (1)} Every $E\in \mathcal{D}_i^2$ is essential in $(W_i^1\cup_{\overline{h}_i} V_i)\setminus t_i^{\ast}$.

{\rm (2)} For any $E\in \mathcal{D}_i^3$, the closure of every component of $E\setminus D_i$ that is outermost in $E$ is an essential disk in $(W_i^1\cup_{\overline{h}_i} V_i)\setminus t_i^{\ast}$, provided $|E\cap D_i|$ is minimal.
\end{claim}
Also, the next claim can be obtained by arguments similar to those for Claim~\ref{lemb} since the genus of $F_2$ is $1\,(> 0)$.

\begin{claim}\label{lemb-2}
For any $E_1\in\mathcal{D}_1^2$ such that $E_1\cap D_2\ne\emptyset$, there exist a component $\gamma_1$ of $\partial E_1\cap F_2$ and a subarc $\gamma_2$ of $\partial D_2$ such that  $\gamma_1\cup\gamma_2$ is an essential simple closed curve in $(\partial W_2^1\setminus \partial_-W_2)\setminus s_2$, which implies $\Phi_2(\gamma_1\cup\gamma_2)\ne\emptyset$.
\end{claim}

\begin{proof}[Proof of Assertion~\ref{e1e2-2}]
Suppose on the contrary that there exist $E_1\in \mathcal{D}_1$ and $E_2\in \mathcal{D}_2$ such that $(E_1,E_2)\ne (D_1,D_2)$ and $E_1\cap E_2= \emptyset$.
We may assume that $E_1\in \mathcal{D}_1^i$ and $E_2\in \mathcal{D}_2^j$ for some $i$ and $j$ such that $i\le j$ and $j\ne 1$, since the remaining cases can be treated similarly.
Assume that $|E_1\cap D_1|$ and $|E_2\cap D_2|$ are minimal.

\setcounter{case}{0}

\begin{case}\label{case9-1}
$E_1\in \mathcal{D}_1^1$, that is, $E_1=D_1$.
\end{case}

In this case, $E_2\in \mathcal{D}_2^2\cup \mathcal{D}_2^3$.
If $E_2\in \mathcal{D}_2^2$, then let $\Delta:=E_2$.
If $E_2\in \mathcal{D}_2^3$, then let $\Delta$ be the closure of a component of $E_2\setminus D_2$ that is outermost in $E_2$.
Note that $\Delta$ is an essential disk in $(W_i^1\cup_{\overline{h}_i} V_i)\setminus t_i^{\ast}$ by Claim~\ref{claim_10_3}.
We have $\Delta\cap \alpha_0=\emptyset$ since $E_2\cap \alpha_0=\emptyset$.
Then, by Proposition~\ref{prop-b>1-1} (1), we have
$$
d_{\partial_- W_2\setminus s_2}(\Phi_2(\alpha_0),h_2(\mathcal{D}^0(V_2\setminus t_2))) \le 1,
$$
a contradiction to the inequality (\ref{eqn-n1-4}).

\begin{case}\label{case9-2}
$E_1\in \mathcal{D}_1^2$.
\end{case}

We may assume that $E_1\cap D_2\ne\emptyset$ since, otherwise, the fact leads to a contradiction as in Case 1.
By Claim~\ref{lemb-2}, there exist a component $\gamma_1$ of $\partial E_1\cap F_2$ and a subarc $\gamma_2$ of $\partial D_2$ such that $\gamma_1\cup \gamma_2$ is an essential simple closed curve in $(\partial W_2^1\setminus \partial_-W_2)\setminus s_2$, which implies $\Phi_2(\gamma_1\cup \gamma_2)\ne\emptyset$ (see Figure~\ref{fig-n1-7}).
\begin{figure}[tb]
 \begin{center}
 \includegraphics[width=45mm]{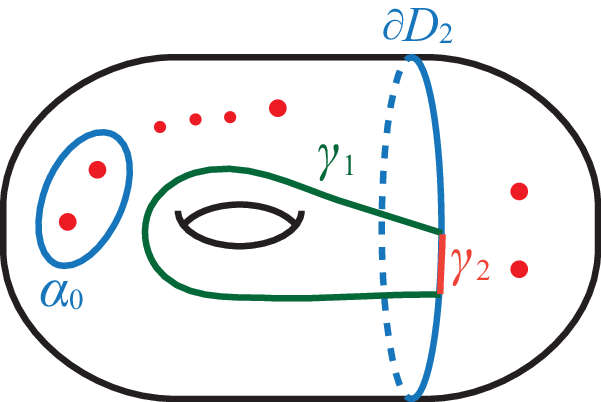}
 \end{center}
 \caption{$\gamma_1$, $\gamma_2$ and $\alpha_0(=\partial D_1)$.}
\label{fig-n1-7}
\end{figure}
Note that $\gamma_1\cup \gamma_2$ and $\alpha_0(=\partial D_1)$ are essential simple closed curves on $F_2$.
Note also that $\alpha_0\cap (\gamma_1\cup\gamma_2)\subset (\alpha_0\cap \gamma_1)\cup(\alpha_0\cap \gamma_2)\subset (\alpha_0\cap \partial E_1)\cup(\alpha_0\cap \partial D_2)=\emptyset$.
%$\alpha_0\cap \gamma_1\subset \partial D_1\cap \partial E_1=\emptyset$ and $\alpha_0\cap \gamma_2\subset \partial D_1\cap \partial D_2=\emptyset$.
Hence, we have
\begin{equation}\label{eq-case9-2-1}
d_{\partial_- W_2\setminus s_2}(\Phi_2(\alpha_0),\Phi_2(\gamma_1\cup\gamma_2))\le 1.
\end{equation}

Then we divide Case~\ref{case9-2} into the following subcases.

\begin{subcccase2} 
$E_2\in \mathcal{D}_2^2$.
\end{subcccase2}

%After boundary compressions toward $\partial_- W_2$ if necessary, we may assume that $E_2':=E_2\cap V_2$ consists of a single disk and $E_2\cap W_2^1$ is a vertical annulus in $W_2^1(\cong \partial_- W_2\times I)$.
Note that $(\gamma_1\cup \gamma_2)\cap E_2\subset (\partial E_1\cap E_2)\cup (\partial D_2\cap E_2)=\emptyset$.
%Note that $\gamma_1\cap E_2\subset \partial E_1\cap E_2=\emptyset$ and $\gamma_2\cap E_2\subset \partial D_2\cap E_2=\emptyset$.
Hence, by the inequality (\ref{eq-case9-2-1}) and Proposition~\ref{prop-b>1} (with regarding $\gamma_1\cup\gamma_2$ as $\alpha$, and $\partial E_2$ as $\beta$), we have
\begin{eqnarray*}
\begin{array}{rcl}
d_{\partial_- W_2\setminus s_2}(\Phi_2(\alpha_0),h_2(\mathcal{D}^0(V_2\setminus t_2)))&\le&d_{\partial_- W_2\setminus s_2}(\Phi_2(\alpha_0),\Phi_2(\gamma_1\cup\gamma_2))\\&&+ d_{\partial_- W_2\setminus s_2}(\Phi_2(\gamma_1\cup\gamma_2),h_2(\mathcal{D}^0(V_2\setminus t_2)))\\
&\le &1+1=2,
\end{array}
\end{eqnarray*}
a contradiction to the inequality (\ref{eqn-n1-4}).

\begin{subcccase2} 
$E_2\in \mathcal{D}_2^3$.
\end{subcccase2}

Let $\Delta$ be the closure of a component of $E_2\setminus D_2$ that is outermost in $E_2$.
%Then $\Delta$ is essential in $(W_2'\cup_{\overline{h}_2}V_2)\setminus t_2^{\ast}$.
%After boundary compressions toward $\partial_- W_2$ if necessary, we may assume that $\Delta':=\Delta\cap V_2$ consists of a single disk and $\Delta\cap W_2^1$ is a vertical annulus in $W_2^1(\cong \partial_- W_2\times I)$.
Note that $|(\gamma_1\cup\gamma_2)\cap \Delta|=|\gamma_2\cap \Delta|$ since $\gamma_1\cap \Delta\subset \partial E_1\cap E_2=\emptyset$.
We may suppose that $|\gamma_2\cap \Delta|\le 1$, by replacing the subarc $\gamma_2$ if necessary, since $\Delta\cap D_2$ is an arc properly embedded in the disk $D_2$ (cf. Figure~\ref{fig-n1-5-2}).
Hence, by the inequality (\ref{eq-case9-2-1}) and Proposition~\ref{prop-b>1-1} (1) or (2) (with regarding $\gamma_1\cup\gamma_2$ as $\alpha$), we have
\begin{eqnarray*}
\begin{array}{rcl}
d_{\partial_- W_2\setminus s_2}(\Phi_2(\alpha_0),h_2(\mathcal{D}^0(V_2\setminus t_2)))&\le&d_{\partial_- W_2\setminus s_2}(\Phi_2(\alpha_0),\Phi_2(\gamma_1\cup\gamma_2))\\&&+ d_{\partial_- W_2\setminus s_2}(\Phi_2(\gamma_1\cup\gamma_2),h_2(\mathcal{D}^0(V_2\setminus t_2)))\\
&\le &1+2=3,
\end{array}
\end{eqnarray*}
a contradiction to the inequality (\ref{eqn-n1-4}).

\begin{case}\label{case9-3}
$E_1\in \mathcal{D}_1^3$.
\end{case}

In this case, $E_2\in \mathcal{D}_2^3$.

Let $\Delta_1$ be the closure of a component of $E_1\setminus D_1$ that is outermost in $E_1$, and let $\overline{\Delta}_1$ be the union of $\Delta_1$ and one of the component of $D_1\setminus \Delta_1$.
By the minimality of $|E_1\cap D_1|$, $\Delta_1$ and hence $\overline{\Delta}_1$ are essential disks in $(W_1^1\cup_{\overline{h}_1} V_1)\setminus t_1^{\ast}$ by Claim~\ref{claim_10_3}, and hence,
%Then by the arguments in the second paragraph from the last (where it is shown that $P_i(\beta'')$ is essential in $\partial_- W_i$) of the proof of Proposition~\ref{prop-b>1}, we obtain 
$\overline{\Delta}_1\in \mathcal{D}_1^2$.
If $\overline{\Delta}_1\cap D_2=\emptyset$, then we may lead to a contradiction as in Case 1. 
Hence, we may assume that $\overline{\Delta}_1\cap D_2\ne\emptyset$.
By Claim~\ref{lemb-2}, we see that there exist a component $\gamma_1$ of $\partial \overline{\Delta}_1\cap F_2$ and a subarc $\gamma_2$ of $\partial D_2$ such that $\gamma_1\cup \gamma_2$ is an essential simple closed curve in $(\partial W_2^1\setminus \partial_-W_2)\setminus s_2$, which implies $\Phi_2(\gamma_1\cup \gamma_2)\ne\emptyset$.
Note that $\gamma_1\cup \gamma_2$ and $\alpha_0$ are essential simple closed curves on $F_2$.
Note also that $\alpha_0\cap (\gamma_1\cup\gamma_2)=(\alpha_0\cap\gamma_1)\cup(\alpha_0\cap\gamma_2)\subset (\alpha_0\cap {\rm int}(\Delta_1\cap F_1))\cup (\alpha_0\cap \alpha_1)=\emptyset$.
%$\alpha_0\cap \gamma_1\subset \partial D_1\cap \partial E_1=\emptyset$ and $\alpha_0\cap \gamma_2\subset \partial D_1\cap \partial D_2=\emptyset$.
Hence, we have
\begin{equation}\label{eq-case9-3-1}
d_{\partial_- W_2\setminus s_2}(\Phi_2(\alpha_0),\Phi_2(\gamma_1\cup\gamma_2))\le 1.
\end{equation}

Let $\Delta_2$ be the closure of a component of $E_2\setminus N(D_2)$ that is outermost in $E_2$.
Then by the minimality of $|E_2\cap D_2|$, $\Delta_2$ is essential in $(W_2^1\cup_{\overline{h}_2}V_2)\setminus t_2^{\ast}$ by Claim~\ref{claim_10_3}.
%After boundary compressions toward $\partial_- W_2$ if necessary, we may assume that $\Delta_2':=\Delta_2\cap V_2$ consists of a single disk and $\Delta_2\cap W_2^1$ is a vertical annulus in $W_2^1(\cong \partial_- W_2\times I)$.
Note that $|(\gamma_1\cup\gamma_2)\cap \Delta_2|=|\gamma_2\cap \Delta_2|\le 1$.
Hence, by the inequality (\ref{eq-case9-3-1}) and Proposition~\ref{prop-b>1-1}, we have
\begin{eqnarray*}
\begin{array}{rcl}
d_{\partial_- W_2\setminus s_2}(\Phi_2(\alpha_0),h_2(\mathcal{D}^0(V_2\setminus t_2)))&\le&d_{\partial_- W_2\setminus s_2}(\Phi_2(\alpha_0),\Phi_2(\gamma_1\cup\gamma_2))\\&&+ d_{\partial_- W_2\setminus s_2}(\Phi_2(\gamma_1\cup\gamma_2),h_2(\mathcal{D}^0(V_2\setminus t_2)))\\
&\le &1+2=3,
\end{array}
\end{eqnarray*}
a contradiction to the inequality (\ref{eqn-n1-4}).

This completes the proof of Assertion~\ref{e1e2-2}.
\end{proof}

\section{Proof of Theorem \ref{thm-1} when $n=1$ and $g=0$}\label{sec-proof-8}

In this section, we give a proof of Theorem \ref{thm-1} for the case when $n=1$ and $g=0$.
Note that $b\ge 4$.

Let $F$ be a $2$-sphere and let $P$ be the union of $2b$ points on $F$.
Let $\alpha_0$ and $\alpha_1$ be simple closed curves on $F\setminus P$ such that $\alpha_0\cap\alpha_1=\emptyset$ and that $\alpha_0\cup\alpha_1$ cuts off two twice-punctured disks from $F\setminus P$ which are disjoint to each other.
%each of $\alpha_0$ and $\alpha_1$ cuts off a twice-punctured disk from $F\setminus P$.
For $i=1,2$, let $V_i^{\ast, 0}$, $t_i^{\ast,0}$, $V_i$, $t_i$, $W_i$, $W_i^1$, $s_i$, $D_i$, $F_i$, $\Phi_i$ be as in Subsection~\ref{sec-b>1}.
Identify $(\partial_+W_1, s_1\cap\,\partial_+W_1)$ and $(\partial_+W_2, s_2\cap\,\partial_+W_2)$ with $(F,P)$ so that $\partial D_1=\alpha_0$ and $\partial D_2=\alpha_1$. %and that $t_1^{\ast,0}\cup t_2^{\ast,0}$ consists of $c$ components.

%For $i=1,2$, let $V_i$ be the 2-sphere and $t_i$ be the union of $b$ arcs properly embedded in $V_i$ which is parallel to $\partial V_i$.
%Recall that $\mathcal{D}(V_i\setminus t_i)$ is the disk complex of $V_i\setminus t_i$.
%
By Proposition~\ref{prop-b>1-hi}, there exist homeomorphisms $h_i:\partial V_i\setminus t_i\rightarrow \partial_- W_i\setminus s_i$ such that
\begin{equation}\label{eqn-n1-5}
d_{\partial_- W_1\setminus s_1}(\Phi_1(\alpha_1),h_1(\mathcal{D}^0(V_1\setminus t_1)))>3,
\end{equation}
\begin{equation}\label{eqn-n1-6}
d_{\partial_- W_2\setminus s_2}(\Phi_2(\alpha_0),h_2(\mathcal{D}^0(V_2\setminus t_2)))>3.
\end{equation}
Let $\overline{h}_i:(\partial V_i,\partial t_i)\rightarrow (\partial_- W_i,s_i\cap \partial_- W_i)$ be the homeomorphism of the pairs induced from $h_i$.
Let $(V_i^{\ast}, t_i^{\ast}):=(W_i,s_i)\cup_{\overline{h}_i}(V_i,t_i)$.
Then $(V_1^{\ast}, t_1^{\ast})\cup_{(F,P)} (V_2^{\ast}, t_2^{\ast})$ is a $(0,b)$-splitting of a link.
Let $\mathcal{D}_i$ be the set of essential disks in $V_i^{\ast}\setminus t_i^{\ast}$ for $i=1,2$. 
To show that $(V_1^{\ast}, t_1^{\ast})\cup_{(F,P)} (V_2^{\ast}, t_2^{\ast})$ has distance $1$ and is strongly keen, we prove the following.

\begin{assertion}\label{e1e2-3}
$E_1\cap E_2\ne \emptyset$ for any $E_1\in \mathcal{D}_1$ and $E_2\in \mathcal{D}_2$ with $(E_1,E_2)\ne (D_1,D_2)$.
\end{assertion}

To prove the above assertion, we divide $\mathcal{D}_i$ $(i=1,2)$ into three sets $\mathcal{D}_i^1$, $\mathcal{D}_i^2$, $\mathcal{D}_i^3$, where
\begin{itemize}
\item $\mathcal{D}_i^1$ consists of the single disk $D_i$,
\item $\mathcal{D}_i^2$ consists of disks which are disjoint from $D_i$, and not isotopic to $D_i$,
\item $\mathcal{D}_i^3$ consists of disks which are not isotoped to be disjoint from $D_i$.
\end{itemize}
Suppose $E\in\mathcal{D}_i^2\cup\mathcal{D}_i^3$. 
Since $g=0$, $D_i$ and $E$ are separating in $V_i^{\ast}\setminus t_i^{\ast}$.
It is easy to see that this implies the following.
\begin{claim}\label{claim-di2di3}
{\rm (1)} Every $E\in \mathcal{D}_i^2$ is essential in $(W_i^1\cup_{\overline{h}_i} V_i)\setminus t_i^{\ast}$.

{\rm (2)} For any $E\in \mathcal{D}_i^3$, the closure of every component of $E\setminus D_i$ that is outermost in $E$ is an essential disk in $(W_i^1\cup_{\overline{h}_i} V_i)\setminus t_i^{\ast}$,
provided $|E\cap D_i|$ is minimal.
\end{claim}

\begin{proof}[Proof of Assertion~\ref{e1e2-3}]
Suppose on the contrary that there exist $E_1\in \mathcal{D}_1$ and $E_2\in \mathcal{D}_2$ such that $(E_1,E_2)\ne (D_1,D_2)$ and $E_1\cap E_2= \emptyset$.
We may assume that $E_1\in \mathcal{D}_1^i$ and $E_2\in \mathcal{D}_2^j$ for some $i$ and $j$ such that $i\le j$ and $j\ne 1$, since the remaining cases can be treated similarly.
Assume that $|E_1\cap D_1|$, $|E_2\cap D_2|$, $|E_1\cap D_2|$ and $|E_2\cap D_1|$ are minimal
(note that this configuration is realized by taking a complete hyperbolic structure with finite area on $F\setminus P$, and realizing $\partial D_1$, $\partial D_2$, $\partial E_1$, $\partial E_2$ as geodesics with respect to the hyperbolic metric).

\setcounter{case}{0}

\begin{case}\label{case10-1}
$E_1\in \mathcal{D}_1^1$, that is, $E_1=D_1$.
\end{case}

\begin{subcase}
$E_2\in \mathcal{D}_2^2$.
\end{subcase}

%After boundary compressions toward $\partial_- W_2$ if necessary, we may assume that $E_2':=E_2\cap V_2$ consists of a single disk and $E_2\cap W_2^1$ is a vertical annulus in $W_2^1(\cong \partial_- W_2\times I)$.
%Then we have $h_2(\partial E_2')\in h_2(\mathcal{D}^0(V_2\setminus t_2))$ and $h_2(\partial E_2')(=\Phi_2(\partial E_2))\cap \Phi_2(\alpha_1)=\emptyset$.
%These imply that 
By Claim~\ref{claim-di2di3} (1) and Proposition~\ref{prop-b>1} (1), we have
$$
d_{\partial_- W_2\setminus s_2}(\Phi_2(\alpha_0),h_2(\mathcal{D}^0(V_2\setminus t_2))) \le 1,
$$
a contradiction to the inequality (\ref{eqn-n1-6}).

\begin{subcase}
$E_2\in \mathcal{D}_2^3$.
\end{subcase}

Let $\Delta$ be the closure of a component of $E_2\setminus D_2$ that is outermost in $E_2$.
By Claim~\ref{claim-di2di3} (2), $\Delta$ is an essential disk in $(W_2^1\cup_{\overline{h}_2}V_2)\setminus t_2^{\ast}$.
%By the minimality of $|E_1\cap D_1|+|E_2\cap D_2|$, we see that each outermost component of $E_2\setminus D_2$ is contained in $(W_2^1\cup_{\overline{h}_2}V_2)\setminus t_2^{\ast}$. %, since $E_2\cap \alpha_2=\emptyset$.
%Also, $\Delta$ is essential in $(W_2^1\cup_{\overline{h}_2}V_2)\setminus t_2^{\ast}$.
We have $\Delta\cap \alpha_0=\emptyset$ since $E_2\cap \alpha_0(=E_2\cap \partial D_1)=\emptyset$.
Then the union of $\Delta$ and one of the components of $D_2\setminus \Delta$ is a disk which belongs to $\mathcal{D}_2^2$ and is disjoint from $D_1$.
This cannot occur as we have seen in the previous case.

\vspace{3mm}
In the rest of this section, let $A\,(\subset F\setminus P)$ be the punctured annulus bounded by $\alpha_0\cup \alpha_1$.

\begin{case}\label{case10-2}
$E_1\in \mathcal{D}_1^2$.
\end{case}

We may assume that $E_1\cap \alpha_1\ne \emptyset$, since otherwise, we can apply the arguments in Case 1.1 to derive a contradiction.

\begin{subcase2}
$E_2\in \mathcal{D}_2^2$.
\end{subcase2}

We may assume that $E_2\cap \alpha_0\ne\emptyset$ by the arguments in Case 1.1.
Since $\partial E_i\cap \alpha_{i-1}=\emptyset$ and $\partial E_i\cap \alpha_{2-i}\ne \emptyset$ $(i=1,2)$, we see that each component of $\partial E_i\cap A$ is an arc whose endpoints are contained in $\alpha_{2-i}$ $(i=1,2)$ (see Figure~\ref{fig-n1-g0-1}).
\begin{figure}[tb]
 \begin{center}
 \includegraphics[width=50mm]{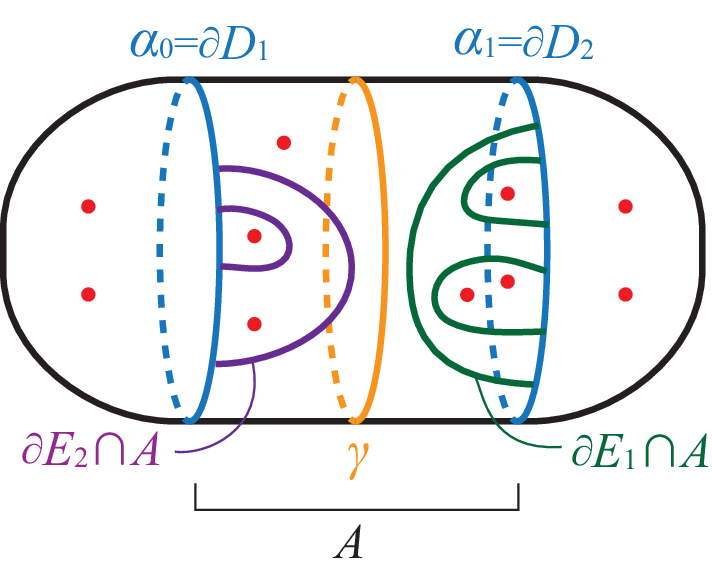}
 \end{center}
 \caption{$\partial E_1\cap A$, $\partial E_2\cap A$ and $\gamma$.}
\label{fig-n1-g0-1}
\end{figure}
Hence, there is a simple closed cure $\gamma$ in $A$ such that $\gamma\cap (\partial E_1\cup \partial E_2)=\emptyset$ and that $\gamma$ separates $\alpha_0$ and $\alpha_1$.
We note that $A$ contains $(2b-4)$ punctures.
Since $b\ge 4$, either of the two components of $A\setminus \gamma$ contains at least two punctures.
Since the arguments are symmetric, we may assume without loss of generality that the component of $A\setminus \gamma$ adjacent to $\alpha_0$ contains at least two punctures.
Note that this fact implies $\Phi_1(\gamma)\ne\emptyset$.
%
%After boundary compressions toward $\partial_- W_1$ if necessary, we may assume that $E_1':=E_1\cap V_1$ consists of a single disk and $E_1\cap W_1^1$ is a vertical annulus in $W_1^1(\cong \partial_- W_1\times I)$.
%Then we have $h_1(\partial E_1')\in h_1(\mathcal{D}^0(V_1\setminus t_1))$ and $(h_1(\partial E_1')(=\Phi_1(\partial E_1))\cup \Phi_1(\alpha_2))\cap \Phi_1(\gamma)=\emptyset$.
%These imply that 
Then, by Proposition~\ref{prop-b>1} (1) with regarding $\alpha=\gamma$ and $\beta=\partial E_1$, we have $d_{\partial_- W_1\setminus s_1}(\Phi_1(\gamma),h_1(\mathcal{D}^0(V_1\setminus t_1)))\le 1$.
Hence,
\begin{eqnarray*}
\begin{array}{rcl}
d_{\partial_- W_1\setminus s_1}(\Phi_1(\alpha_1),h_1(\mathcal{D}^0(V_1\setminus t_1)))&\le&d_{\partial_- W_1\setminus s_1}(\Phi_1(\alpha_1),\Phi_1(\gamma))\\&&+ d_{\partial_- W_1\setminus s_1}(\Phi_1(\gamma),h_1(\mathcal{D}^0(V_1\setminus t_1)))\\
&\le &1+1=2,
\end{array}
\end{eqnarray*}
a contradiction to the inequality (\ref{eqn-n1-5}).

\begin{subcase2}
$E_2\in \mathcal{D}_2^3$.
\end{subcase2}

Let $\Delta\,(\subset (W_2^1\cup_{\overline{h}_2}V_2)\setminus t_2^{\ast})$ be the closure of a component of $E_2\setminus D_2$ that is outermost in $E_2$.
Then we claim that $\Delta\cap \alpha_0\ne\emptyset$.
In fact, if $\Delta\cap \alpha_0=\emptyset$, then by Claim~\ref{claim-di2di3} (2) and Proposition~\ref{prop-b>1-1}, we have
$$
d_{\partial_- W_2\setminus s_2}(\Phi_2(\alpha_0),h_2(\mathcal{D}^0(V_2\setminus t_2)))\le 1,
$$
contradicting the inequality (\ref{eqn-n1-6}).

Since $\Delta$ is outermost in $E_2\setminus D_2$ and $\Delta\cap \alpha_0\ne\emptyset$, we see that $\Delta\cap A$ contains exactly two arcs joining $\alpha_0$ and $\alpha_1$, and other components of $\Delta\cap A$ are disjoint from $\alpha_1$.
This shows that there are exactly two components of $A\setminus \Delta$ that are adjacent to $\alpha_1$.
Let $G^1$ and $G^2$ be the closures of the components (see Figure~\ref{fig-n1-g0-2}).
\begin{figure}[tb]
 \begin{center}
 \includegraphics[width=50mm]{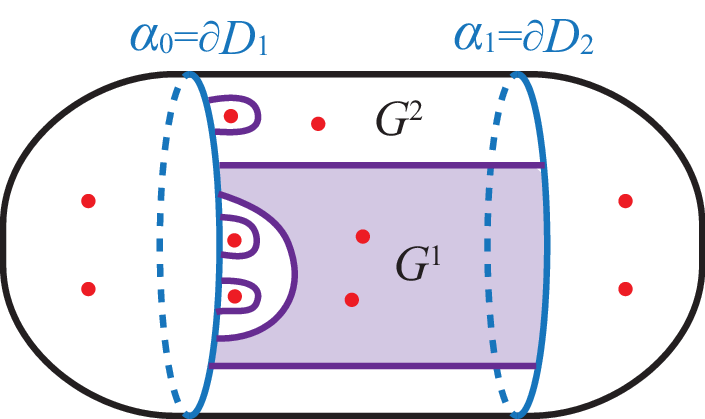}
 \end{center}
 \caption{$\partial E_2\cap A$ and $G^1$, $G^2$.}
\label{fig-n1-g0-2}
\end{figure}

\begin{claim}\label{claim-puncture}
$G^i$ contains at most one puncture $(i=1,2)$.
\end{claim}

\begin{proof}
Suppose on the contrary that $G^1$ or $G^2$, say $G^1$, contains more than one punctures.
Let $\gamma$ be a simple closed curve in $G^1$ that bounds a disk containing the punctures (see Figure~\ref{fig-n1-g0-3}).
\begin{figure}[tb]
 \begin{center}
 \includegraphics[width=50mm]{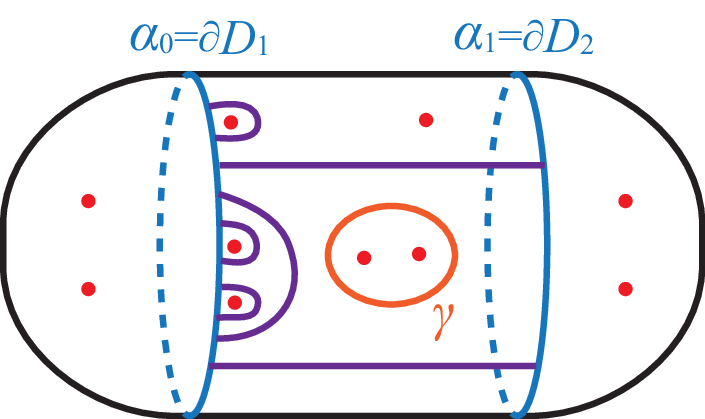}
 \end{center}
 \caption{$\gamma$.}
\label{fig-n1-g0-3}
\end{figure}
Note that this implies $\Phi_2(\gamma)\ne\emptyset$.
By Proposition~\ref{prop-b>1-1} (1), we have 
\begin{eqnarray*}
\begin{array}{rcl}
d_{\partial_- W_2\setminus s_2}(\Phi_2(\alpha_0),h_2(\mathcal{D}^0(V_2\setminus t_2)))&\le&d_{\partial_- W_2\setminus s_2}(\Phi_2(\alpha_0),\Phi_2(\gamma))\\&&+ d_{\partial_- W_2\setminus s_2}(\Phi_2(\gamma),h_2(\mathcal{D}^0(V_2\setminus t_2)))\\
&\le &1+1=2.
\end{array}
\end{eqnarray*}
a contradiction to the inequality (\ref{eqn-n1-6}).
\end{proof}

Recall that $E_1\cap \alpha_0=\emptyset$ and $E_1\cap \alpha_1\ne\emptyset$.
Hence, each component of $E_1\cap A$ is an arc whose endpoints are contained in $\alpha_1$.
This and Claim~\ref{claim-puncture} together with the minimality of $|E_1\cap D_2|$ show that each component of $E_1\cap A$ together with a subarc of $\alpha_1$ bounds a once-punctured disk in $G^i$ $(i=1,2)$.
Hence, $E_1\cap A$ has at most two parallel classes in the punctured annulus $A$ (see Figure~\ref{fig-n1-g0-4}).
\begin{figure}[tb]
 \begin{center}
 \includegraphics[width=50mm]{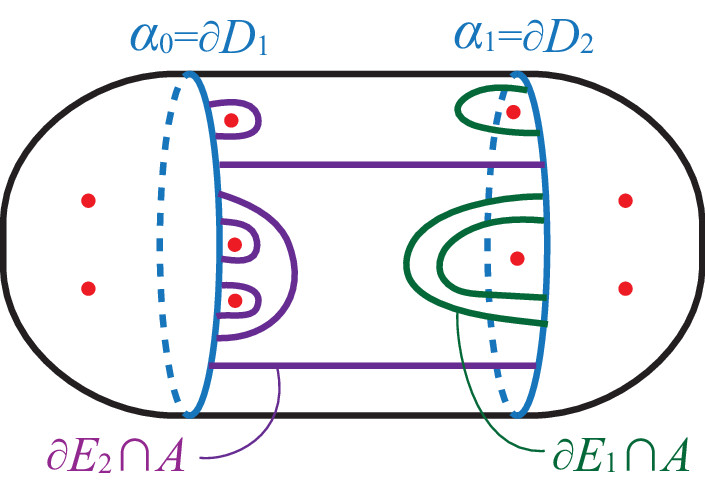}
 \end{center}
 \caption{$\partial E_1\cap A$.}
\label{fig-n1-g0-4}
\end{figure}
Since $A$ contains $(2b-4)$ punctures and $2b-4\ge 4$ by the assumption, there is a simple closed curve $\gamma$ in $A$ such that $\gamma\cap \partial E_1=\emptyset$ and that $\gamma$ bounds a twice-punctured disk in $A$, and hence $\Phi_1(\gamma)\ne\emptyset$ (see Figure~\ref{fig-n1-g0-5}).
\begin{figure}[tb]
 \begin{center}
 \includegraphics[width=50mm]{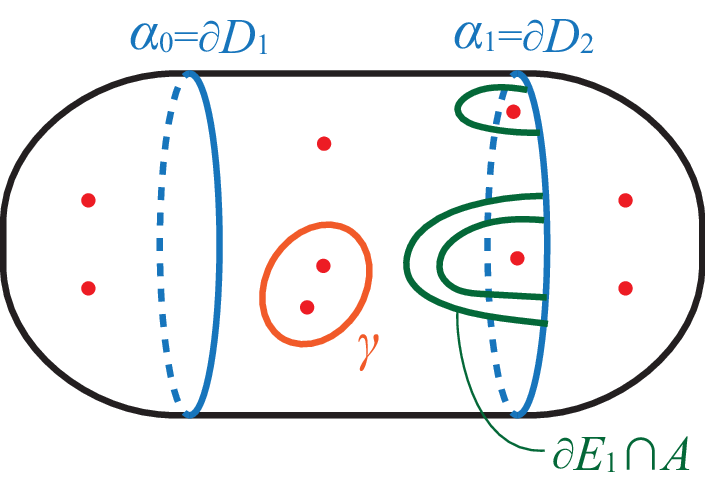}
 \end{center}
 \caption{$\gamma$.}
\label{fig-n1-g0-5}
\end{figure}
Then, by Proposition~\ref{prop-b>1} (1) with regarding $\alpha=\gamma$ and $\beta=\partial E_1$, we have $d_{\partial_- W_1\setminus s_1}(\Phi_1(\gamma),h_1(\mathcal{D}^0(V_1\setminus t_1)))\le 1$.
Hence,
\begin{eqnarray*}
\begin{array}{rcl}
d_{\partial_- W_1\setminus s_1}(\Phi_1(\alpha_1),h_1(\mathcal{D}^0(V_1\setminus t_1)))&\le&d_{\partial_- W_1\setminus s_1}(\Phi_1(\alpha_1),\Phi_1(\gamma))\\&&+ d_{\partial_- W_1\setminus s_1}(\Phi_1(\gamma),h_1(\mathcal{D}^0(V_1\setminus t_1)))\\
&\le &1+1=2.
\end{array}
\end{eqnarray*}
a contradiction to the inequality (\ref{eqn-n1-5}).

\begin{case}\label{case10-3}
$E_1\in \mathcal{D}_1^3$.
\end{case}

In this case, $E_2\in \mathcal{D}_2^3$.
Let $\Delta_i\,(\subset (W_i^1\cup_{\overline{h}_i}V_i)\setminus t_i^{\ast})$ be the closure of a component of $E_i\setminus D_i$ that is outermost in $E_i$ ($i=1,2$).
Then we may assume that $\Delta_i\cap \alpha_{2-i}\ne\emptyset$ for each $i=1,2$, since otherwise, by Proposition~\ref{prop-b>1-1}, we have
$$
d_{\partial_- W_j\setminus s_j}(\Phi_i(\alpha_{2-j}),h_j(\mathcal{D}^0(V_j\setminus t_j)))\le 1
$$
for $j=1$ or $2$, contradicting the inequality (\ref{eqn-n1-5}) or (\ref{eqn-n1-6}).

Since $\Delta_i$ is outermost in $E_i\setminus D_i$ and $\Delta_i\cap \alpha_{2-i}\ne\emptyset$, we see that $\Delta_i\cap A$ contains exactly two arcs $\psi_i^1$, $\psi_i^2$ joining $\alpha_0$ and $\alpha_1$, and the other components are disjoint from $\alpha_{i-1}$.
This shows that there are exactly two components of $A\setminus \Delta_i$ that are adjacent to $\alpha_{i-1}$ (see Figure~\ref{fig-n1-g0-6}).
\begin{figure}[tb]
 \begin{center}
 \includegraphics[width=110mm]{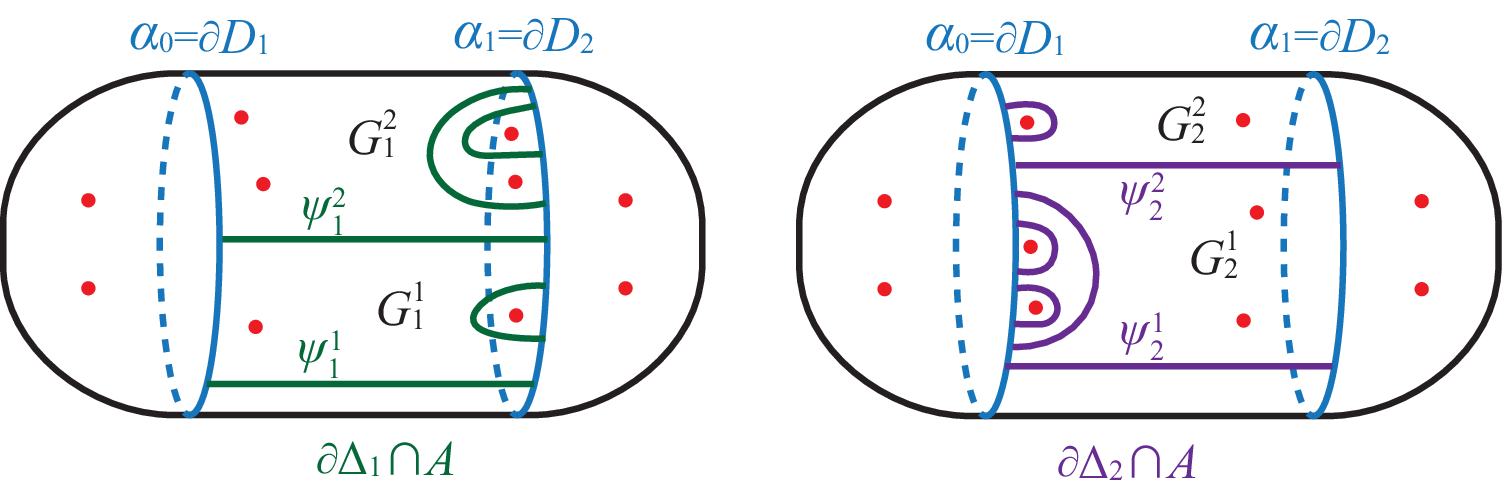}
 \end{center}
 \caption{$\psi_i^1$, $\psi_i^2$, $G_i^1$ and $G_i^2$.}
\label{fig-n1-g0-6}
\end{figure}
Let $G_i^1$ and $G_i^2$ be the closures of the components.
Then the next claim can be proved by using the same arguments in the proof of Claim~\ref{claim-puncture}

\begin{claim}\label{claim-puncture2}
$G_i^j$ contains at most one puncture $(i,j\in\{1,2\})$.
\end{claim}

Recall that $A$ contains $2b-4\,(\ge 4)$ punctures.
Since $(\Delta_1\cap A)\cap (\Delta_2\cap A)\subset \partial E_1\cap \partial E_2=\emptyset$, this fact together with Claim~\ref{claim-puncture2} implies that each $G_i^j$ contains exactly one puncture (see Figure~\ref{fig-n1-g0-7}), and $b$ must be $4$.
%Note that the endpoints of each components of $(\partial \Delta_1\cap A)\setminus (\psi_1^1\cup \psi_1^2)$ is contained in $\alpha_1$, and this fact together with Claim~\ref{claim-puncture2} shows that the component together with a subarc of $\alpha_1$ bounds a once-punctured disk.
%Since $(\Delta_1\cap A)\cap (\Delta_2\cap A)=\emptyset$, this fact together with Claim~\ref{claim-puncture2} shows that the set of arcs in $(\partial \Delta_1\cap A)\setminus (\psi_1^1\cup \psi_1^2)$ consists of at most two parallel classes in the punctured annulus $A$ (see Figure~\ref{fig-n1-g0-7}).
%
\begin{figure}[tb]
 \begin{center}
 \includegraphics[width=110mm]{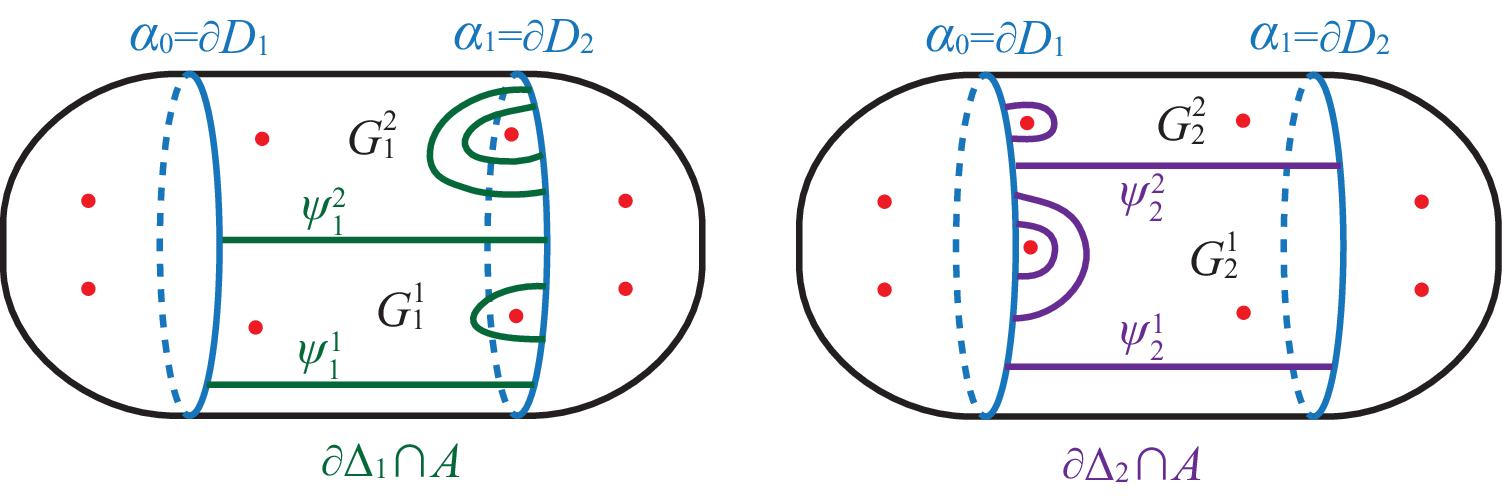}
 \end{center}
 \caption{$\partial \Delta_1\cap A$ and $\partial \Delta_2\cap A$.}
\label{fig-n1-g0-7}
\end{figure}
%
%Since $A$ contains $2b-4(\ge 4)$ punctures, the above facts imply that $G_1^1\cup G_1^2$ contains at least 2 punctures.
%This together with Claim~\ref{claim-puncture2} implies that each of $G_1^1$ and $G_1^2$ contains exactly one punctures, and hence we also have $b=4$.
%Note also that the set of arcs in $(\partial \Delta_2\cap A)\setminus (\psi_2^1\cup \psi_2^2)$ consists of at most two ambient isotopy classes in $A$.
%
Then there exists a simple closed curve $\gamma$ in $A$ (and hence in $F_1$) that bounds a twice-punctured disk, say $D_{\gamma}$, in $A(\subset F_1)$, that intersects $\partial \Delta_1$ twice, and is disjoint from $\alpha_1$ (see Figure~\ref{fig-n1-g0-8}).
\begin{figure}[tb]
 \begin{center}
 \includegraphics[width=110mm]{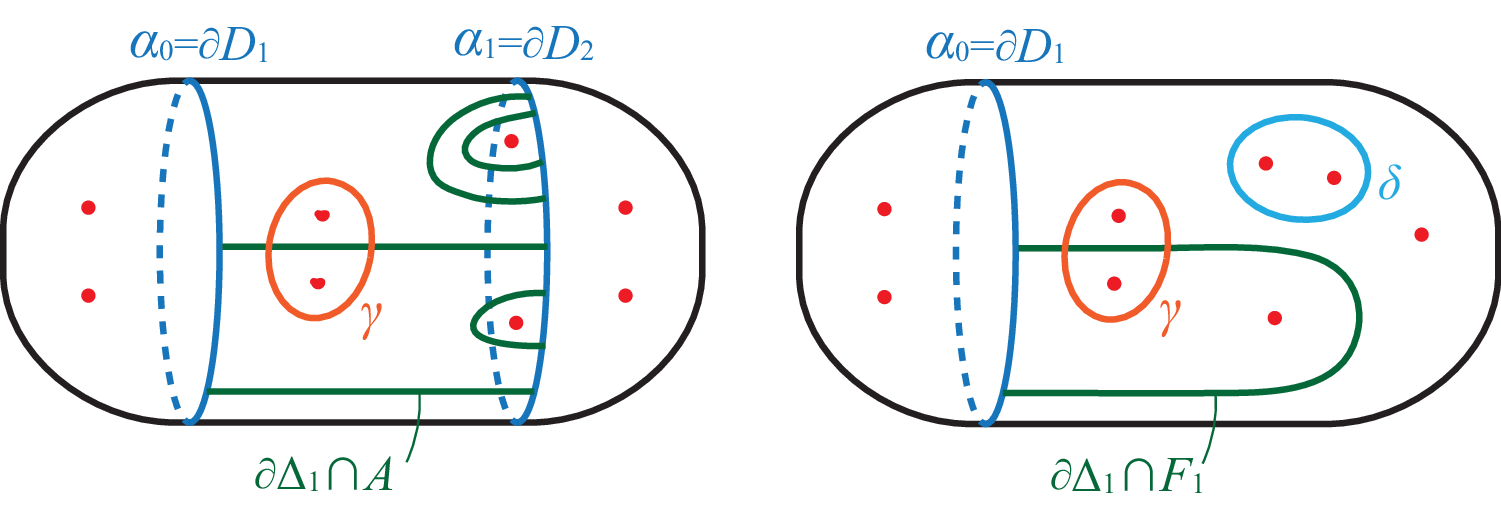}
 \end{center}
 \caption{$\gamma$ and $\delta$.}
\label{fig-n1-g0-8}
\end{figure}
Note that $F_1$ contains 6 punctures, and hence $F_1\setminus D_{\gamma}$ contains 4 punctures.
Since $F_1\setminus (\Delta_1\cup D_{\gamma})$ consists of two components, either of the components must contain at least 2 punctures.
Then there exists a simple closed curve $\delta$ that bounds a twice-punctured disk in (the interior of) the component.
Note that $\alpha_1\cap \gamma=\emptyset$, $\gamma\cap \delta=\emptyset$, $\delta\cap\Delta_1=\emptyset$, and $\Phi_1(\alpha_1)\ne \emptyset$, $\Phi_1(\gamma)\ne \emptyset$, $\Phi_1(\delta)\ne \emptyset$.
Further, by Proposition~\ref{prop-b>1-1} (1) with regarding $\alpha=\delta$ and $\Delta=\Delta_1$, we have $d_{\partial_- W_1\setminus s_1}(\Phi_1(\delta),h_1(\mathcal{D}^0(V_1\setminus t_1)))\le 1$.
Hence,
\begin{eqnarray*}
\begin{array}{rcl}
d_{\partial_- W_1\setminus s_1}(\Phi_1(\alpha_1),h_1(\mathcal{D}^0(V_1\setminus t_1)))&\le&d_{\partial_- W_1\setminus s_1}(\Phi_1(\alpha_1),\Phi_1(\gamma))\\
&&+d_{\partial_- W_1\setminus s_1}(\Phi_1(\gamma),\Phi_1(\delta))\\
&&+ d_{\partial_- W_1\setminus s_1}(\Phi_1(\delta),h_1(\mathcal{D}^0(V_1\setminus t_1)))\\
&\le &1+1+1=3,
\end{array}
\end{eqnarray*}
a contradiction to the inequality (\ref{eqn-n1-5}).

This completes the proof of Assertion~\ref{e1e2-3}.
\end{proof}

\part{Proof of Theorems \ref{thm-3} and \ref{thm-2}}
%\appendix

%\def\thesection{A}

\section{Strongly keen $(0,2)$-splittings}\label{app-2}

Let $F$ be a 2-sphere and $P$ the union of 4 points in $F$.
Note that each essential simple closed curve in $F\setminus P$ separates $F\setminus P$ into two twice-punctured disks.
Recall from Subsection~\ref{sec-cc} that in the curve complex $\mathcal{C}(F\setminus P)$, two vertices $\alpha$ and $\beta$ are joined by a 1-simplex if and only if $\alpha$ and $\beta$ intersect in two points.

We show that, for any positive integer $n$, there exist strongly keen $(0,2)$-splittings of links with distance $n$.
This follows from well-known facts on the structure of the Farey graph and a result in \cite{BKKS} on geodesics in the Farey graph.
The facts that are needed in the proof of Theorems~\ref{thm-1} and \ref{thm-3} are summarized in Appendix~\ref{app-b}.

We call a pair $(B^3,t)$ of the 3-ball $B^3$ and the union of arcs $t$ properly embedded in $B^3$ a {\it tangle}.
A tangle $(B^3,t)$ is said to be {\it trivial} if $t$ is parallel to $\partial B^3$.
A {\it rational tangle} is a trivial tangle with two arcs, where its boundary fixed. 
A well-known fact is that rational tangles correspond to extended rational numbers, called the {\it slopes} of the tangles.
For example, the rational tangle of slope $\frac{p}{q}$ can be illustrated as in Figure~\ref{fig-rtangle}, where $\frac{p}{q}$ admits a continued fraction expansion $[a_1, a_2, \dots, a_n]$ (see Appendix~\ref{app-b}).
\begin{figure}[tb]
 \begin{center}
 \includegraphics[width=100mm]{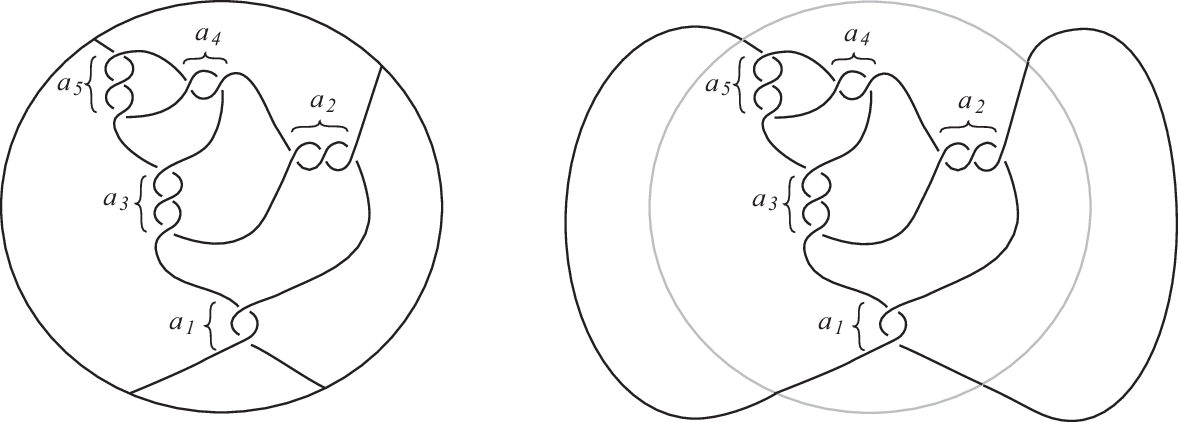}
 \end{center}
 \caption{Rational tangle of slope $\frac{79}{182}=[2,3,3,2,3]$ and $2$-bridge knot $S(182, 79)$.}
\label{fig-rtangle}
\end{figure}
In the figure, the numbers $a_i$ denote the numbers of right-hand half twists.
Note that $(0,2)$-splitting is a decomposition of a pair of the 3-sphere and a link in the 3-sphere into two rational tangles. 
Here, we may suppose that the slope of one of the rational tangles is $\frac{1}{0}$ and the slope of the other is $\frac{p}{q}$.
Conversely, for any extended rational number $\frac{p}{q}\in \mathbb{Q}\cup\{\frac{1}{0}\}$, there exists a $(0,2)$-splitting of a link corresponding to the decomposition into rational tangles with slopes $\frac{1}{0}$ and $\frac{p}{q}$.
Such a link is denoted by $S(q,p)$ and called a {\it $2$-bridge link} when $q=0$ or $q\ge 2$ (and $p$, $q$ are coprime).

\begin{proof}[Proof of Theorem \ref{thm-3}]
When $n\ge 2$, take $\frac{p}{q}$ with a continued fraction expansion $[a_1, a_2, \dots, a_{n-1}]$, where $a_i\ge 3$ for every $i\in\{1,2,\dots,n-1\}$.
Then, by Theorem~\ref{thm-b} in Appendix~\ref{app-b}, the spine of the ladder $\mathcal{L}\left(\frac{1}{0},\frac{p}{q}\right)$ is the unique geodesic connecting $\frac{1}{0}$ and $\frac{p}{q}$ in $\mathcal{F}$ of length $n$.
(For the definition of ladder, see Appendix~\ref{app-b}.)
Let $(V_1,t_1)\cup_{(F,P)}(V_2,t_2)$ be the $(0,2)$-splitting of the link $S(q,p)$ corresponding to the decomposition into rational tangles with slopes $\frac{1}{0}$ and $\frac{p}{q}$.
Then the distance of $(V_1,t_1)\cup_{(F,P)}(V_2,t_2)$ is the length of the simplicial geodesic in $\mathcal{F}$ connecting $\frac{1}{0}$ and $\frac{p}{q}$.
The above facts imply that $(V_1,t_1)\cup_{(F,P)}(V_2,t_2)$ is a strongly keen $(0,2)$-splitting with distance $n$.

In case when $n=1$, it can be easily seen that the $(0,2)$-splitting of the unknot has distance $n(=1)$ and is strongly keen. 
\end{proof}

\section{Bridge splittings of distance $0$}\label{app-1}

In this section, we give a characterization of the bridge splittings of distance 0. 
Let $L$ be a link in a 3-manifold $M$, and let $E(L):={\rm cl}(M\setminus N(L))$.
Let $(V_1, t_1) \cup_{(F, P)} (V_2, t_2)$ 
be a  $(g, b)$-splitting of $(M, L)$(, where $b \geq 1$). 
We say that 
$(V_1, t_1) \cup_{(F, P)} (V_2, t_2)$ 
is {\it stabilized} 
if there is a pair $(D_1, D_2)$ of 
essential disks $D_1\subset V_1 \setminus t_1$ and 
$D_2\subset V_2 \setminus t_2$ such that 
$\partial D_1$ and $\partial D_2$ intersect 
transversely in one point (cf. \cite[p.303]{HS}).

\begin{theorem}\label{thm-n0}
Let $(V_1, t_1) \cup_{(F, P)} (V_2, t_2)$ be as above. 
Then the distance of the bridge splitting $(V_1, t_1) \cup_{(F, P)} (V_2, t_2)$ is $0$ 
if and only if 
either one of the following holds. 
\begin{enumerate}
\item[{\rm (1)}] $E(L)$ is reducible.
\item[{\rm (2)}] $(V_1, t_1) \cup_{(F, P)} (V_2, t_2)$ is stabilized. 
\end{enumerate}
\end{theorem}

\begin{proof}[Proof of \lq\lq only if\rq\rq\ part of Theorem~\ref{thm-n0}] 
Suppose that 
the distance of the bridge splitting $(V_1, t_1) \cup_{(F, P)} (V_2, t_2)$ is $0$. 
Then there is a pair $(D_1, D_2)$ of 
essential disks $D_1\subset V_1 \setminus t_1$ and 
$D_2\subset V_2 \setminus t_2$
such that 
$\partial D_1 = \partial D_2$. 
Let $S = D_1 \cup D_2$. 
Note that $S$ is a 2-sphere in $M$ 
such that 
$S \cap L = \emptyset$. 
Then we have the following cases. 

\medskip\noindent
Case 1. $S$ is non-separating. 

\medskip
In this case,  $E(L)$ is reducible. 

\medskip\noindent
Case 2. $S$ is separating. 

\medskip
In this case, we have the following two subcases. 

\medskip\noindent
Case 2-1. $S$ is essential in $E(L)$. 

\medskip
In this case,  $E(L)$ is reducible. 

\medskip\noindent
Case 2-2. $S$ is inessential in $E(L)$. 

\medskip
In this case, $S$ bounds a 3-ball $B^3$ in $E(L)$. 
Let $\bar{B^3}$ be the 3-manifold obtained from $B^3$ by identifying $D_1$ and $D_2$ by a homeomorphism 
extending the natural identification $\partial D_1 = \partial D_2$. 
It is easy to see that  $\bar{B^3}$ is homeomorphic to the 3-sphere, and 
the image of $(F \cap B^3) \cup (D_1 \cup D_2)$ is a genus-$h$ Heegaard splitting of  $\bar{B^3}$ ($h \ge 1$). 
Then by the uniqueness of Heegaard splittings of the 3-sphere \cite{Wa}, 
we see that the Heegaard splitting is stabilized. 
It is easy to see that this fact implies the bridge splitting 
$(V_1, t_1) \cup_{(F, P)} (V_2, t_2)$ is stabilized. 
(We note that  several authors gave alternative proofs for
the uniqueness of Heegaard splittings of the 3-sphere. 
See, for example, Schleimer's exposition \cite{Sch} and its references.)
\end{proof}

For the proof of \lq\lq if\rq\rq\ part of Theorem~\ref{thm-n0}, we prepare some 
notations. 

Let $\hat{F} = F \cap E(L)$. 
($\hat{F}$ is a genus-$g$ surface with $2b$ boundary components, 
which is properly embedded in $E(L)$.)
Let 
$\hat{V}_i$ be the closure of the component of $E(L)\setminus \hat{F}$ that is contained in $V_i$ ($i=1, 2$). 
Further let $B_i$ be the closure of the union 
of the components of $\partial E(L) \setminus N(\hat{F})$ that are contained in 
$\hat{V}_i$. 
Note that $B_i$ consists of $b$ annuli, and that 
$\hat{V}_i$ can be regarded as a compression body in the terminology of Casson-Gordon \cite{CG} with 
$\partial_-\hat{V}_i = B_i$, 
$\partial_+\hat{V}_i = \hat{F}$. 
This shows that 
$\hat{V}_1 \cup_{\hat{F}} \hat{V}_2$ is 
a Heegaard splitting of the 3-manifold triad $(E(L); B_1, B_2)$. 

We can define the distance of the 
Heegaard splitting 
$\hat{V}_1 \cup_{\hat{F}} \hat{V}_2$, 
%denoted 
%$d(\hat{V}_1 \cup_{\hat{F}} \hat{V}_2)$, 
by tracing the definition of the distance of Heegaard splitting given by Hempel \cite{He}, 
and it is a direct consequence of the definition 
that the distance coincides with the distance of the bridge splitting 
$(V_1, t_1) \cup_{(F, P)} (V_2, t_2)$. 

%$$d(\hat{V}_1 \cup_{\hat{F}} \hat{V}_2)
%= d((V_1, t_1) \cup_{(F, P)} (V_2, t_2))
%(= d(L, F)). $$

\begin{proof}[Proof of \lq\lq if\rq\rq\ part of Theorem~\ref{thm-n0}] 
Suppose that $(V_1, t_1) \cup_{(F, P)} (V_2, t_2)$ is stabilized, 
i.e., 
there is a pair $(D_1, D_2)$ of 
essential disks $D_1\subset V_1 \setminus t_1$ and 
$D_2\subset V_2 \setminus t_2$
such that 
$\partial D_1$ and $\partial D_2$ intersect 
transversely in one point. 
\begin{figure}[tbp]
 \begin{center}
 \includegraphics[width=30mm]{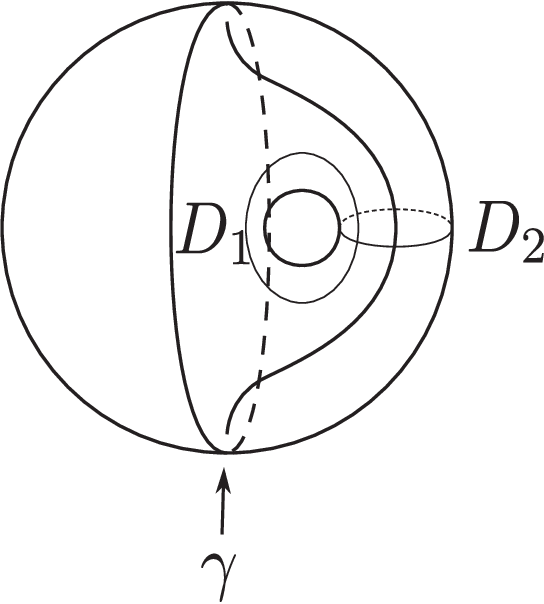}
 \end{center}
 \caption{}
\label{fig-stabilized}
\end{figure}
It is directly observed (see Figure~\ref{fig-stabilized}) that 
$N(D_1 \cup D_2)$ is a 3-ball, and 
$F \cap N(D_1 \cup D_2)$ is a torus with one boundary component denoted by $\gamma$. 
Then it is easy to see that $\partial N(D_1 \cup D_2) \setminus \gamma$ 
consists of two components and that the closure of one of the components is an essential disk in $V_1 \setminus t_1$, and the closure of the other component is an essential disk in $V_2 \setminus t_2$. 
This shows that the distance of $(V_1, t_1) \cup_{(F, P)} (V_2, t_2)$ is $0$. 

Suppose that $E(L)$ is reducible, i.e., there is an essential 2-sphere $\hat{S}$ in $E(L)$. 
Then by \cite[Lemma~1.1]{CG}, we may suppose that $\hat{S} \cap \hat{F}$ 
consists of a single circle. 
Let $\hat{D}_i = \hat{S} \cap \hat{V}_i$.   
Then the pair $( \hat{D}_1 , \hat{D}_2 )$ shows that 
the distance of the Heegaard splitting $\hat{V}_1 \cup \hat{V}_2$ is $0$, 
which implies the distance of $(V_1, t_1) \cup_{(F, P)} (V_2, t_2)$ is $0$. 
\end{proof}

\section{$(0,3)$-splittings with distance $1$}\label{app-3}

\begin{proof}[Proof of Theorem \ref{thm-2}]
We first note that the ambient manifold of the link $L$ is the 3-sphere $S^3$, 
since $L$ admits a (0, 3)-splitting. 

Suppose that $L$ admits a $(0, 3)$-splitting $(B_1^3, t_1) \cup_{(S, P)} (B_2^3, t_2)$ 
with a pair of essential disks $(D_1, D_2)$ in
$B_1^3\setminus  t_1$, $B_2^3\setminus  t_2$ 
respectively such that $D_1 \cap D_2 = \emptyset$, 
and 
$\partial D_1$ and $\partial D_2$ are not isotopic in $S \setminus P$. 
Since $t_i$ ($\subset B_i^3$) consists of three arcs, $D_i$ cuts off a 3-ball with one trivial arc from $B_i^3$.  
This shows that  $\partial D_i$ cuts off a disk with two punctures, 
denoted $D_i^S$, from $S$. 
These show that 
$\partial D_1\cup \partial D_2$ bounds an annulus with two punctures, denoted $A^S$, in $S$. 

\begin{figure}[tbp]
 \begin{center}
 \includegraphics[width=65mm]{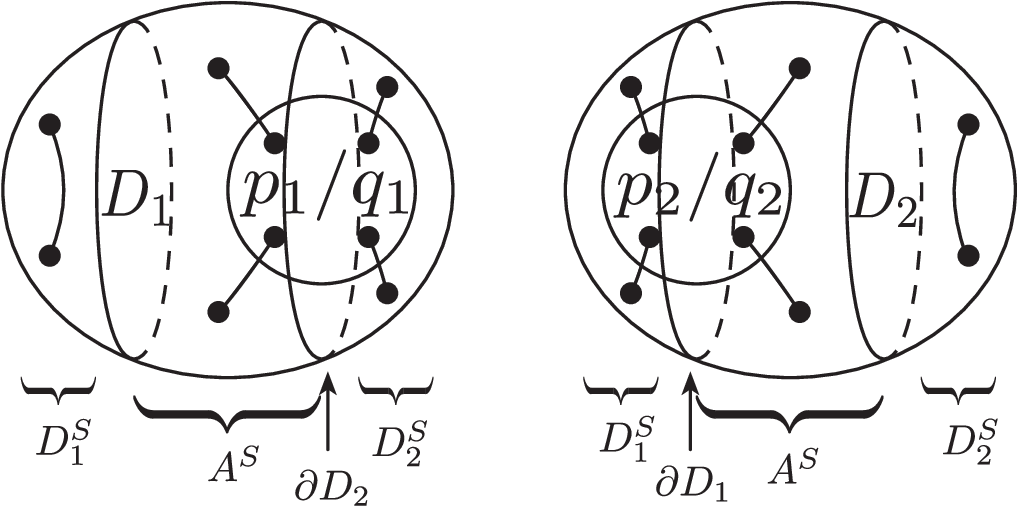}
 \end{center}
 \caption{}
\label{fig-3-bridge_dist1}
\end{figure}

%Let 
%$B_{i, 1}^3$ ($B_{i, 2}^3$ resp.) 
%be the closure of the component of 
%$B_i^3 \setminus D_i$ which contains one (two resp.) component(s) of $t_i$. 
It is directly observed from Figure~\ref{fig-3-bridge_dist1} that the 2-sphere $D_1 \cup A^S \cup D_2$ 
gives the connected sum of two links 
$S(q_1,p_1)$ and $S(q_2,p_2)$
%$L( \beta_1/ \alpha_1)$ and $L( \beta_2/ \alpha_2)$, 
and hence
$L = S(q_1,p_1) \sharp S(q_2,p_2)$. 
(Note that $S(q_j,p_j)$ is possibly a trivial knot ($j=1$ or $2$).)

We note that since $S$ is a genus-$0$ surface, the bridge splitting  
$(B_1^3, t_1) \cup_{(S, P)} (B_2^3, t_2)$ 
is not stabilized. 
Hence by Theorem~\ref{thm-n0} the distance of $(B_1^3,t_1)\cup_{(S,P)} (B_2^3,t_2)$ is $0$ if and only if $E(L)$ is reducible. 
Note that  $E(S(q_i,p_i))$ is reducible 
if and only if 
$S(q_i,p_i)$
is the 2-component trivial link. 
Further it is easy to see that $E(L)$ is irreducible if and only if both 
$E(S(q_1,p_1))$ and 
$E(S(q_2,p_2))$ 
are irreducible. 
These together with Theorem~\ref{thm-n0} show that 
the distance of $(B_1^3,t_1)\cup_{(S,P)} (B_2^3,t_2)$ is $0$  
if and only if 
either 
$S(q_1,p_1)$ or 
$S(q_2,p_2)$ 
is the 2-component trivial link. 
This fact implies: 
the distance of $(B_1^3,t_1)\cup_{(S,P)} (B_2^3,t_2)$ is $1$ if and only if one of the following holds. 

\begin{enumerate}
\item
Both 
$S(q_1,p_1)$ and 
$S(q_2,p_2)$ 
are trivial knots, 
i.e., 
$L$ is a trivial knot. 

\item 
Either one of 
$S(q_1,p_1)$ or 
$S(q_2,p_2)$ 
is a trivial knot, and 
the other is a 2-bridge link which is not a 2-component trivial link. 

\item 
For $i = 1,2$, 
$S(q_i,p_i)$ 
is a 2-bridge link which is not a 2-component trivial link. 
\end{enumerate}

This proves the first half of Theorem~\ref{thm-2}.

To prove the last half of Theorem~\ref{thm-2}, 
we suppose that the distance of the $(0,3)$-splitting $(B_1^3,t_1)\cup_{(S,P)} (B_2^3,t_2)$ is $1$.
Then let $D_1$, $D_2$ and $A_S$ be as above.
Let $\gamma$ be an essential simple closed curve on $A_S$ that separates the $2$ punctures, and let $S_1$ and $S_2$ be the two subdisks of $S$ bounded by $\gamma$ such that $\partial D_i\subset S_i$ ($i=1,2$).
Let $B_1'$ be the closure of the component of $B_1^3\setminus D_1$ containing two components of $t_1$.
Then $(B_1',t_1\cap B_1')$ is a rational tangle (which corresponds to the component of $B_1^3\setminus D_1$ containing the rational tangle with slope $p_1/q_1$ in Figure~\ref{fig-3-bridge_dist1}).
Moreover, one of the two disks of $\partial B_1'\setminus \gamma$ contains $D_1$ and one of the four points $t_1\cap \partial B_1'$, and the other disk contains the rest three points (see Figure~\ref{fig-e1}).
Hence, there exists an essential disk $E_1$ in $B_1'\setminus t_1$ such that $\partial E_1$ is contained in the interior of $S_2$.
Note that $E_1$ is also an essential disk in $B_1^3\setminus t_1$ and that $E_1$ is not isotopic to $D_1$.
Similarly, there exists an essential disk $E_2$ in $B_2^3\setminus t_2$ such that $\partial E_2$ is contained in the interior of $S_1$ and is not isotopic to $D_2$.
These imply that $E_1$ and $E_2$ are another pair of disks realizing the distance $1$.
Therefore, $(B_1^3,t_1)\cup_{(S,P)}(B_2^3,t_2)$ is not keen.
\end{proof}
\begin{figure}[tb]
 \begin{center}
 \includegraphics[width=35mm]{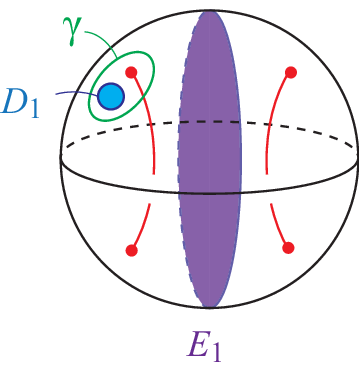}
 \end{center}
 \caption{$(B_1', t_1\cap B_1')$ and $E_1$.}
\label{fig-e1}
\end{figure}

\begin{figure}[tbp]
 \begin{center}
 \includegraphics[width=65mm]{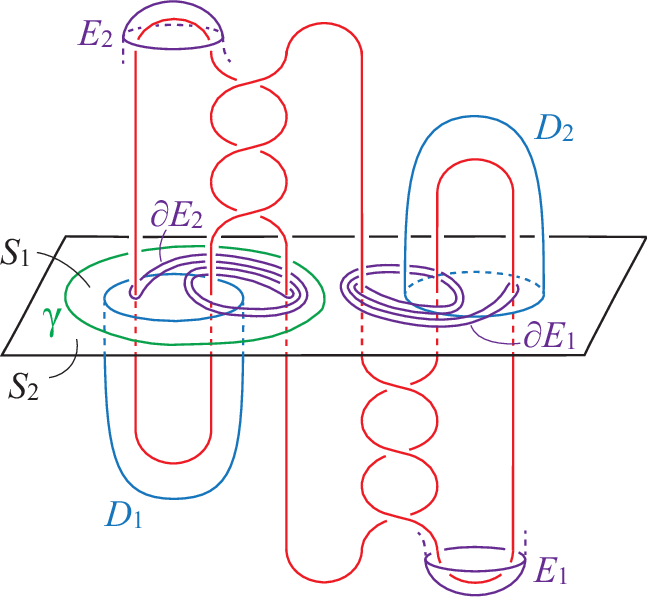}
 \end{center}
 \caption{$(0,3)$-splitting of with distance $1$.}
\label{fig-3-bridge}
\end{figure}
Figure~\ref{fig-3-bridge} shows an example of a $(0,3)$-splitting of a link with distance $1$.

%\section*{Acknowledgement}
%We would like to thank Dr Jesse Johnson for many helpful discussions, particularly for teaching us an idea of constructing a unique geodesic path in the curve complex. 
%We would also like to thank the referee for careful reading of the paper and helpful suggestions.

%%%%%%%%%%%%%%%%%%%%%%%%%%%%%%%%%
% References
%%%%%%%%%%%%%%%%%%%%%%%%%%%%%%%%%

\appendix

\part{Appendix}

\section{Image of disk complex}\label{appendix-disk-complex}

In this section, we prove the following proposition, which is used in the proof of Claim~\ref{claim-d-cplx}.

\begin{proposition}\label{prop-image-d}
Let $V$ be a genus-$g$ handlebody with $g\ge 2$ and $t$ be the union of $b$ arcs properly embedded in $V$ which is parallel to $\partial V$.
Let $F:=\partial V\setminus t$.
Let $l\,(\subset F)$ be a simple closed curve which is non-separating in $F$, and let $X$ be the subsurface ${\rm cl}(F\setminus N_{F}(l))$ of $F$.
Suppose that $l$ intersects every element of $\mathcal{D}^0(V\setminus t)$.
Then either one of the following holds.
\begin{itemize}
\item[{\rm (1)}] $(V,l)$ is homeomorphic to the twisted $I$-bundle $\Omega\tilde{\times} I$ over a non-orientable surface $\Omega$, where each component of $t$ is an $I$-fiber, and $l$ is the core curve of the annulus $\partial\Omega\tilde{\times}I$.
\item[{\rm (2)}] ${\rm diam}_{X}(\pi_X(\mathcal{D}^0(V\setminus t)))\le 12$.
\end{itemize}
\end{proposition}

The above proposition can be proved by arguments in \cite{Li}.
We give an outline of the proof in the remainder of this section.

Let $N:=N_{F}(l)$ (hence, $X={\rm cl}(F\setminus N)$).
Let $D$ be an essential disk in $F$.
We may view $D$ as a $2n$-gon with its vertices being the points in $\partial D\cap \partial X\,(=\partial D\cap \partial N)$.
We call each component of $\partial D\cap X$ an {\it $\alpha$-edge} of $\partial D$, and each component of $\partial D\cap N$ an {\it $\beta$-edge} of $\partial D$.
We say that an arc $\gamma$ properly embedded in $D$ is {\it edge-parallel} if the both endpoints of $\gamma$ lie in the same edge of $\partial D$.

For any pair of essential disks, say $D_1$ and $D_2$, in $V\setminus t$ which appear in the rest of this proof, we may assume that $|D_1\cap D_2|$, $|\partial D_1\cap \partial X|$ and $|\partial D_2\cap \partial X|$ are minimal in their isotopy classes at the same time (see \cite[Lemma 3.1 and Remark 3.2]{Li}).
This implies that no component of $D_1\cap D_2$ is a simple closed curve.

Let $M$ be the minimal value of $|\partial D\cap \partial X|$ among all the essential disk $D$ in $V\setminus t$.
Let $D$ be an essential disk with $|\partial D\cap \partial X|=M$.
Note that $M$ is an even number.
Let $\pi_{AC}: \mathcal{C}^0(F)\rightarrow \mathcal{P}(\mathcal{AC}^0(X))$ be the map introduced in Subsection~\ref{sec-subsurface}.

\begin{assertion}\label{ass-appendix}
One of the following holds.
\begin{itemize}
\item[{\rm (i)}] $\pi_{AC}(\mathcal{D}^0(V\setminus t))$ lies in a ball of radius 3 in $\mathcal{AC}^0(X)$.
\item[{\rm (ii)}] $M=2$ or $4$.
\end{itemize}
\end{assertion}

\begin{proof}
This can be proved by arguments in the proof of \cite[Lemma 3.4]{Li}.
We give only an outline here.
We note that if ${\rm diam}_{\mathcal{AC}(X)}(\pi_{AC}(\partial E),\pi_{AC}(\partial D))\le 3$ for any essential disk $E$ in $V\setminus t$, then $\pi_{AC}(\mathcal{D}^0(V\setminus t))$ lies in a ball of radius 3 centered at an element of $\pi_{AC}(\partial D)$, which gives the conclusion (i) of Assertion.
Hence, in the rest of the proof, we suppose that there is an essential disk $E$ in $V\setminus t$ such that ${\rm diam}_{\mathcal{AC}(X)}(\pi_{AC}(\partial E),\pi_{AC}(\partial D))> 3$.
%Assume that (i) of Assertion fails.
Then we have
\begin{itemize}
\item[($\ast$)] for any $\alpha$-edge $\alpha_D$ of $\partial D$ and any $\alpha$-edge $\alpha_E$ of $\partial E$, we have $\alpha_D\cap \alpha_E\ne \emptyset$,
\end{itemize}
since if there are mutually disjoint $\alpha$-edges $\alpha_D$ and $\alpha_E$ of $\partial D$ and $\partial E$, respectively, then 
\begin{eqnarray*}
\begin{array}{rcl}
{\rm diam}_{\mathcal{AC}(X)}(\pi_{AC}(\partial D),\pi_{AC}(\partial E))&\le&{\rm diam}_{\mathcal{AC}(X)}(\pi_{AC}(\partial D))+d_{\mathcal{AC}(X)}(\alpha_D,\alpha_E)\\&&+{\rm diam}_{\mathcal{AC}(X)}(\pi_{AC}(\partial E))\\
&\le &1+1+1=3,
\end{array}
\end{eqnarray*}
a contradiction.
Let $\Delta$ be the closure of a component of $E\setminus D$ that is outermost in $E$.
By the minimality of $|\partial D\cap \partial E|$, we can show that the outermost arc $\delta$ adjacent to $\Delta$ is not edge-parallel in $E$ (see \cite[Lemma 3.3]{Li}).
Hence, by ($\ast$), $\Delta$ is either a triangle or quadrilateral.
Let $D_1$ and $D_2$ be the closures of the components of $D\setminus \delta$.
Then by applying the arguments for Cases (i) and (ii) in the proof of \cite[Lemma 3.4]{Li} we can show that $|\partial (D_i\cup \Delta)\cap \partial X|<|\partial D\cap \partial X|$ for each $i=1,2$.
On the other hand, at least one of the disks $D_1\cup \Delta$ and $D_2\cup\Delta$ is essential in $V\setminus t$. 
These contradict the minimality of $|\partial D\cap \partial X|$.
Thus, we have $M\le 4$, that is, $M=2$ or $4$.
\end{proof}

If the conclusion (i) of Assertion~\ref{ass-appendix} holds, then we have ${\rm diam}_{X}(\pi_X(\mathcal{D}^0(V\setminus t)))\le 12$, by \cite[Lemma 2.2]{MM2}.
Thus, in the remainder of the proof, we assume that $M=2$ or $4$.

\setcounter{case}{0}

\begin{case}\label{case-a-1}
M=2, that is, $D$ is a bigon.
\end{case}

In this case, $D\cap N$ consists of an arc.
Let $G={\rm cl}(\partial N_{V\setminus t}(N\cup D)\setminus F)$.
Note that $G$ is a disk properly embedded in $V\setminus t$ such that $\partial G\subset X$.
Further, $G$ cuts off a solid torus from $V$, and hence $\partial G$ is essential in $X$ (recall that $g\ge 2$).
This contradicts the assumption that $l$ intersects every element of $\mathcal{D}^0(V\setminus t)$.

\begin{case}\label{case-a-2}
M=4, that is, $D$ is quadrilateral.
\end{case}

Note that a quadrilateral in $V\setminus t$ possesses a product structure $I\times I$ with $I\times \partial I$ a pair of essential arcs in $X$ and $\partial I\times I$ a pair of essential arcs in $N$.
Then as explained in \cite[5.2 Case 2]{IJK2} there is a maximal essential $I$-bundle region for $(V,t)$ with respect to $X$ containing the $I$-bundle structure of the quadrilateral $D$.
More precisely, there exists a compact submanifold $J$ of $V$ such that
\begin{itemize}
\item[1.] $J$ is an $I$-bundle over a compact surface with nonempty boundary such that $t\cap J$ is a union of (possibly empty) $I$-fibers,
\item[2.] the vertical boundary $\partial_v J$(: the total space of the $I$-bundle over the boundary of the base space of $J$) has nonempty intersection with $N$, and $\partial_v J\cap N$ is either an annulus or a rectangular disk $I\times I$, where $I\times \partial I\subset \partial N$ and $\partial I\times I$ is a pair of properly embedded essential arcs in $N$,
\item[3.] the horizontal boundary $\partial_h J\,(:={\rm cl}(\partial J\setminus \partial_v J))$ is a subsurface of $X$, and $J\cap X=\partial_h J$,
\item[4.] each component of the frontier of $\partial_h J$ in $X$ is an essential simple closed curve, or an essential arc in $X$, and
\item[5.] If $J'$ is another submanifold of $V$ satisfying the above conditions 1$\sim$4, then $J'$ is ambient isotopic to $J$ by an isotopy preserving $X$.
\end{itemize}
If $J=V$, then we have the conclusion (1) of Proposition~\ref{prop-image-d}.
Suppose that $J\ne V$.
Let $\gamma$ be a component of the frontier of $\partial_h J\cap X$.
Then we have the following.

\begin{assertion}\label{ass-appendix-2}
For any essential disk $E$ in $V\setminus t$, we have ${\rm diam}_{\mathcal{AC}(X)}(\{\gamma\},\pi_{AC}(\partial E))\le 3$.
\end{assertion}

\begin{proof}
Assume on the contrary that there is an essential disk $E$ in $V\setminus t$ such that ${\rm diam}_{\mathcal{AC}(X)}(\{\gamma\},\pi_{AC}(\partial E))> 3$.
Let $D^{\ast}$ be a quadrilateral essential disk in $V\setminus t$ such that $|D^{\ast}\cap E|$ is minimal among all the quadrilateral essential disks.
If there are $\alpha$-edges $\alpha_{D^{\ast}}$ and $\alpha_E$ of $\partial D^{\ast}$ and $\partial E$, respectively, such that $\alpha_{D^{\ast}}\cap \alpha_E=\emptyset$, then 
\begin{eqnarray*}
\begin{array}{rcl}
{\rm diam}_{\mathcal{AC}(X)}(\{\gamma\},\pi_{AC}(\partial E))&\le&{\rm diam}_{\mathcal{AC}(X)}(\{\gamma\},\pi_{AC}(\partial D^{\ast}))+d_{\mathcal{AC}(X)}(\alpha_{D^{\ast}},\alpha_E)\\&&+{\rm diam}_{\mathcal{AC}(X)}(\pi_{AC}(\partial E))\\
&\le &1+1+1=3,
\end{array}
\end{eqnarray*}
a contradiction.
Hence, each $\alpha$-edge of $\partial D^{\ast}$ and each $\alpha$-edge of $\partial E$ intersect.
Thus, by using arguments in the proof of Assertion~\ref{ass-appendix}, we can find an outermost disk $E$ which is either a triangle or a quadrilateral.
By applying cut-and-paste arguments on $D^{\ast}$ with using the outermost disk, we can obtain a new quadrilateral, say $D^{\ast\ast}$, which is essential in $V\setminus t$ such that $|D^{\ast\ast}\cap E|<|D^{\ast}\cap E|$, a contradiction.
\end{proof}

This completes the proof of Proposition~\ref{prop-image-d}.

\section{Geodesics in Farey graph}\label{app-b}

Let $\mathcal{F}$ be the {\it Farey graph}, that is, a simplicial graph where each vertex is an extended rational number denoted by $\frac{p}{q}$, and a pair of vertices is joined by an edge if and only if these two vertices represent $\frac{p}{q}$ and $\frac{r}{s}$ satisfying $|ps-qr|=1$.
It is well-known that the $1$-skeleton of the curve complex of the 4-punctured sphere or a torus with at most one hole is (isomorphic to) the Farey graph by the correspondence sending $l\in\mathcal{C}^0(\ast)$ to the slope of $l(\in \overline{\mathbb{Q}})$.
It is also well-known that the Farey graph $\mathcal{F}$ can be naturally embedded into a compactification of the hyperbolic plane $\overline{\mathbb{H}}=\mathbb{H}^2\cup\partial \mathbb{H}^2$, where the vertices of $\mathcal{F}$ corresponds with {\it extended rational points} $\overline{\mathbb{Q}}=\mathbb{Q}\cup \{\frac{1}{0}\}\subset\overline{\mathbb{R}}=\mathbb{R}\cup \{\frac{1}{0}\}=\partial \mathbb{H}^2$, and the edges are represented by hyperbolic geodesics.
Then $\mathbb{H}^2$ is completely partitioned by the ideal triangles, called {\it Farey triangles}, whose sides are the edges of the Farey graph.
In this paper, we regard the Farey graph as the embedded graph in $\overline{\mathbb{H}}$.
We note that both hyperbolic geodesics and simplicial geodesics are considered.

For any $x,y\in\overline{\mathbb{Q}}\, (\subset \partial \mathbb{H}^2)$, the {\it ladder} associated with $x,y$, denoted by $\mathcal{L}(x,y)$, is the union of all Farey triangles whose interior intersects with the oriented hyperbolic geodesic joining $x$ and $y$. 
Then a ladder is a union of Farey triangles $\{\Delta_i\}$ such that $\Delta_i\cap \Delta_{i+1}$ is a single edge of $\mathcal{F}$ and $\Delta_i\cap \Delta_j$ is either an empty set or a single point, which is called a {\it pivot point}, if $|i-j|\ge 2$.
(See Figure~\ref{fig-ladder}.)
For a ladder $\mathcal{L}=\mathcal{L}(x,y)$ containing at least three Farey triangles, the {\it spine} $K$ of $\mathcal{L}$ is the simplicial path in $\mathcal{L}$ with the following properties:
\begin{itemize}
\item the endpoints of $K$ are $x$ and $y$,
\item all the vertices of $K$ except for the endpoints are exactly all the pivot points in $\mathcal{L}$,
\item all the edges of $K$ except for the initial and final one are edges in $\mathcal{L}$ whose interior intersects with the geodesic connecting $x$ and $y$.
\end{itemize} 
It is known that the spine is uniquely determined for a ladder (see \cite{BKKS}).

Let $\mathcal{L}(x,y)$ be the ladder associated with $x,y\in\overline{\mathbb{Q}}$, and let $\gamma$ be the hyperbolic geodesic from $x$ to $y$.
Note that each Farey triangle in $\mathcal{L}(x,y)$ contains a pivot point on the left or right of the oriented geodesic $\gamma$.
Label the triangles with $L$ or $R$ accordingly.
We say $\mathcal{L}(x,y)$ is of {\it type} $(a_1,a_2,\dots,a_n)$ if the ladder has $a_1,a_2,\dots,a_n$ consecutive Farey triangles with same labels read off in the orientation given to the geodesic.
\begin{figure}[tb]
 \begin{center}
 \includegraphics[width=90mm]{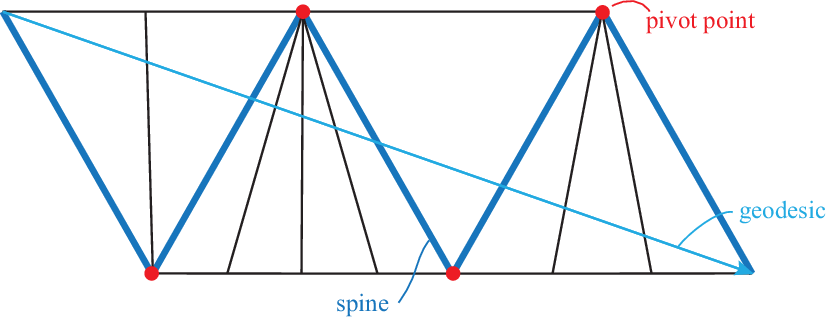}
 \end{center}
 \caption{Ladder of type $(2,4,1,3)$, pivot points and spine.}
\label{fig-ladder}
\end{figure}
By \cite[Proposition 5]{BKKS} and \cite[Proposition 2.2]{Ser}, the ladder $\mathcal{L}\left(\frac{1}{0},\frac{p}{q}\right)$ is of type $(a_1,a_2,\dots,a_n)$ for any rational number $\frac{p}{q}$ with $0<\frac{p}{q}<1$ if and only if 
\begin{eqnarray*}
\begin{array}{rrl}
\displaystyle\frac{p}{q}&=&[a_1, a_2, \dots, a_n]\\
&:=&\displaystyle\frac{1}{a_1+\displaystyle\frac{1}{a_2+\displaystyle\frac{1}{\cdots+\displaystyle\frac{1}{a_n}}}}.
\end{array}
\end{eqnarray*}
Also, for a path $\mathcal{P}$ in $\mathcal{F}$ connecting $\frac{1}{0}$ and $\frac{p}{q}$, it is known by \cite[Corollary 8]{BKKS} that $\mathcal{P}$ is a simplicial geodesic in $\mathcal{F}$ if and only if  $\mathcal{P}$ is a simplicial geodesic in the ladder  $\mathcal{L}\left(\frac{1}{0},\frac{p}{q}\right)$.
As a consequence of these facts, we immediately have:

\begin{theorem}\label{thm-b}
If $\frac{p}{q}=[a_1,a_2,\dots,a_{n-1}]$ with $a_i\ge 3$ for every $i\in\{1,2,\dots,n-1\}$, then the spine of the ladder $\mathcal{L}\left(\frac{1}{0},\frac{p}{q}\right)$ is the unique geodesic connecting $\frac{1}{0}$ and $\frac{p}{q}$ in $\mathcal{F}$ of length $n$.
In particular, the diameter of the curve complex of the $4$-punctured sphere (resp. a torus with at most one hole) is infinite.
\end{theorem}


\begin{thebibliography}{99}

%%%%%%%%%%%%%%%%%%%%%%%%%%%%%%%%%%%%%%%
%\bibitem{B1}
%T. Bridgeland,
%Equivalences of triangulated categories and Fourier--Mukai transforms,
%Bull. London Math. Soc., {\bf 31} (1999), 25--34.
%%%%%%%%%%%%%%%%%%%%%%%%%%%%%%%%%%%%%%%%
%\bibitem{Sa}
%I. Satake, 
%Algebraic Structures of Symmetric Domains,
%Publ. Math. Soc. Japan, {\bf 14}, 
%Iwanami Shoten, Tokyo; 
%Princeton Univ. Press, Princeton, NJ, 1980.
%%%%%%%%%%%%%%%%%%%%%%%%%%%%%%%%%%%%%%%%%
%\bibitem{Vo} D. A. Vogan,  Jr.,
%A Langlands classification for unitary representations,
%In: Analysis on Homogeneous Spaces and Representation Theory of Lie Groups,
%Okayama-Kyoto,  1997,
%(eds. T. Kobayashi, M. Kashiwara, T. Matsuki, K. Nishiyama and T. Oshima),
%Adv. Stud. Pure Math., {\bf 26}, Math. Soc. Japan, 2000, pp. 299--324.
%%%%%%%%%%%%%%%%%%%%%%%%%%%%%%%%%%%%%%%%%%

%\bibitem{AS}
%A. Abrams and S. Schleimer,
%{\it Distances of Heegaard splittings}, 
%Geom. Topology, 9 (2005), 95--119.

\bibitem{BKKS}
H. Baik, C. Kim, S. Kwak and H. Shin, 
{\it On translation lengths of Anosov maps on the curve graph of the torus}, 
Geom. Dedicata, 214 (2021), 399--426.

%%\bibitem{BM}
%%J. Birman and W. Menasco, 
%%{\it The curve complex has dead ends}, 
%%Geom. Dedicata, 177(2015), 71--74.

%\bibitem{BS}
%J. Berge and M. Scharlemann, 
%{\it Multiple genus 2 Heegaard splittings: a missed case}, 
%Algebraic and Geometric Topology, 11 (2011), 1781--1792.


\bibitem{CG}
A. J. Casson and C. McA. Gordon 
{\it Reducing Heegaard splittings}, 
Topology, 27 (1987), 275--283.


\bibitem{E}
Q. E, 
{\it On keen weakly reducible Heegaard splittings}, 
Topology Appl., 231 (2017), 128--135.

%\bibitem{E2}
%T. Evans, 
%{\it High distance Heegaard splittings of 3-manifolds}, 
%Topology Appl., 153 (2006), 2631--2647.

%%\bibitem{Har}
%%K.Hartshorn, 
%%{\it Heegaard splittings of Haken manifolds have bounded distance}, 
%%Pacific J. Math., 204 (2002), 61--75.

%\bibitem{Ha}
%W. J. Harvey, 
%{\it Boundary structure of the modular group. In Riemann
%surfaces and related topics}, Proceedings of the 1978 Stony Brook Conference
%(State Univ. New York, Stony Brook, N.Y., 1978), pages 245-251, Princeton,
%N.J., 1981. Princeton Univ. Press.


\bibitem{HS}
C. Hayashi and K. Shimokawa,
{\it Thin position of a pair $($$3$-manifold, $1$-submanifold$)$}, 
Pacific J. Math., 197 (2001), 301--324.

\bibitem{He}
J. Hempel, 
{\it 3-manifolds as viewed from the curve complex}, 
Topology, 40 (2001), 631--657.

%\bibitem{IM}
%K. Ichihara and K. Motegi,
%{\it Braids and Nielsen-Thurston types of automorphisms of punctured surfaces}, 
%Tokyo J. Math., 28 (2005), 527--538.

\bibitem{IS}
K. Ichihara and T. Saito,
{\it Knots with arbitrary high distance bridge decompositions}, 
Bull. Korean Math. Soc., 50 (2013), 1989--2000.

\bibitem{IJK1}
A. Ido, Y. Jang and T. Kobayashi,
{\it Heegaard splittings of distance exactly $n$}, 
Algebr. Geom. Topol., 14 (2014), 1395--1411.

\bibitem{IJK2}
A. Ido, Y. Jang and T. Kobayashi,
{\it Bridge splittings of links with distance exactly $n$}, 
Topology Appl., 196 (2015), 608--617.

\bibitem{IJK3}
A. Ido, Y. Jang and T. Kobayashi,
{\it On keen Heegaard splittings}, 
Adv. Stud. Pure Math., 78 (2018), 293--311.

\bibitem{IK}
D. Iguchi and Y. Koda,
{\it Twisted book decompositions and the Goeritz groups}, 
Topology Appl., 272 (2020), 107064.

%\bibitem{Jo}
%J. Johnson,
%{\it Non-uniqueness of high distance Heegaard splittings}, 
%arXiv:1308.4599.

%%\bibitem{K2}
%%T. Kobayashi,
%%{\it Heights of simple loops and pseudo-anosov homeomorphisms}, 
%%Contemporary Math., 78 (1988), 327--338.

\bibitem{Kra}
I. Kra,
{\it On the Nielsen-Thurston-Bers type of some self-maps of Riemann surfaces}, 
Acta Math., 147 (1981), 231--270.

\bibitem{Li}
T. Li,
{\it Images of the disk complex}, 
Geom. Dedicata., 158 (2012), 121--136.

%\bibitem{MQZ}
%J. Ma, Ruifeng Qiu and Yanqing Zou,
%{\it Heegaard distances cover all non-negative integers}, 
%preprint.

%\bibitem{MM1}
%H. Masur and Y. Minsky,
%{\it Geometry of the complex of curves. I. Hyperbolicity}, 
%Invent. Math., 138 (1999), 103--149.

\bibitem{MM2}
H. Masur and Y. Minsky,
{\it Geometry of the complex of curves. II. Hierarchical structure}, 
Geom. Funct. Anal., 10 (2000), 902--974.

%%\bibitem{MM3}
%%H. Masur and Y. Minsky,
%%{\it Quasiconvexity in the curve complex}, 
%%In the Tradition of Ahlfors
%%and Bers, III (W. Abikoff and A. Haas, eds.), Contemporary Mathematics 355, Amer. Math.
%%Soc. (2004), 309-320.

%\bibitem{QZG}
%Ruifeng Qiu, Yanqing Zou and Qilong Guo,
%{\it The Heegaard distances cover all non-negative integers}, 
%Pacific J. Math., 275 (2015), 231--255.

\bibitem{Sai}
T. Saito,
{\it Genus one $1$-bridge knots as viewed from the curve complex}, 
Osaka J. Math., 41 (2004), 427--454.

%%\bibitem{S}
%%S. Schleimer, 
%%{\it Notes on the complex of curves}, http://%%homepages.warwick.ac.uk/$\sim$masgar/

\bibitem{Sch}
S. Schleimer,
{\it Waldhausen's theorem}, 
Workshop on Heegaard Splittings, 299--317, Geom. Topol. Monogr., 12, Geom. Topol. Publ., Coventry, 2007.

\bibitem{Ser}
C. Series,
{\it Continued fractions and hyperbolic geometry}, 
http://homepages.warwick.ac.uk/$\sim$masbb/HypGeomandCntdFractions-2.pdf.


\bibitem{Wa}
F. Waldhausen,
{\it Heegaard-Zerlegungen der 3-Sph\"{a}re}, 
Topology, 7 (1968), 195--203.

\end{thebibliography}
\end{document}